\documentclass[11pt]{article}
\usepackage{fullpage}
\usepackage[psamsfonts]{eucal}
\usepackage{amsmath}
\usepackage{amsfonts}
\usepackage{amssymb}
\usepackage{amstext}
\usepackage[symbol]{footmisc}
\usepackage[noadjust]{cite}
\usepackage{amscd}
\usepackage{amsthm}
\usepackage{makeidx}
\usepackage{tikz}
\usepackage{graphicx}
\usepackage[colorlinks=true,linkcolor=blue]{hyperref}

\usepackage{supertabular}
\usepackage{enumitem}
\usepackage{titling}
\usepackage{wasysym}
\usepackage{comment}
\usepackage{cleveref}
\usepackage{listings}
\usepackage{relsize}

\newcommand{\RR}{\mathbb{R}}
\newcommand{\EE}{\mathbb{E}}

\newcommand{\PP}{\mathbb{P}}

\newcommand{\CC}{\mathbb{C}}
\newcommand{\girth}{\mathrm{girth}}

\renewcommand{\P}{\mathbb{P}}

\usepackage{microtype}
\usepackage{abbreviations}

\theoremstyle{plain}
\newtheorem{thm}{Theorem}[section]
\newcommand{\cT}{\mathcal T}
\newtheorem{prop}[thm]{Proposition}
\newtheorem{lemma}[thm]{Lemma}

\newtheorem{question}[thm]{Question}

\newtheorem{dfn}[thm]{Definition}

\newtheorem{rmk}[thm]{Remark}
\newtheorem{cor}[thm]{Corollary}

\def\title#1{
    \thispagestyle{plain}
    \vspace*{-14pt}
    \vskip 39pt
    {\centering{\titlefont #1\par}}%
    \vspace*{4pt}
    \vskip 1em
}
\def\author#1{\par
    {\centering{\authorfont#1}\par\vspace*{0.05in}}
}

\def\titlefont{\fontsize{14}{21}\boldmath\selectfont\centering{}}
\def\authorfont{\fontsize{13}{15}}

\renewcommand{\E}{\mathbb{E}}

\DeclareMathOperator{\dist}{dist}
\DeclareMathOperator{\Ext}{Ext}

\newcommand{\vone}{\mathbf{1}}
\definecolor{forestgreen}{RGB}{34, 139, 34}

\newcommand{\bT}{\mathbb{T}}

\newcommand{\bfS}{\mathbf{S}}

\newcommand{\bW}{\mathbb{W}}
\newcommand{\deq}{\mathrel{\mathop:}=} 

\newcommand{\tcG}{{\widetilde \cG}}

\newcommand{\fR}{\mathfrak{R}}
\newcommand{\cH}{\mathcal{H}}

\newcommand{\cB}{\mathcal{B}}
\newcommand{\wt}{\widetilde}
\newcommand{\bP}{\mathbb{P}}
\newcommand{\vareps}{\varepsilon}
\newcommand{\cG}{\mathcal{G}}
\newcommand{\bV}{\mathbb{V}}

\newcommand{\fq}{\mathfrak{q}}

\newcommand{\oOmega}{\overline{\Omega}}
\newcommand{\be}{\begin{equation}}
\newcommand{\ee}{\end{equation}}
\newcommand{\qq}[1]{[{#1}]}

\numberwithin{equation}{section}
\newcommand{\ff}{\mathfrak{f}}
\newcommand{\cF}{\mathcal{F}}
\newcommand{\rd}{{\rm d}}
\newcommand{\bI}{\mathbb{I}}
\newcommand{\sG}{{\sf G}}
\newcommand{\sS}{{\sf S}}
\newcommand{\bd}{\mathbf{ d}}
\newcommand{\bD}{\mathbf{ D}}
\newcommand{\GNd}{\sG_{N,d}}
\newcommand{\GNdp}{\tilde \sG_{N,d}}
\renewcommand{\Pr}{\mathbb{P}}

\begin{document}

\title{Arbitrary Spectral Edge of Regular Graphs}
\author{Dingding Dong \footnote[1]{Harvard University. \texttt{ddong@math.harvard.edu}} \qquad\quad  Theo McKenzie \footnote[2]{Stanford University. \texttt{theom@stanford.edu}}}
\quad

\begin{abstract}
    We prove that for each $d\geq 3$ and $k\geq 2$, the set of limit points of the first $k$ eigenvalues of sequences of $d$-regular graphs is
    \[
    \{(\mu_1,\dots,\mu_k):  d=\mu_1\geq \dots\geq \mu_{k}\geq2\sqrt{d-1}\}.
    \]
    The result for $k=2$ was obtained by Alon and Wei, and our result confirms a conjecture of theirs. Our proof uses an infinite random graph sampled from a distribution that generalizes the random regular graph distribution. To control the spectral behavior of this infinite object, we show that Huang and Yau's proof of Friedman's theorem bounding the second eigenvalue of a random regular graph generalizes to this model. We also bound the trace of the non-backtracking operator, as was done in Bordenave's separate proof of Friedman's theorem. 
\end{abstract}

\section{Introduction}

For a graph $\GG$ on $N$ vertices, we denote by $\lambda_1(\GG)\geq\dots\geq \lambda_N(\GG)$ the eigenvalues of the adjacency operator on $\GG$. Given a family of graphs $\GG_1,\GG_2,\ldots$, all of bounded degree $d$ for fixed $d\geq 3$, a natural question asks what are the possible limit points of the $k$-dimensional vectors $(\lambda_1(\GG_i),\ldots,\lambda_k(\GG_i))$. This question is a variant of the inverse spectral problem---as opposed to the classical problem asking whether we can recover a graph up to isomorphism from the first few eigenvalues, this question asks which spectral edges are in fact possible.

In terms of what is possible, a consequence of the Perron--Frobenius theorem is that $\lambda_1(\GG_i)\leq d$ for all $i$. If we further assume that the graphs are $d$-regular, the Alon-Boppana bound \cite{alon1986eigenvalues,Nil91} gives that  $\lambda_2(\GG)\geq 2\sqrt{d-1}-o_N(1)$. More generally (stated for example in  \cite{Cio06,FL96,Nil04})
for every fixed $k\geq 1$, as $i\to\infty$, we have $\lambda_k(\GG_i)\geq2\sqrt{d-1}-o_N(1)$. Therefore, for regular graphs, the question reduces to asking the following. 

\begin{question}\label{q:mainquestion}
Fix $d\geq 3$ and $k\geq 2$. For which vectors in
\be\label{eq:limitpoint}
\{(\mu_1,\dots,\mu_k):  d=\mu_1\geq \dots\geq \mu_{k}\geq2\sqrt{d-1}\}\ee
does there exist an infinite family of $d$-regular graphs $\{\GG_{i}\}_{i\in \N}$ with $\lim_{i\to\infty}\lambda_j(\GG_i)=\mu_j$ for every $j\in[k]$?
\end{question}
Alon and Wei \cite{alon2023limit}  investigated \Cref{q:mainquestion} and showed that when $k=2$, the answer was \emph{all} pairs of numbers $d=\mu_1\geq \mu_2\geq 2\sqrt{d-1}$. They also showed  for all $k\geq 2$ that all vectors in the slightly smaller set \[\{(\mu_1,\dots,\mu_k):d=\mu_1\geq \dots\geq \mu_{k}\geq2\sqrt{d-1}+ (d-1)^{-1/2}\}\] exist as limit points. Note this set converges to \eqref{eq:limitpoint} as $d\to\infty$. They also proved that the entire set in \eqref{eq:limitpoint} exists as limit points if we relax the requirement of being $d$-regular to being of degree \emph{at most} $d$ (note that in this scenario there could be other limit points, as the $d$-regular Alon-Boppana bound no longer holds).

Alon and Wei conjectured that  every point in \eqref{eq:limitpoint} can be achieved as a limit point for some sequence of $d$-regular graphs \cite[Conjecture 1.3]{alon2023limit}.  The main result of this paper is that this conjecture is true.

\begin{thm}\label{thm:main}
    Fix $d\geq 3$ and $k\geq 2$. For all  $d=\mu_1\geq \dots\geq \mu_k\geq 2\sqrt{d-1}$, there exists a sequence of $d$-regular finite simple graphs $\GG_1,\GG_2,\dots$ such that $\lim_{i\to\infty}\lambda_j(\GG_i)=\mu_j$ for all $1\leq j\leq k$.
\end{thm}

The $d$-regular construction of Alon and Wei starts by taking a series of small, high girth graphs $\mathcal{F}_1,\dots, \mathcal{F}_{k-1}$ of degree at most $d$. Then, vertices of degree less than $d$ are extended with large trees. To be specific,  for $L\in \N\cup \infty_{+}$, we define the \emph{depth-$L$ tree extension} $T^L\mathcal F$ to be the graph obtained by attaching $d-\deg_{\cF}(v)$ rooted  $(d-1)$-ary trees of height $L$ to each $v\in V(\cF)$. Note that in $T^L\mathcal F$, all vertices are of degree $d$ except for the leaves of the trees. 
These tree extensions $T^L\mathcal{F}_1,\dots,T^L\mathcal{F}_{k-1}$ are then implanted into a much larger high girth $d$-regular graph $\mathcal{F}_0$. 
The authors were able to show that, under appropriate conditions, the top $k$ eigenvalues of the obtained graph are close to those of a disjoint union of $\mathcal{F}_0,T^L\mathcal{F}_1,\dots,T^L\mathcal{F}_{k-1}$  \cite[Lemma 4.4]{alon2023limit} .
Using this argument, one can deduce \Cref{thm:main} from the following.

\begin{thm}\label{thm:main-2}
    Fix $2\sqrt{d-1}< \mu< d$ and $\epsilon>0$. Fix any $R>0$. Then there exists a simple (perhaps infinite) $d$-regular graph $\GG$ such that
    \begin{enumerate}
        \item $\GG=T^\infty \GG_0$ for some finite graph $\GG_0$,
        \item $\GG$ has girth $\geq R$,
        \item $|\lambda_1(\GG)-\mu|<\epsilon$ and $\lambda_2(\GG)<2\sqrt{d-1}+\epsilon$.
    \end{enumerate}
\end{thm}

Alon and Wei obtained the same statement as \Cref{thm:main-2} with the weaker bound of $\lambda_2(\GG)\leq 2\sqrt{d-1}+(d-1)^{-1/2}+\epsilon$. 
They used the following construction: take a random $N$-lift of $K_{d+1}$, denoted $\cF_*$, and iteratively delete vertices from $\cF_*$ until the resulting graph $\cF'$ satisfies $|\lambda_1(T^\infty \cF')-\mu|\leq \epsilon$. 
The authors showed that if vertices are deleted in a specific order, then for every subgraph $\cF'$ of $\cF_*$ appearing during the iteration, $\lambda_2(T^\infty \cF')\leq 2\sqrt{d-1}+(d-1)^{-1/2}+\epsilon$. It is unclear whether a stronger inequality is possible for this particular construction.

Our work introduces a slightly different construction that satisfies the stronger bound $\lambda_2(\GG)\leq 2\sqrt{d-1}+\epsilon$. 
Given $2\sqrt{d-1}< \mu< d$, we will pick some $p=p(\mu)\in[0,1]$ (specified in \Cref{def:construction}) and then do the following:
\begin{enumerate}
    \item Let $\cH$ be a random $d$-regular graph (possibly with loops and multiedges) on $N$ vertices.
    \item Let $\GG_0$ be a random subgraph of $\cH$ obtained by including every edge of $\cH$ independently with probability $p$.
    \item Take $\GG=T^\infty \GG_0$.
\end{enumerate}

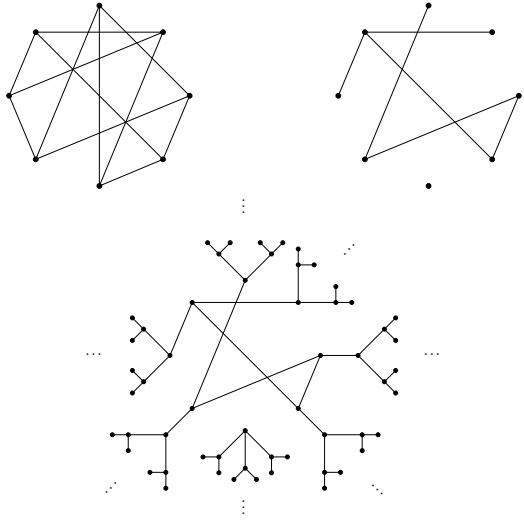
\begin{figure}
    \centering
    \scalebox{0.6}{
\begin{tikzpicture}
    \node[circle, fill=black, draw, inner sep=0pt,minimum size=3pt] at (0,2) {};
    \node[circle, fill=black, draw, inner sep=0pt,minimum size=3pt] at (2,0) {};
    \node[circle, fill=black, draw, inner sep=0pt,minimum size=3pt] at (0,-2) {};
    \node[circle, fill=black, draw, inner sep=0pt,minimum size=3pt] at (-2,0) {};
    \node[circle, fill=black, draw, inner sep=0pt,minimum size=3pt] at (1.41,1.41) {};
    \node[circle, fill=black, draw, inner sep=0pt,minimum size=3pt] at (1.41,-1.41) {};
    \node[circle, fill=black, draw, inner sep=0pt,minimum size=3pt] at (-1.41,1.41) {};
    \node[circle, fill=black, draw, inner sep=0pt,minimum size=3pt] at (-1.41,-1.41) {};
    \draw(0,2)--(0,-2);
    \draw(0,2)--(2,0);
    \draw(0,2)--(-1.41,-1.41);
    \draw(-1.41,1.41)--(1.41,1.41);
    \draw(-2,0)--(1.41,1.41);
    \draw(0,-2)--(1.41,1.41);
    \draw(2,0)--(1.41,-1.41);
    \draw(2,0)--(-1.41,-1.41);
    \draw(0,-2)--(1.41,-1.41);
    \draw(-1.41,1.41)--(1.41,-1.41);
    \draw(-2,0)--(-1.41,-1.41);
    \draw(-2,0)--(-1.41,1.41);
\end{tikzpicture}
\qquad \qquad \qquad \qquad 
\begin{tikzpicture}
    \node[circle, fill=black, draw, inner sep=0pt,minimum size=3pt] at (0,2) {};
    \node[circle, fill=black, draw, inner sep=0pt,minimum size=3pt] at (2,0) {};
    \node[circle, fill=black, draw, inner sep=0pt,minimum size=3pt] at (0,-2) {};
    \node[circle, fill=black, draw, inner sep=0pt,minimum size=3pt] at (-2,0) {};
    \node[circle, fill=black, draw, inner sep=0pt,minimum size=3pt] at (1.41,1.41) {};
    \node[circle, fill=black, draw, inner sep=0pt,minimum size=3pt] at (1.41,-1.41) {};
    \node[circle, fill=black, draw, inner sep=0pt,minimum size=3pt] at (-1.41,1.41) {};
    \node[circle, fill=black, draw, inner sep=0pt,minimum size=3pt] at (-1.41,-1.41) {};
    \draw(0,2)--(-1.41,-1.41);
    \draw(-1.41,1.41)--(1.41,1.41);
    \draw(2,0)--(1.41,-1.41);
    \draw(2,0)--(-1.41,-1.41);
    \draw(-1.41,1.41)--(1.41,-1.41);
    \draw(-2,0)--(-1.41,1.41);
\end{tikzpicture}
}

\scalebox{0.5}{
\begin{tikzpicture}
    \node[circle, fill=black, draw, inner sep=0pt,minimum size=3pt] at (0,2) {};
    \node[circle, fill=black, draw, inner sep=0pt,minimum size=3pt] at (2,0) {};
    \node[circle, fill=black, draw, inner sep=0pt,minimum size=3pt] at (0,-2) {};
    \node[circle, fill=black, draw, inner sep=0pt,minimum size=3pt] at (-2,0) {};
    \node[circle, fill=black, draw, inner sep=0pt,minimum size=3pt] at (1.41,1.41) {};
    \node[circle, fill=black, draw, inner sep=0pt,minimum size=3pt] at (1.41,-1.41) {};
    \node[circle, fill=black, draw, inner sep=0pt,minimum size=3pt] at (-1.41,1.41) {};
    \node[circle, fill=black, draw, inner sep=0pt,minimum size=3pt] at (-1.41,-1.41) {};
    \node[circle, fill=black, draw, inner sep=0pt,minimum size=3pt] at (3,0) {};
    \draw(2,0)--(3,0);
    \node[circle, fill=black, draw, inner sep=0pt,minimum size=3pt] at (3.7,0.7) {};
    \draw(3.7,0.7)--(3,0);
    \node[circle, fill=black, draw, inner sep=0pt,minimum size=3pt] at (4,1) {};
    \draw(3.7,0.7)--(4,1);
    \node[circle, fill=black, draw, inner sep=0pt,minimum size=3pt] at (4,0.4) {};
    \draw(3.7,0.7)--(4,0.4);
    \node[circle, fill=black, draw, inner sep=0pt,minimum size=3pt] at (3.7,-0.7) {};
    \draw(3.7,-0.7)--(3,0);
    \node[circle, fill=black, draw, inner sep=0pt,minimum size=3pt] at (4,-0.4) {};
    \draw(3.7,-0.7)--(4,-0.4);
    \node[circle, fill=black, draw, inner sep=0pt,minimum size=3pt] at (4,-1) {};
    \draw(3.7,-0.7)--(4,-1);

    \node[circle, fill=black, draw, inner sep=0pt,minimum size=3pt] at (2.41,1.41) {};
    \node[circle, fill=black, draw, inner sep=0pt,minimum size=3pt] at (2.41,1.41+0.3*1.41) {};
    \draw(2.41,1.41+0.3*1.41)--(2.41,1.41);
    \node[circle, fill=black, draw, inner sep=0pt,minimum size=3pt] at (2.41+0.3*1.41,1.41) {};
    \draw(2.41+0.3*1.41,1.41)--(2.41,1.41);
    \node[circle, fill=black, draw, inner sep=0pt,minimum size=3pt] at (1.41,2.41) {};
    \draw(2.41,1.41)--(1.41,1.41);
    \node[circle, fill=black, draw, inner sep=0pt,minimum size=3pt] at (1.41,2.41+0.3*1.41) {};
    \draw(1.41,2.41+0.3*1.41)--(1.41,2.41);
    \node[circle, fill=black, draw, inner sep=0pt,minimum size=3pt] at (1.41+0.3*1.41,2.41) {};
    \draw(1.41+0.3*1.41,2.41)--(1.41,2.41);
    \draw(1.41,2.41)--(1.41,1.41);
    \node[circle, fill=black, draw, inner sep=0pt,minimum size=3pt] at (0.7,2.7) {};
    \draw(0,2)--(0.7,2.7);
    \draw(0,2)--(-0.7,2.7);
    \node[circle, fill=black, draw, inner sep=0pt,minimum size=3pt] at (1,3) {};
    \draw(1,3)--(0.7,2.7);
    \node[circle, fill=black, draw, inner sep=0pt,minimum size=3pt] at (0.4,3) {};
    \draw(0.4,3)--(0.7,2.7);
    \node[circle, fill=black, draw, inner sep=0pt,minimum size=3pt] at (-0.7,2.7) {};
    \draw(-0.7,2.7)--(-1,3);
    \draw(-0.7,2.7)--(-0.4,3);
    \node[circle, fill=black, draw, inner sep=0pt,minimum size=3pt] at (-0.4,3) {};
    \node[circle, fill=black, draw, inner sep=0pt,minimum size=3pt] at (-1,3) {};

    \node[circle, fill=black, draw, inner sep=0pt,minimum size=3pt] at (0,-3) {};
    \draw(0,-2)--(0,-3);
    \draw(0,-2)--(0.7,-2.7);
    \draw(0,-2)--(-0.7,-2.7);
    \node[circle, fill=black, draw, inner sep=0pt,minimum size=3pt] at (0.3,-3.3) {};
    \node[circle, fill=black, draw, inner sep=0pt,minimum size=3pt] at (-0.3,-3.3) {};
    \draw(-0.3,-3.3)--(0,-3);
    \draw(0.3,-3.3)--(0,-3);
    \node[circle, fill=black, draw, inner sep=0pt,minimum size=3pt] at (0.7,-2.7) {};
    \draw(0.7+0.3*1.41,-2.7)--(0.7,-2.7);
    \draw(0.7,-2.7-0.3*1.41)--(0.7,-2.7);
    \node[circle, fill=black, draw, inner sep=0pt,minimum size=3pt] at (0.7+0.3*1.41,-2.7) {};
    \node[circle, fill=black, draw, inner sep=0pt,minimum size=3pt] at (0.7,-2.7-0.3*1.41) {};
    \node[circle, fill=black, draw, inner sep=0pt,minimum size=3pt] at (-0.7,-2.7) {};
    \draw(-0.7,-2.7-0.3*1.41)--(-0.7,-2.7);
    \draw(-0.7-0.3*1.41,-2.7)--(-0.7,-2.7);
    \node[circle, fill=black, draw, inner sep=0pt,minimum size=3pt] at (-0.7,-2.7-0.3*1.41) {};
    \node[circle, fill=black, draw, inner sep=0pt,minimum size=3pt] at (-0.7-0.3*1.41,-2.7) {};

    \draw(-1.41-0.7,-1.41-0.7)--(-1.41,-1.41);
    \draw(1.41+0.7,-1.41-0.7)--(1.41,-1.41);
    \node[circle, fill=black, draw, inner sep=0pt,minimum size=3pt] at (-1.41-0.7,-1.41-0.7) {};
    \node[circle, fill=black, draw, inner sep=0pt,minimum size=3pt] at (1.41+0.7,-1.41-0.7) {};
    \node[circle, fill=black, draw, inner sep=0pt,minimum size=3pt] at (-2.41-0.7,-1.41-0.7) {};
    \draw(-1.41-0.7,-1.41-0.7)--(-2.41-0.7,-1.41-0.7);
    \node[circle, fill=black, draw, inner sep=0pt,minimum size=3pt] at (-2.41-0.7,-1.41-0.3*1.41-0.7) {};
    \draw(-2.41-0.7,-1.41-0.3*1.41-0.7)--(-2.41-0.7,-1.41-0.7);
    \node[circle, fill=black, draw, inner sep=0pt,minimum size=3pt] at (-2.41-0.3*1.41-0.7,-1.41-0.7) {};
    \draw(-2.41-0.3*1.41-0.7,-1.41-0.7)--(-2.41-0.7,-1.41-0.7);
    \node[circle, fill=black, draw, inner sep=0pt,minimum size=3pt] at (-1.41-0.7,-2.41-0.7) {};
    \node[circle, fill=black, draw, inner sep=0pt,minimum size=3pt] at (-1.41-0.7,-2.41-0.3*1.41-0.7) {};
    \draw(-1.41-0.7,-2.41-0.3*1.41-0.7)--(-1.41-0.7,-2.41-0.7);
    \node[circle, fill=black, draw, inner sep=0pt,minimum size=3pt] at (-1.41-0.7-0.3*1.41,-2.41-0.7) {};
    \draw(-1.41-0.3*1.41-0.7,-2.41-0.7)--(-1.41-0.7,-2.41-0.7);
    \draw(-1.41-0.7,-2.41-0.7)--(-1.41-0.7,-1.41-0.7);

        \node[circle, fill=black, draw, inner sep=0pt,minimum size=3pt] at (2.41+0.7,-1.41-0.7) {};
    \draw(1.41+0.7,-1.41-0.7)--(2.41+0.7,-1.41-0.7);
    \node[circle, fill=black, draw, inner sep=0pt,minimum size=3pt] at (2.41+0.7,-1.41-0.3*1.41-0.7) {};
    \draw(2.41+0.7,-1.41-0.3*1.41-0.7)--(2.41+0.7,-1.41-0.7);
    \node[circle, fill=black, draw, inner sep=0pt,minimum size=3pt] at (2.41+0.3*1.41+0.7,-1.41-0.7) {};
    \draw(2.41+0.3*1.41+0.7,-1.41-0.7)--(2.41+0.7,-1.41-0.7);
    \node[circle, fill=black, draw, inner sep=0pt,minimum size=3pt] at (1.41+0.7,-2.41-0.7) {};
    \node[circle, fill=black, draw, inner sep=0pt,minimum size=3pt] at (1.41+0.7,-2.41-0.3*1.41-0.7) {};
    \draw(1.41+0.7,-2.41-0.3*1.41-0.7)--(1.41+0.7,-2.41-0.7);
    \node[circle, fill=black, draw, inner sep=0pt,minimum size=3pt] at (1.41+0.3*1.41+0.7,-2.41-0.7) {};
    \draw(1.41+0.3*1.41+0.7,-2.41-0.7)--(1.41+0.7,-2.41-0.7);
    \draw(1.41+0.7,-2.41-0.7)--(1.41+0.7,-1.41-0.7);

    \node[circle, fill=black, draw, inner sep=0pt,minimum size=3pt] at (-2.7,-0.7) {};
    \node[circle, fill=black, draw, inner sep=0pt,minimum size=3pt] at (-2.7,0.7) {};
    \draw(-2,0)--(-2.7,-0.7);
    \draw(-2,0)--(-2.7,0.7);
    \node[circle, fill=black, draw, inner sep=0pt,minimum size=3pt] at (-3,1) {};
    \node[circle, fill=black, draw, inner sep=0pt,minimum size=3pt] at (-3,0.4) {};
    \draw(-2.7,0.7)--(-3,1);
    \draw(-2.7,0.7)--(-3,0.4);
    \draw(-2.7,-0.7)--(-3,-1);
    \draw(-2.7,-0.7)--(-3,-0.4);
    \node[circle, fill=black, draw, inner sep=0pt,minimum size=3pt] at (-3,-1) {};
    \node[circle, fill=black, draw, inner sep=0pt,minimum size=3pt] at (-3,-0.4) {};

    \draw(0,2)--(-1.41,-1.41);
    \draw(-1.41,1.41)--(1.41,1.41);
    \draw(2,0)--(1.41,-1.41);
    \draw(2,0)--(-1.41,-1.41);
    \draw(-1.41,1.41)--(1.41,-1.41);
    \draw(-2,0)--(-1.41,1.41);

    \node[] at (5,0) {$\cdots$};
    \node[] at (-4,0) {$\cdots$};
    \node[rotate=90] at (0,4) {$\cdots$};
    \node[rotate=90] at (0,-4) {$\cdots$};
    \node[rotate=45] at (2*1.41,2*1.41) {$\cdots$};
    \node[rotate=135] at (2*1.41+0.7,-2*1.41-0.7) {$\cdots$};
    \node[rotate=45] at (-2*1.41-0.7,-2*1.41-0.7) {$\cdots$};
\end{tikzpicture}
}
    \caption{A random construction for \Cref{thm:main-2} when $d=3$}
    \label{fig:1}
\end{figure}

We will show that for some suitably chosen  $p\in[0,1]$, with positive probability, the resulting $T^\infty \GG_0$ is a simple graph that satisfies \Cref{thm:main-2}. The two graphs to keep in mind as examples for \Cref{thm:main-2} are the random $d$-regular graph (for $\mu=d$) and the infinite $d$-regular tree (for $\mu=2\sqrt{d-1}$). We think of our construction $T^\infty \GG_0$ as some interpolation between the two. Indeed, when $p=1$, $T^\infty \GG_0$ will be the random $d$-regular graph  and when $p=0$, $T^\infty \GG_0$ will be a disjoint union of infinite $d$-regular trees. As we will show, for every $2\sqrt{d-1}< \mu< d$, there exists some $p=p(\mu)$ such that the corresponding graph $\cG$ has spectral radius close to $\mu$ with high probability.

\subsection{Proof techiniques}

Our proof of \Cref{thm:main-2} uses several ideas in spectral theory.  To show concentration of $\lambda_1(T^\infty \GG_0)$, we  use the Ihara--Bass formula for regular graphs \cite{angel2015non,Bas92} that gives a correspondence between eigenvalues of $T^\infty \GG_0$ and eigenvalues of $B$, the nonbacktracking walk operator of $T^\infty \GG_0$ (\Cref{def:nbw}). We then   give probabilistic  bounds on the number of cycles in $\GG_0$ that show concentration of $\lambda_1(B)$, which leads to the desired result for $\lambda_1(T^\infty \GG_0)$. These types of ideas were used in Bordenave's proof of Friedman's theorem, which states that for any $d\geq 3$ a randomly sampled $d$-regular graph $\cH$ on $N$ vertices has, with high probability, $\max\{\lambda_2(\cH),|\lambda_N(\cH)|\}=2\sqrt{d-1}+o_N(1)$ \cite{bordenave2015new, friedman2003proof}.

A strong statement is true about the eigenvector of the nonbacktracking operator that corresponds to any eigenvalue greater than $\sqrt{d-1}$, in that it is not supported on edges in the tree extension directed towards $\GG_0$. This means  that to estimate large eigenvalues of the nonbacktracking operator of $T^\infty\GG_0$, it suffices to look at the operator on $\GG_0$. This eigenvector can be approximated by using the tree neighborhood of most vertices, and realizing that the neighborhood of most vertices approximates a random tree branching process \cite{bordenave2015quantum}. We will reason about the combinatorial statistics of these neighborhoods through the Kesten-Stigum theorem, which states that the distribution of the sizes of neighborhoods in this process converge to a deterministic limit \cite{kesten1966limit}.

To upper bound $\lambda_2(T^\infty \GG_0)$, we will prove a mild generalization of Friedman's theorem to the random infinite graph model described above. A challenge in generalizing Friedman's theorem to our model is that previous proofs \cite{friedman2003proof,bordenave2015new,huang2024spectrum, chen2024new} make use of two natural facts concerning finite regular graphs: (1) they are finite, which allows us to bound the second eigenvalue by bounding a normalized trace; (2) we have explicit knowledge about the top eigenvalue and eigenvector, which allows us to project away the dependence on this eigenvalue when applying a trace method or performing Green's function analysis. Neither of these two points is true in our model.

We found that the proof that was easiest to generalize to our model was that of Huang and Yau, which gives a local law for the random regular graph distribution tight enough to imply Friedman's theorem \cite{huang2024spectrum}. The local law states that the \emph{Green's function} $G(z)\deq (\frac{A}{\sqrt{d-1}}-z\bI)^{-1}$ approximates that of the infinite regular tree. This proof generalizes nicely to our setting, as much of it is dedicated to showing that the contribution of certain subgraphs to the Green's function is approximately that given by replacing the subgraph with an infinite tree. Our model is identical, except some subgraphs are \emph{actually} infinite trees, meaning that many parts of the proof immediately generalize to our model (namely, the true random regular graph is the most difficult version to analyze). 

Another advantage here is that the Green's function allows us to analyze the infinite tree extension $T^{\infty}\GG_0$ through the original graph $\GG_0$. The Green's function of $T^\infty \GG_0$ on $\GG_0$ is equal to that of a finite graph where each connection to an infinite tree is replaced by a weighted loop. This ``finitization'' of the Green's function is possible through the Schur complement formula, and has been crucial in proving sharp spectral statistics in random regular graphs \cite{bauerschmidt2019local,huang2024optimal,huang2024spectrum}. Moreover, we will show that any eigenvalue $\lambda>2\sqrt{d-1}$ has an \emph{Anderson localized} eigenvector, in that the eigenvector is exponentially decaying away from $\GG_0$. This allows us to bound $\lambda_2(T^{\infty}\GG_0)$ by bounding the normalized projected trace, where the projection is on the finite vertex set of $\GG_0$.

Nevertheless, Huang and Yau still use explicit knowledge of the top eigenvalue and eigenvector in \cite{huang2024optimal}. Having calculated the Green's function, they show that after subtracting the contribution of the $d$ eigenvalue we achieve a sufficiently tight local law. As noted above, we cannot do this, as we do not know our top eigenvector; even if we could show that it is almost deterministic, the contribution of this top eigenspace in our model becomes quite complicated.
To avoid this, we perform our second eigenvalue bound in two steps. The first step is to show that in fact, our trace bound is strong enough to show that $\lambda_1(T^{\infty}\GG_0)-\lambda_2(T^{\infty}\GG_0)>\epsilon$, for some constant $\epsilon>0$, using an argument similar to Broder and Shamir's proof that a random regular graph has second eigenvalue $O((d-1)^{3/4})$ \cite{broder1987second}. 
This allows us to ignore the structure of the top eigenvector, and in the second step we can reduce to analyzing our Green's function for $z$ away from $\lambda_1(T^{\infty}\GG_0)$. This simplifies our analysis, as we can show that in this region Green's function of the graph sharply approximates that of the infinite tree, rather than having to consider the error term coming from the top eigenvalue explicitly.

\subsection{Related work}

The infinite graph we use has the form of a finite graph extended with regular trees; the properties of such a graph, such as the fact that the absolutely continuous spectrum is that of the infinite tree, were studied in \cite{colin2013scattering}. Our method of preserving spectrum when passing from disjoint infinite graphs to a large graph with fixed $(\lambda_1,\ldots,\lambda_k)$ follows a number of works asking whether graphs with optimal spectral expansion can preserve this property after undergoing some small perturbation.
Kahale used this general idea of taking a small gadget graph and implanting it within a larger, optimally expanding graph to show a tight relationship between spectral expansion and vertex expansion \cite{kahale1995eigenvalues}. Implanted gadget graphs into large Ramanujan graphs have been used to relate spectral gap with combinatorial properties \cite{mckenzie2021high}, as well as to relate spectral gap with eigenvector localization \cite{alon2021high,ganguly2021non}.

In order to give tight bounds on $\lambda_1(T^\infty \GG)$, we must consider the percolated graph. The spectrum of percolated graphs has been well studied \cite{bordenave2017mean, bordenave2015quantum}, as has the spectrum for branching processes \cite{arras2023existence}.
Moreover, we see this as closely related to the study of the nonbacktracking operator for Erd\H{o}s-R\'enyi (where the local limit is a branching process), where we see similar structure of the eigenvector of the nonbacktracking operator \cite[Theorem 3]{bordenave2015non}.

Other versions of the inverse spectral problem have also been of interest. Koll\'ar and Sarnak studied possible gaps in the spectrum for cubic graphs \cite{kollar2021gap}. Magee showed a version of the theorem of Alon and Wei for hyperbolic surfaces by showing that the set of limit points of the smallest nontrivial eigenvalue of an arithmetic hyperbolic surface is the interval $[0,1/4]$ \cite{magee2024limit} (note that there is an Alon-Boppana type bound for hyperbolic surfaces at $1/4$, meaning this region is the largest possible). 
The related problem of studying the  the spectral edge of hyperbolic surfaces is also of much interest  \cite{hide2023near, anantharaman2023friedman}. Specifically, the methods of \cite{anantharaman2023friedman} include proving analogous results to those of Friedman's original theorem. Moreover, similar ideas to the eigenvalue multiplicity problem on graphs from \cite{jiang2021equiangular} and \cite{mckenzie2021support}  have been used to upper and lower bound the multiplicity of eigenvalues of negatively curved surfaces  \cite{letrouit2024maximal}. As our method applies a generalization of Friedman's theorem, and Green's function analysis can deduce spectral behavior in both settings, we are curious as to whether our result can generalize in such a way that it can prove an arbitrary spectral edge for hyperbolic surfaces.

\quad

\noindent \textbf{Acknowledgements.}
T.M. thanks Jiaoyang Huang and Horng-Tzer Yau for helpful discussions.

\section{Preliminaries}\label{sec:prelim}
\noindent \textbf{Notation.} We fix $d\geq 3$ to be constant throughout the paper. For a graph $\GG=(V,E)$ (possibly infinite), let $A:=A_\GG$ denote the adjacency operator of $\GG$. Unless otherwise specified, $\GG$ may be a multigraph. We say that $\cG$ is simple if it does not have loops and multiedges. 
The spectrum of $\GG$ is defined as
\[
\sigma(\GG)\deq\{\lambda\in \R:A-\lambda\bI\textnormal{ does not have a bounded inverse}\}
\]
where $\bI$ is the identity operator. 
For every $i\geq 1$, $\lambda_i(\GG)$ will represent the $i$th largest value in the spectrum, allowing for continuous spectrum. Namely,
\[
\lambda_i(\GG)\deq \sup_{\substack{S\subseteq  \R^{|V|}\\\dim(S)=i}}\inf_{v\in \ell_2(S)} v^*Av.
\] For a vector $\psi$ and a subset of coordinates $W$, we let $\psi_W$ denote the projection of $\psi$ onto $W$.

For a graph $\cG=(V,E)$, a cycle in $\cG$ is a sequence of distinct edges $e_1,\dots,e_k\in E$ such that for each $e_i=u_iv_i$, we have $v_i=u_{i+1}$ (where $k+1$ is identified with 1). We define  $\girth(\GG)$ to be the length of the shortest cycle in $\GG$. For $v\in V$ and $r\geq 1$, let $\cB_{r}(v,\cG)$ denote the radius $r$ ball centered at $v$ in $\cG$. For $A,B\subseteq V$, let $\dist_\GG(A,B)$ denote the distance from $A$ to $B$ in $\GG$. For $U\subseteq V$, we use $\GG\setminus U$ to denote the induced subgraph of $\GG$ on $V\setminus U$. For a subgraph $\GG'\subseteq \GG$, we define $\GG\setminus \GG'$ to be the graph on $V(\GG)$ with edge set $E(\GG)\setminus E(\GG')$.

Let $V_0$ be a set of $N$ vertices. The random $d$-regular graph on $V_0$ denotes the configuration model (see \cite[Section 2]{wormald1999models}), formed through choosing a uniformly random matching on the set of half-edges
$\{(i,j):i\in V_0,\, j\in [d]\}$,
and adding an edge $i_1i_2$ to the graph for every matched pair $(i_1,j_1),(i_2,j_2)$.

We use standard asymptotic notation $O(\cdot),o(\cdot),\Omega(\cdot),\omega(\cdot),\Theta(\cdot)$. We say that $f= \wt O(g)$ if $|f(N)|\leq |s(N)g(N)|$ for some subpolynomial function $s$.  We say that $a\lesssim b$ if $a(N)=O(b(N))$, $a\ll b$ if $a=o_N(b)$, and  $a\asymp b$ if $a(N)=\Theta (b(N))$. All logarithms are base $\se$ unless otherwise specified.

\subsection{Tree extension and eigenvector localization}\label{subsec:tree-extension}

In this section, we give formal definition of the tree extension operator $T^L$ mentioned in the introduction, and show that tree extensions have localized eigenvectors for large eigenvalues.

\begin{dfn}[$d$-augmentation]\label{def:d-augment}
    Let $\cH$ be a graph with maximum degree at most $d$. The $d$-augmentation of $\cH$, written as $T\cH$, is obtained by adding $d-\deg(v)$ distinct new neighbors to every $v\in V(\cH)$. In particular, all vertices in $T\cH$ have degree $d$ or 1. If $\cH$ is $d$-regular, then we have $T\cH=\cH$.
\end{dfn}

\begin{dfn}[Tree extension]\label{def:tree-extension}
    Let $\cH$ be a graph with maximum degree at most $d$, and $L\geq 1$. The depth-$L$ tree extension $T^L\cH$ is defined by
    \[
    T^L\cH=\underbrace{T\cdots T}_{\text{$L$ times}}\cH.
    \]
    Given a tree extension $T^{L}\cH$, for every $\ell\geq 0$, we define $V_\ell$ to be the set of vertices in $T^{L}\cH$ that have distance exactly $\ell$ from $\cH$.
\end{dfn}

\begin{lemma}[Eigenvector localization]\label{lem:localization}
    Suppose $\GG=T^{\infty}\cF$ for some $\cF$. Then $\sigma(\GG)$ has no continuous spectrum in the
region $(2\sqrt{d-1},\infty)$, and for any eigenvalue $\lambda>2\sqrt{d-1}$ of $\cG$ with unit eigenvector $\psi$, the projection of $\psi$ onto $V_0$, denoted as $\psi_0$, satisfies $\|\psi_0\|> \frac{\lambda-2\sqrt{d-1}}{2d}$.
\end{lemma}
\begin{proof}
Let $A$  be the adjacency operator of $\cG$. For any normalized vector $\psi$, let $\psi_0$ denote the projection of $\psi$ onto $V_0$ and $\psi_{\geq 1}$ denote the projection of $\psi$ onto $V\setminus V_0$.
Since $\cG\setminus V_0$ is a disjoint union of infinite rooted $(d-1)$-ary trees, the spectral radius of which are $2\sqrt{d-1}$, we have $\psi_{\geq 1}^*A\psi_{\geq 1}\leq 2\sqrt{d-1}$. Therefore there are finitely many eigenvalues greater than $2\sqrt{d-1}$, and they form pure point spectrum. 
     
Suppose $\lambda>2\sqrt{d-1}$ is an eigenvalue of $\cG$ with unit eigenvector $\psi$. Then we have
\[
\lambda=\psi^*A\psi=\psi_{0}^*A\psi_{0}+2\psi^*_0A\psi_1+\psi_{\geq 1}^*A\psi_{\geq 1}\leq d\|\psi_0\|^2+2\|\psi_0\|\sqrt{1-\|\psi_0\|^2}+2\sqrt{d-1}(1-\|\psi_0\|^2).
\]
Solving for $\|\psi_0\|$, we recover that $\|\psi_0\|>\frac{\lambda-2\sqrt{d-1}}{2d}$.
\end{proof}
The following lemma shows that eigenvectors of $T^\infty\cF$ with large eigenvalues have exponentially decaying eigenvectors at some quantifiable rate down the tree. 
\begin{lemma}\label{lem:expodecay}
Suppose $\cG=T^{\infty}\cF$ has eigenvalue $\lambda=(\zeta+\frac1\zeta)\sqrt{d-1}$ for some $\zeta>1$ with normalized eigenvector $\psi$. Then for all $\ell\geq 0$, we have $\sum_{u\in V_\ell}|\psi(u)|^2\leq \zeta^{-2\ell+2}$.
\end{lemma}

\begin{proof}

Fix any $v\in V_1$. Let $W$ denote the set of $v$ and its descendants in $\cG$. We create a new graph $\cG_v$ by taking the infinite rooted $(d-1)$-ary tree $\cG[W]$ and adding a loop  of weight $\lambda-\sum_{u\sim v, u\in W}\psi(u)/\psi(v)$ onto $v$. Then $\lambda$ is an eigenvalue of $\cG_v$ with eigenvector $\psi_W$.

Deleting the root vertex from $\cG_v$ will leave a graph with spectral radius $2\sqrt{d-1}$. Therefore, by Cauchy interlacing, $\lambda$ is the unique eigenvalue of $\cG_v$ that is greater than $2\sqrt{d-1}$. We now claim that the loop has weight $\zeta\sqrt{d-1}$. Indeed, assume the loop is of weight $\zeta'\sqrt{d-1}$ for $\zeta'>1$. Define the vector $\phi\in \R^{|W|}$ such that for every $u\in W$,
\[
\phi(u)\deq\psi(v)(\zeta')^{-\ell}(d-1)^{-\ell/2}
\]
where $\ell$ is the distance between $u$ and $v$. One can check that $\GG_v$ has eigenvalue $\lambda=(\zeta'+1/\zeta')\sqrt{d-1}$ with eigenvector $\phi$. Moreover, by the Perron-Frobenius theorem, the spectral radius is monotonic in the weight of the loop, meaning that to have eigenvalue $\lambda$, we must have $\zeta'=\zeta$. Therefore $\phi$ is a constant multiple of $\psi_W$. 

 Squaring and summing, we get that
\[
\sum_{u\in V_{\ell+1}}|\psi(u)|^2\leq \sum_{v\in V_1}(d-1)^\ell\cdot  |\psi(v)|^2\zeta^{-2\ell}(d-1)^{-\ell}\leq \zeta^{-2\ell}.
\]
\end{proof}

Due to the localization of this eigenvector, trees of finite depth will approximately have the same large eigenvalues. 
\begin{cor}\label{cor:truncation}
    Let $\cF$ be a finite graph with maximum degree at most $d$. Then for any $\lambda>2\sqrt{d-1}$ in $\sigma(T^\infty \cF)$, there exist $C,c>0$ such that  for all $L\geq 0$, there exists some $\lambda'\in \sigma(T^L \cF)$ with $|\lambda-\lambda'|<C(1+c)^{-L}$.
\end{cor}

We now state the Alon--Wei patching lemma \cite[Lemma 4.4]{alon2023limit}. As mentioned in the introduction, a combination of the patching lemma with localization properties of tree extensions will give a reduction from \Cref{thm:main} to \Cref{thm:main-2}.

\begin{dfn}[$R$-patching]\label{def:patching}
    Fix $R\geq 3$. Let $\cF_0$ be a $d$-regular graph with girth at least $8R$, and $\cF_1,\dots,\cF_{k}$ be graphs with maximum degree at most $d$. Moreover, suppose
    \begin{enumerate}
        \item $\sum_{i=1}^{k}|\{v\in V(T\cF_i):\deg_{T\cF_i}(v)=1\}|=Md$ for some $M\geq 1$,
        \item there exists $U\subseteq V(\cF_0)$, $|U|=M$ such that $\text{dist}_{\cF_0}(u,v)> 4R$ for every pair of vertices $u,v\in U$.
    \end{enumerate}
    Then an $R$-patching of $\cF_1,\dots,\cF_{k}$ to $\cF_0$ is constructed by identifying the $Md$ vertices of degree 1 in $\bigcup_{i=1}^{k} T\cF_i$ with the $Md$ vertices of degree $d-1$ in $\cF_0\setminus U$. Observe that every $R$-patching is a $d$-regular graph.
\end{dfn}

\begin{lemma}[{\cite[Lemma 4.4]{alon2023limit}}]\label{lem:patching}
Fix $R\geq 3$ and $L\geq 8R$. Suppose $P$ is an $R$-patching of $T^L\cF_1,\dots,T^L\cF_{k}$ to $\cF_0$, such that $\cF_0$ has girth at least $8R$ and $T^L\cF_1,\dots,T^L\cF_{k}$ each has girth at least $4R$. Let $$\mu=\max\{\lambda_2(\cF_0),\lambda_2(T^L\cF_{1}),\dots, \lambda_2(T^L\cF_{k})\},$$ and suppose
 $$\lambda_1(T^L\cF_1)\geq \dots\geq \lambda_1(T^L\cF_{k})\geq\max\{2\sqrt{d-1},\mu\}.$$ Then for every $2\leq i\leq k+1$, we have $$|\lambda_i(P)-\lambda_1(T^L\cF_{i-1})|\leq \max\left\{\frac{\sqrt{d-1}}{R}, \frac{2d^3\sum_{i=1}^{k}|V(T^L\cF_i)|}{|V(\cF_0)|}\right\}.$$
\end{lemma}

\subsection{Nonbacktracking walk matrix}\label{subsec:nonbacktracking}

In this section, we introduce the notion of a nonbacktracking walk matrix.

\begin{dfn}\label{def:directed-edges}
    For a graph $\GG=(V,E)$ and $e=\{u,v\}\in E$ an undirected edge in the graph, let $e_1,e_2=(u,v),(v,u)$ denote the two directed edges with underlying edge $e$. Define $\hat e_i$ to be the unique reversing of $e_i$ (so $\hat e_1=e_2$ and $\hat e_2=e_1$). Let $\vec E=\bigcup_{e\in E}\{e_1,e_2\}$ denote the set of directed edges in $\GG$.
\end{dfn}

\begin{dfn}[Nonbacktracking walk matrix]\label{def:nbw}
For a graph $\GG=(V,E)$, the nonbacktracking walk matrix of $\GG$ is the $\vec E\times \vec E$ matrix $B$ given by 
    \[
    B_{f,e}=\mathbf 1(v=x, f\neq \hat e) \qquad \text{for all $e=(u,v),\, f=(x,y)\in\vec E$.}
    \]
\end{dfn}

\begin{dfn}\label{def:nbw-2}
    Let $\GG=(V,E)$ be a graph and $\ell\geq 1$. A  nonbacktracking walk of length $\ell$ in $\GG$ is defined to be a sequence of directed edges $e_1\cdots e_{\ell+1}$  such that
    \begin{itemize}
        \item $e_i\in \vec E$   for all $1\leq i\leq \ell+1$,
        \item for all $1\leq i\leq \ell$ with $e_i=(u,v)$ and $e_{i+1}=(u',v')$, we have $v=u'$ and $e_i\neq \hat e_{i+1}$.
    \end{itemize}
    We say that a nonbacktracking walk $e_1\cdots e_{\ell+1}$ starts at $e$ and ends at $f$ if $e_1=e$ and $e_{\ell+1}=f$. In particular, for every $\ell\geq 1$, $\tr(B^\ell)$ equals the number of length-$\ell$ nonbacktracking walks in $\cG$ that start and end at the same edge.
\end{dfn}

We now state the Ihara--Bass formula that gives a correspondence between eigenvalues and eigenvectors of $\GG$ and its nonbacktracking walk matrix $B$. The following is stated explicitly for finite graphs in {\cite[Proposition 3.1]{LP16}}, however, the infinite version is shown and utilized in the proof of {\cite[Theorem 1.5]{angel2015non}}.

\begin{prop}[]
\label{prop:ihara}
Let $\GG=(V,E)$ be a potentially infinite $d$-regular graph. For every eigenvalue $|\mu|>2\sqrt{d-1}$, for the unique $|\zeta|>1$ satisfying 
$\mu=(\zeta+\zeta^{-1})\sqrt{d-1}$, there are two corresponding eigenvalues of $B$, $\zeta\sqrt{d-1}$ and $\zeta^{-1}\sqrt{d-1}$. Besides eigenvalues of this form, the spectrum of $B$ is contained in the complex disk of radius $\sqrt{d-1}$.

For every $\theta\in\CC$, define the map $S_\theta:\RR^V\to \CC^{\vec E}$ by
$$(S_\theta f)(u,v)=\theta f(v)-f(u) \qquad \text{ for all $(u,v)\in\vec E$}.$$
Let $\mu$ be an eigenvalue of $\GG$ with eigenvector $\psi$. Suppose $\theta\neq \pm 1$ satisfies $\theta^2 -\mu\theta+(d-1)=0$. Then $\theta$ is an eigenvalue of $B$ with eigenvector $S_\theta \psi$.
\end{prop}
\begin{rmk}
    Note that if $\mu> 2\sqrt{d-1}$, then the given $\zeta$ is the constant from \Cref{lem:expodecay}.
\end{rmk}

\subsection{Green's function preliminaries}\label{subsec:green}
Let $\C^+$ denote the set of complex numbers $z$ with $\Im[z]>0$. We define the \emph{Green's function} of the normalized adjacency operator $H\deq A_{\GG}/\sqrt{d-1}$ of a $d$-regular graph $\GG=(V,E)$ by
\begin{equation*}
  G(\GG,z) \deq  (H-z\bI)^{-1},\quad z\in \C^+.
\end{equation*}
When $H$ is finite-dimensional, we have the spectral decomposition 
\[
G(\GG,z)=\sum_{i=1}^N \frac{\psi_i \psi_i^\top}{\lambda_i-z}
\]
where $\psi_i$ is the unit eigenvector for $\lambda_i(\GG)$.
We will drop the dependence on $\GG$ and $z$ when it is clear from the context. 

Note that the Green's function has the interpretation as a walk generating function, as 
\[
G=-\frac1{z}-\frac{A_{\cG}}{z^2\sqrt{d-1}}-\frac{A_{\cG}^2}{z^3(d-1)}-\cdots.
\]
Although we will consider $z$ that lie within the radius of convergence, this interpretation will remain useful in our analysis. 

Let $V_0$ be a vertex subset of $\cG$.
We denote the $V_0$-\emph{Stieltjes transform} of $H$ by
\begin{equation} \label{e:m}
  m_N(z) \deq \frac{1}{N} \sum_{i\in V_0} G_{ii}(z).
\end{equation}

Therefore, rather than taking the entire spectral density, we have reweighted our spectral density by the norm of the projection onto $V_0$. Therefore for the spectral density $\varrho(x)$ of $H$, 
\[
m_N(z)=\int_{\RR} \frac{P_{V_0}(x)\varrho(x)}{x-z}\rd x,
\]
where $P_{V_0}(x)$ is the spectral projector onto $V_0$ at eigenvalue $x$. In fact, to bound the largest eigenvalues, it is sufficient to consider $m_N(z)$ as we know by \Cref{lem:localization} that a constant fraction of all eigenvector mass for large eigenvalue is contained in $\psi_0$.  Projecting to this finite set allows us to use the combinatorial properties of the Green's function.

We define the \emph{Kesten-McKay measure} to be the spectral density of the infinite tree
\begin{align}
\varrho_d(x):=\mathbf1_{x\in [-2,2]} \left(1+\frac1{d-1}-\frac {x^2}d\right)^{-1}\frac{\sqrt{4-x^2}}{2\pi}.
\end{align}
Our goal will be to approximate $m_N(z)$ with $m_d(z)$,
the Stieltjes transform of the Kesten--McKay law
\begin{align*}
    m_d(z):=\int_\R \frac{\varrho_d(x)\rd x}{x-z},\quad z\in \C^+.
\end{align*}

We also recall the \emph{semi-circle distribution} $\varrho_{\rm sc}(x)$, its Stieltjes transform $m_{sc}(z)$, and the quadratic equation satisfied by $m_{sc}(z)$:
\begin{align}\label{eq:msc}
 \varrho_{\rm sc}(x):=\vone_{x\in[-2,2]}\frac{\sqrt{4-x^2}}{2\pi},
 \quad 
  m_{sc}(z):=\int_\R \frac{\varrho_{\rm sc}(x)\rd x}{x-z}=\frac{-z+\sqrt{z^2-4}}{2},\quad 
  m_{sc}(z)^2+zm_{sc}(z)+1=0.
\end{align}
The Stieltjes transform of the Kesten--McKay law $m_d(z)$ can be expressed in terms of the Stieltjes transform $m_{sc}(z)$ through a Schur complement, 
\begin{align}\label{e:md_equation}
    m_d(z)=\frac{1}{-z-\frac{d}{d-1}m_{sc}(z)}.
\end{align}
Note that the Kesten-McKay law is absolutely continuous with respect to the semicircular law. 

As with previous arguments \cite{huang2024optimal,huang2024spectrum}, we will deduce $m_N\approx m_d$ by considering a function $Q$ that approximates the semicircular law $m_{sc}$, and showing that $Q\approx m_{sc}$ implies $m_N\approx m_d$.
To this end, for any $\bV\subseteq V$, we define $G^{(\bV)}$  to be  the Green's function of the graph $\GG^{(\bV)}\deq\GG\setminus\bV$. Let $\vec E_0\subseteq \vec E$ be the set of directed edges supported on $V_0$, and let
\be\label{eq:qdef}
Q(\GG,z):=\frac1{|\vec{E}_0|}\sum_{(o,i)\in \vec{E}_0}G_{oo}^{(i)}.
\ee
Note that $Q$ is not uniform over $o\in V_0$ and depends on the degree of $o$ in $\GG_0$. We will show in \Cref{sec:newgf} that $Q-m_{sc}$ is small enough to bound $\lambda_2(T^\infty \GG)$.

\subsection{Properties of random regular graphs}\label{subsec:rrg}

In this section, we state  properties of random regular graphs that we will utilize. These concern the spectrum and distribution of cycles. 

\begin{dfn}
    Let $\cH$ be a graph with possibly loops and multiedges. We say that $\cH$ is tangle-free if it contains at most one cycle (loops and multiple
edges count as cycles). We say that $\cH$ is $\ell$-tangle-free if every neighborhood of radius $\ell$ in $\cH$ contains at most
one cycle.
\end{dfn}

\begin{lemma}[{\cite[Lemma 4.1]{bauerschmidt2019local}}]
\label{lem:tanglefree}Let $\cH$ be a random $d$-regular graph on $N$ vertices.
Then for any constant $0<\fc<1$, with probability at least $1-O(N^{-1+\fc})$,  all vertices in $\cH$ are $(\frac{\fc}{4}\log_{d-1}N)$-tangle-free, and all but $N^\fc$ vertices have tree neighborhood of radius $\frac{\fc}4\log_{d-1}N$.
\end{lemma}

\begin{lemma}[{\cite[Corollary 2]{mckay2004short}}]\label{lem:simple}
    Let $\cH$ be a random $d$-regular graph on $N$ vertices. Then for any fixed $R\geq 0$, we have
    \[
    \PP[\girth(\cH)\geq R]=\exp\left({-\sum_{r=1}^R\frac{(d-1)^r}{r}+o_N(1)}\right).
    \]
\end{lemma}

 We will also utilize the corollary of \cite[Theorem 4]{garmo1999asymptotic}, that  the distribution of long cycles of random regular graphs is well behaved.
\begin{lemma}\label{lem:longcycles}
Let $\cH$ be a random $d$-regular graph on $N$ vertices. Suppose $k$ is a subpolynomial function of $N$ that satisfies  $10\log_{d-1}N\leq k$. Let $U_{\cH}$  be the set of length-$k$ cycles in $\cH$. Then with probability $1-o_N(1)$, we have $|U_{\cH}|=(1+\wt{O}(N^{-1}))\E[|U_{\cH}|]$.
\end{lemma}

Finally we state the sharper version of Friedman's theorem appearing as \cite[Theorem 1.3]{huang2024spectrum}.

\begin{lemma}\label{lem:fried}
   Let $\cH$ be a random $d$-regular graph on $N$ vertices.
With probability at least $1-N^{-1+o_N(1)}$, 
    \[
\max\{\lambda_2(A_\cH),-\lambda_N(A_\cH)\}\leq 2\sqrt{d-1}+O(N^{-.01}).
    \]
\end{lemma}

\subsection{Percolation theory and branching processes}
\label{subsec:percolation}

For $0<p<1$, we define \textit{the $p$-percolation} of a graph $\cH$, denoted as $\mathrm{perc}(\cH,p)$, to be the random subgraph obtained by including each edge of $\cH$ independently with probability $p$. 

At one point in the proof of \Cref{thm:main-2}, we will show a preliminary spectral gap of $\cG_0$, the $p$-percolation of the random $d$-regular graph, between its first and second eigenvalues. To achieve this we need to control the general shape of $\cG_0$ and show that it approximates the percolation of the infinite tree. Let $\GG^{*}=\{(\cG,v):v\in V(\cG)\}$ denote the set of rooted graphs with underlying graph $\cG$. For a function $\tau:\GG^{*}\rightarrow [0,1]$, we say that $\tau$ is  $h$-\emph{local} if each $\tau(\cG,v)$ only depends on the radius $h$ neighborhood of $v$. The following lemma is a straightforward application of an Azuma inequality.  
\begin{lemma}[{\cite[Lemma 30]{bordenave2015quantum}}]\label{lem:bordenaveconcentration}
Let $\cH$ be a $d$-regular graph on $N$ vertices, and $S\subseteq V(\cH)$ be a vertex subset such that every $v\in S$ has a tree neighborhood of radius $h$. Let $p\in [0,1]$ and $\cG=perc(\cH,p)$. Suppose $\tau:\cH^{*}\rightarrow{[0,1]}$ is $h$-local. Then for any $\gamma\geq 0$, we have
\[
\Pr\left(\left|\sum_{v\in S}\tau(\cG,v)-\E\sum_{v\in S}\tau(\cG,v)\right|\geq |S|\gamma\right)\leq 2\exp\left(-\frac{|S|\gamma^2}{2\left(1+d\frac{(d-1)^{h}-1}{d-2}\right)^2}\right).
\]
\end{lemma}
$1+d\frac{(d-1)^{h}-1}{d-2}$ is the size of a tree neighborhood of radius $h$ in a $d$-regular graph. 

We will also make use of the Kesten--Stigum theorem (\cite{kesten1966limit}, see also \cite{lyons1995conceptual}),  which states that the normalized size of the depth-$\ell$ neighborhood in a supercritical branching process converges to an explicit distribution. We state a simplified version that applies to our work here. 
\begin{lemma}[{\cite[Theorem 1]{lyons1995conceptual}}]\label{lem:KS}Let $\mathcal T$ be the infinite $(d-1)$-ary tree and  $\cG$ be the $p$-percolation of $\mathcal T$, with $p>\frac{1}{\sqrt{d-1}}$.  For any vertex $v$ in $\mathcal T$, let $Z_\ell$ denote the number of walks from $v$ to the boundary of $\cB_\ell(v,\cG)$. Then there is a scalar random variable $W$ with positive moments and  $\E[W]=1$ such that 
\[
\lim_{\ell\rightarrow \infty} \frac{Z_\ell}{(p(d-1))^\ell} \rightarrow W.
\]
\end{lemma}

We will also use that the moment generating function of $ \frac{Z_\ell}{(p(d-1))^\ell}$ is bounded. Specifically, the distribution has exponential tails. 

\begin{lemma}[{\cite[Theorem 2]{biggins1993large}}]\label{lem:biggins}
    Let $d$, $p$, $Z_\ell$ be as above. Then for sufficiently large $\ell'$, there is a constant $c_{\ell'}>0$ such that for $\ell\geq \ell'$, 
    \be\label{eq:expouppertail}\Pr(Z_{\ell}\geq x p^\ell(d-1)^\ell)\leq \exp(-c_{\ell'} x ).\ee
\end{lemma}

\section{Proof of \Cref{thm:main}}\label{sec:proof-main}
In this section, we further examine our construction and show that \Cref{thm:main-2} implies  \Cref{thm:main}.
\begin{dfn}\label{def:construction}

    Fix $2\sqrt{d-1}< \mu<d$. We define $\theta$ to be the unique solution to $\theta^2-\mu\theta+d-1=0$ with $\theta\in (\sqrt{d-1},d-1)$.  We set $p\deq\theta/(d-1)$.
Let $V_0$ be a set of $N$ vertices and $\cH$ be the random $d$-regular graph on $V_0$. Let $\GG_0$ be the random subgraph of $\cH$ obtained by including every edge independently with probability $p$ (so $\cG_0=\mathrm{perc}(\cH,p)$).
Let $\GG:=T^\infty \GG_0$ be the infinite tree extension of $\GG_0$.
\end{dfn}

We will estimate the top eigenvalues of $\cG$ conditioned on high probability events. We start by assuming our graph is $\frac16\log_{d-1}N$-tangle-free, and that we have a reasonable number of edges in $V_0$. 
By \Cref{lem:tanglefree}, the graph $\cH$ (and thus $\cG$) is $\frac16\log_{d-1}N$-tangle-free with probability $1-O(N^{-1/3})$.
Let $\vec E$ be the set of directed edges in $\cG$ (\Cref{def:directed-edges}) and $\vec{E}_0:=\{(u,v)\in\vec E:u,v\in V_0\}$ be the set of directed edges on $V_0$. By a Chernoff bound, we have
\be\label{eq:E1size}
\Pr\left(|\vec{E}_0|\leq \frac12pdN\right)\leq \se^{-pdN/16}.
\ee
Because in our regime $p=\Theta_d(1)$, this probability is sufficiently small. 
\begin{dfn}\label{dfn:oOmega}
We define $\overline{\Omega}$ to be the high probability event that the graph $\GG$ is $\frac16\log_{d-1}N$-tangle free and $|\vec{E}_0|\geq \frac12pdN$.
\end{dfn}
We will now show that $\GG$ satisfies the requisite properties for \Cref{thm:main-2}.
The following proposition is proven in \Cref{sec:first-ev}.
\begin{prop}\label{thm:first-ev}
    Let $\GG$ be the graph as in \Cref{def:construction}. Then there exists  $\delta>0$, not depending on $N$, such that with probability $1-o_N(1)$, we have $|\lambda_1(\GG)-\mu|=O(N^{-0.05})$ and $\lambda_2(\GG)\leq (1-\delta)\mu$.
\end{prop}

The following proposition is proven in \Cref{sec:second-ev} and \Cref{sec:newgf}.
\begin{prop}\label{thm:second-ev}  Let $\GG$ be the graph as in \Cref{def:construction}. For any constant $\fq\geq 0$, with probability $1-O(N^{-\fq})$ conditioned on $\oOmega$,  we have $\lambda_2(\GG)\leq 2\sqrt{d-1}+O((\log\log(N))^{-1})$.
\end{prop}

Assuming these two propositions, we can show \Cref{thm:main-2}.
\begin{proof}[Proof of \Cref{thm:main-2}]
Let $\cG$ be given as in \Cref{def:construction}. By construction, $\cG$ satisfies the first condition.
    By \Cref{lem:tanglefree} and \eqref{eq:E1size}, we know that $\oOmega$ occurs with  with probability $1-O(N^{-1/3})$. Conditioned on $\oOmega$,  by Propositions \ref{thm:first-ev} and \ref{thm:second-ev}, $\cG$ satisfies the third condition with probability at least $1-o_N(1)$.   For the second condition, by \Cref{lem:simple}, $\GG$ has girth at least $R$ with constant probability (in particular $\cG$ is simple as long as $R>2$). Thus with constant probability, our random construction satisfies all the conditions in \Cref{thm:main-2}.
\end{proof}

Given \Cref{thm:main-2}, we can now prove \Cref{thm:main}.

\begin{proof}[Proof of \Cref{thm:main}]
First, we will show that we can reduce to the case where $d=\mu_1>\mu_2\geq \dots\geq \mu_k> 2\sqrt{d-1}$. Assume that there is some $j'\geq 0$ such that $\mu_{k-j'}=\dots=\mu_{k}= 2\sqrt{d-1}$. We know from the previously mentioned generalizations to the Alon-Boppana bound (see e.g. \cite{Cio06}) that \emph{any} family of $d$-regular graphs $\{\GG_i\}_{i=1}^\infty$ will satisfy $\liminf_{i\rightarrow \infty}\lambda_{k}(A_{\GG_i})\geq 2\sqrt{d-1}$. Therefore, it is sufficient to show that $\limsup_{i\rightarrow \infty}\lambda_{k-j'}(A_{\GG})\leq 2\sqrt{d-1}$, in which case we use the same construction for $\mu_1,\ldots,\mu_{k-j'-1}$.

Similarly, if $\mu_2=\cdots= \mu_j=d$ for $j\geq 2$, we can take a family of graphs $\{\GG_i'\}_{i=1}^\infty$ that solves the problem for $\mu_1,\mu_{j+1},\mu_{j+2},\ldots,\mu_k$, and take $\GG_i$ to be the union of $\GG'_i$ with $j-1$ disconnected random $d$-regular graphs. With high probability, this will not introduce any further large eigenvalues by \Cref{lem:fried}, so will solve the problem for $\mu_1,\mu_2,\ldots,\mu_k$.

Therefore, we can assume that $d=\mu_1>\mu_2\geq \dots\geq \mu_k> 2\sqrt{d-1}$.
Consider any $0<\epsilon<\min\left\{\frac{d-\mu_2}{10},\frac{\mu_k-2\sqrt{d-1}}{10}\right\}$. 
    Fix $R>\epsilon^{-1}\sqrt{d-1}$. For every $1\leq i\leq k-1$, let  $T^\infty \cF_{i}$ be the graph obtained from applying \Cref{thm:main-2} with parameters $\mu_{i+1}$, $4R$ and $\epsilon$. In particular, we have $\text{girth}(T^\infty \cF_i)\geq 4R$,  $|\lambda_1(T^\infty \cF_i)-\mu_{i+1}|<\epsilon$ and $\lambda_2(T^\infty \cF_i)<2\sqrt{d-1}+\epsilon$.

    By \Cref{cor:truncation}, there exists $L>0$ sufficiently large such that 
     $|\lambda_1(T^\infty \cF_i)-\lambda_1(T^L \cF_i)|<\epsilon$ for all $1\leq i\leq k-1$. Thus, we have $|\lambda_1(T^L \cF_i)-\mu_{i+1}|<2\epsilon$ and in particular $\lambda_1(T^L \cF_i)\geq\mu_{i+1}-2\epsilon\geq 2\sqrt{d-1}+8\epsilon$ for all $1\leq i\leq k-1$. Moreover, we have $\lambda_2(T^L\cF_i)\leq \lambda_2(T^\infty \cF_i)<2\sqrt{d-1}+\epsilon$.

    Finally, let $\cF_0$ be a sufficiently large $d$-regular graph such that
    \begin{itemize}[parsep=0ex]
        \item $\cF_0$ has girth $\geq 8R$,
        \item $\lambda_2(\cF_0)<2\sqrt{d-1}+\epsilon$,
        \item there exists $U\subseteq V(\cF_0)$ such that $|U|=M$ with $Md=\sum_{i=1}^{k-1}|\{v\in V(T^{L+1}\cF_i):\deg _{T^{L+1}\cF_i}(v)=1\}|$, and $\dist_{\cF_0}(u,v)>4R$ for every pair $u,v\in U$,
        \item $|V(\cF_0)|\geq \epsilon^{-1}\cdot 2d^3\sum_{i=1}^{k-1}|V(T^L \cF_i)|$.
    \end{itemize}
By \Cref{lem:simple} and \Cref{lem:fried}, we can always find a graph $\cF_0$ satisfying these points.

Let $P$ be an $R$-patching of $T^L\cF_1,\dots,T^L\cF_{k-1}$ to $\cF_0$ (so  $P$ is $d$-regular).
We now verify that $P$ satisfies the condition of \Cref{lem:patching}. We  already know that $\text{girth}(\cF_0)\geq 8R$ and $\text{girth}(T^L\cF_1),\dots,\text{girth}(T^L\cF_{k-1})\geq 4R$. Moreover, we have
    \[
    \mu:=\max\{\lambda_2(\cF_0),\lambda_2(T^L\cF_{1}),\dots, \lambda_2(T^L\cF_{k-1})\}<2\sqrt{d-1}+\epsilon.
    \]
    Up to reordering $T^L\cF_1,\dots,T^L\cF_{k-1}$ so that their first eignvalues have descending order, we have
    \[
    \lambda_1(T^L\cF_1)\geq \dots\geq \lambda_1(T^L\cF_{k-1})\geq 2\sqrt{d-1}+8\epsilon\geq \max\{2\sqrt{d-1},\mu\}.
    \]
    Therefore by \Cref{lem:patching}, we have $$|\lambda_i(P)-\lambda_1(T^L\cF_{i-1})|<\max\left\{\frac{\sqrt{d-1}}{R}, \frac{2d^3\sum_{i=1}^{k-1}|V(T^L\cF_i)|}{|V(\cF_0)|}\right\}\leq\epsilon\qquad \text{ for all $2\leq i\leq k$,}$$
    where the second inequality follows from our choices of $R$ and $\cF_0$. Altogether, we get that
    \[
    |\lambda_i(P)-\mu_{i}|\leq |\lambda_i(P)-\lambda_1(T^L\cF_{i-1})|+|\lambda_1(T^L \cF_{i-1})-\mu_i|=3\epsilon\qquad  \text{ for all $2\leq i\leq k$.}
    \]
    Taking $\epsilon\to 0$, we can obtain a sequence of $d$-regular graphs that certifies \Cref{thm:main}.
\end{proof}

\section{First eigenvalue bound}\label{sec:first-ev}

\subsection{Test eigenvector}

Recall from \Cref{def:construction} that $\cH$ is the random $d$-regular graph on the $N$-vertex set $V_0$, $\cG_0=\mathrm{perc}(\cH,p)$, and $\cG=T^\infty \cG_0$.
In this section, we show that with high probability $\cG$ has an eigenvalue at $\mu+O(N^{-0.05})$. Note that by definitions of $\theta,p$, we have $\mu=\frac{\theta^2+d-1}{\theta}=p(d-1)+p^{-1}$.

\begin{lemma}\label{lem:perronvector}Let $\cG$ be as in \Cref{def:construction}.
With probability $1-o_N(1)$, there is an eigenvalue $\lambda$ of $\cG$ that satisfies
\begin{align}
|\lambda-(p(d-1)+p^{-1})|=O(N^{-0.05}).
\end{align}
\end{lemma}

We will do this by showing a vector obtained via the power method is close to satisfying the eigenvector equation with our desired eigenvalue. Define \[t=0.22\log_{d-1}N.\] For every $v\in V_0$, define $Z_t(v)$ to be the number of vertices in $V_0$ of distance exactly $t$ from $v$. We define the vector $\psi_t\in\RR^V$ as 
\[
\psi_t(v)=\begin{cases}
    \frac{Z_t(v)}{p^t(d-1)^t} & v\in V_0,\, \cB_{t+1}(v,\cG)\textnormal{ is a tree}
\\
0 & v\in V_0, \,\cB_{t+1}(v,\cG)\textnormal{ is not a tree}\\
p^{-\ell}(d-1)^{-\ell} \psi_t(u) & v\in V_\ell, \ell\geq 1 \text{ lies in the $(d-1)$-ary tree rooted at $u\in V_0$}.
\end{cases}
\]
We now show that $\psi_t$ almost satisfies the eigenvector equation of $A_{\cG}$ with eigenvalue $\mu$.

\begin{lemma}\label{lem:psiell}
Let $A=A_{\cG}$  be the adjacency operator of $\cG$. Then
    with probability $1-\exp(-\Omega(N^{0.06}))$, we have
\end{lemma}
\begin{equation}\label{eq:psiell}
\frac{\|(A-(p(d-1)+p^{-1}))\psi_t\|^2}{\|\psi_t\|^2}=O(N^{-0.1}).
\end{equation}
We first show how this concentration is enough to find $\lambda_1$.
\begin{proof}[Proof of \Cref{lem:perronvector}]
As $A$ is a self-adjoint operator, \Cref{lem:psiell} implies that $A-(p(d-1)+p^{-1})\mathbb I$ has an eigenvalue of magnitude $O(N^{-0.05})$ with high probability, which gives the result. 
\end{proof}

\begin{proof}[Proof of \Cref{lem:psiell}]
To do this, we will apply \Cref{lem:bordenaveconcentration} with the functions
\begin{align*}
\tau_1(G,v)&=\min\left\{\left|(A\psi_t)(v)-(p(d-1)+p^{-1})\psi_t(v)\right|^2,1\right\},\\
\tau_2(G,v)&=\min\{N^{-0.01}|\psi_t(v)|^2, 1\}.
\end{align*}

Since $\psi_t(v)$ is $(t+1)$-local,  $\tau_1$ and $\tau_2$ are $(t+2)$- and $(t+1)$-local, respectively. We define $$S=\{v\in V_0:\cB_{t+2}(v,\cH)\text{ is a tree}\}.$$ By \Cref{lem:tanglefree}, we know that  $|S|\geq N-N^{0.88}$ with probability $1-O(N^{-0.12})$.
Applying \Cref{lem:bordenaveconcentration} with $S$ and $\gamma=N^{-1/4}$, we get that with probability $1-\exp(-\Omega(N^{0.06}))$,
\be\label{eq:tauconcentration}
\left|\sum_{v\in S}\tau_1(H,v)-\E\sum_{v\in S}\tau_1(H,v)\right|,\, \left|\sum_{v\in S}\tau_2(H,v)-\E\sum_{v\in S}\tau_2(H,v)\right|\leq N^{3/4}.
\ee

We now compute the expectation  $\E\sum_{v\in S}\tau_1(H,v)$. For every $v\in S$, let $d_v$ denote the number of neighbors of $v$ in $V_0$. For every $x,y\in V_0$, let $Z_t^{(x)}(y)$ denote the number of vertices $u\in V_0$ of distance exactly $t$ from $y$ such that $x$ does not lie on the length-$t$ walk from $y$ to $u$. Since $$Z_t(x)=\sum_{y\sim x,\, y\in V_0}Z^{(x)}_{t-1}(y),$$ we have that
\begin{align*}
    \sum_{x\sim v}\psi_t(x)&=\frac{1}{p^t(d-1)^t}\left(\sum_{x\sim v,\, x\in V_0}Z_t(x)+ (d-d_v)\frac{Z_t(v)}{p(d-1)}\right) \\
    &=\frac{1}{p^t(d-1)^t}\left(\sum_{x\sim v,\,x\in V_0}\left(Z_{t-1}^{(x)}(v)+\sum_{y\sim x,\, y\in V_0\setminus\{v\}}Z_{t-1}^{(x)}(y)\right)+(d-d_v)\frac{Z_t(v)}{p(d-1)}\right).
\end{align*}
Observe that 
\[
\sum_{x\sim  v,\, x\in V_0}Z_{t-1}^{(x)}(v)=(d_v-1)Z_{t-1}(v)
\]
and
\[
\sum_{x\sim v,\, x\in V_0}\sum_{y\sim x,\, y\in V_0\setminus\{v\}}Z_{t-1}^{(x)}(y)=Z_{t+1}(v).
\]
Thus we get that
\[
(A\psi_t)(v)=\sum_{x\sim v}\psi_t(x)=\frac{1}{p^t(d-1)^t}\left(Z_{t+1}(v)+(d_v-1)Z_{t-1}(v)+(d-d_v)\frac{Z_t(v)}{p(d-1)}\right),
\]
which gives
\begin{align*}
    &(A\psi_t)(v)-(p(d-1)+p^{-1})\psi_t(v)\\&=\frac{1}{p^t(d-1)^t}\left(Z_{t+1}(v)+(d_v-1)Z_{t-1}(v)+(d-d_v)\frac{Z_t(v)}{p(d-1)}-p(d-1)Z_{t}(v)-p^{-1}Z_t(v)\right)\\
    &=\frac{1}{p^t(d-1)^t}\left(\left(Z_{t+1}(v)-p(d-1)Z_t(v)\right)-\frac{d_v-1}{p(d-1)}\left(Z_t(v)-p(d-1)Z_{t-1}(v)\right)\right).
\end{align*}

Since
\[
\E[(Z_t(v)-p(d-1)Z_{t-1}(v))^2\mid Z_{t-1}(v)]=\text{Var}[Z_t(v) \mid Z_{t-1}(v)]=p(1-p)(d-1)Z_{t-1}(v),
\]
we get that
\[
\E[\tau_1(H,v)\mid Z_{t-1}(v)]\leq \E[|(A\psi_t)(v)-(p(d-1)+p^{-1})\psi_t(v)|^2 \mid Z_{t-1}(v)]=O\left(\frac{Z_{t-1}(v)}{p^{2t}(d-1)^{2t}}\right).
\]
For every $v\in S$, we have
$\E[Z_{t-1}(v)]=p^{t-1}d(d-1)^{t-2}$. Therefore, we have
\[
\EE[\tau_1(H,v)]=O\left(\frac{p^{t-1}d(d-1)^{t-2}}{p^{2t}(d-1)^{2t}} \right)=O(N^{-0.11}).
\]

By \eqref{eq:tauconcentration}, we get that with probability $1-\exp(-\Omega(N^{0.06}))$,
\begin{align}
    \label{eq:tau1bound}
\sum_{v\in S}\tau_1(H,v)=O(N^{0.89}).
\end{align}

Along with this, we need to consider the other part $\sum_{v\in V\setminus S}\tau_1(H,v)$. Observe that for every $v\in V\setminus V_0$, we have $(A\psi)(v)=(p(d-1)+p^{-1})\psi(v)$ so $\tau_1(H,v)=0$. By \Cref{lem:tanglefree}, we know that  $|V_0\setminus S|\leq N^{0.88}$ with probability $1-O(N^{-0.12})$.
By \eqref{eq:expouppertail} we can further assume that $\psi_t(v)\leq N^{0.01}$ for all $v\in S$, which gives $||\psi_t||_\infty\leq N^{0.01}$. Thus we have the bound
\[
\sum_{v\in V\setminus S}\tau_1(H,v)\leq d||\psi_t||_\infty^2N^{0.88}\leq d N^{0.02+0.88} =O(N^{0.9}).
\]

Combining this with \eqref{eq:tau1bound} gives 
\be\label{eq:psiellerror}
\|(A-(p(d-1)+p^{-1}))\psi_t\|^2=O(N^{0.9}).
\ee

We now estimate $\sum_{v\in S}\E[\tau_2(H,v)]$. By \eqref{eq:expouppertail}, with high probability we have $N^{-0.01}|\psi_t(v)|^2\leq 1$ for all $v\in S$, so $\E(\tau_2(G,v))\asymp N^{-0.01}\E[\psi_t(v)^2]$.
Moreover, by \Cref{lem:KS}, we have
\be\E[\tau_2(G,v)]=\Omega(N^{-0.01}).\ee
Thus, we get that with high probability
\begin{align}
\label{eq:tau2bound}||\psi_t||^2\geq N^{0.01}\sum_{v\in S}\tau_2(G,v)\geq N^{0.01}\left(\sum_{v\in S}\EE[\tau_2(G,v)]-N^{3/4}\right)=\Omega(N).
\end{align}

Moreover, with probability $1-\exp(-\Omega(N))$, $\tau_1$ is never 1 on $S$, so we condition on this as well. Combining this with \eqref{eq:psiellerror} and \eqref{eq:tau2bound} gives that
\begin{align*}
\frac{\|(A-(p(d-1)+p^{-1}))\psi_t\|^2}{\|\psi_t\|^2}=O(N^{-0.1}).
\end{align*}
\end{proof}

\subsection{Proof of \Cref{thm:first-ev}}

So far we have shown that $\cG$ has an eigenvalue close to $p(d-1)+p^{-1}$. It remains to show that it is indeed the largest eigenvalue of $\cG$. 
We start with the following observation. Recall that $B$ is the nonbacktracking walk operator of $\cG$ (\Cref{def:nbw}).

\begin{lemma}\label{lem:eigpreservation}
Suppose $\theta>\sqrt{d-1}$ is an eigenvalue of $B$ with unit eigenvector $\phi$. Let $B_0$ denote the $\vec E_0\times \vec E_0$ submatrix of $B$ (with $\vec E_0$ defined in \Cref{sec:proof-main}), and $\phi_0$ denote the projection of $\phi$ onto $\vec E_0$. Then  $\theta$  is also an eigenvalue of  $B_0$ with eigenvector $\phi_0$.
\end{lemma}
\begin{proof}
Let $(u,v)$ be a directed edge in $\vec E\setminus \vec E_0$ that points towards $V_0$ (i.e., $\text{dist}(v,V_0)=\text{dist}(u,V_0)-1$). Then for all $k\geq 1$, we have $$\theta^k\phi(u,v)=(B^k\phi)(u,v)=\sum_{(a,b)\in C_k(u,v)}\phi(a,b),$$
where $C_k(u,v)$ is the set of distance-$k$ children of $(u,v)$. Since $|C_k(u,v)|=(d-1)^k$, we have
\[
\sum_{(a,b)\in C_k(u,v)}|\phi(a,b)|^2\geq (d-1)^k \left(\sum_{(a,b)\in C_k(u,v)}\frac{|\phi(a,b)|}{(d-1)^k}\right)^2\geq (d-1)^k\left(\frac{\theta^k|\phi(u,v)|}{(d-1)^k}\right)^2=\left(\frac{\theta}{\sqrt{d-1}}\right)^{2k}|\phi(u,v)|^2.
\]
If $\phi(u,v)\neq 0$, then this diverges as $k\rightarrow \infty$, contradicting the fact that $||\phi||=1$. Therefore we must have $\phi(u,v)=0$.

Observe that for all $(u,v)\in \vec E_0$, if $B_{(u,v)(a,b)}=1$ for some $(a,b)\in \vec E$, then either (1) $(a,b)\in \vec E_0$, or (2) $(a,b)\in \vec E\setminus \vec E_0$ is pointed towards $V_0$, in which case $\phi(a,b)=0$. For all $(u,v)\in \vec E_0$, we have
\begin{align*}
    B_0\phi_0(u,v)&=\sum_{(B_0)_{(u,v)(a,b)}=1}\phi_0(a,b)=\sum_{\substack{B_{(u,v)(a,b)}=1\\(a,b)\in \vec E_0}}\phi(a,b)\\
    &=B\phi(u,v)=\theta\phi(u,v)=\theta\phi_0(u,v),
\end{align*}
so we get that $\theta$ is also an eigenvalue of $B_0$ with eigenvector $\phi_0$.
\end{proof}

Since $\cG$ has an eigenvalue $p(d-1)+p^{-1}+o_N(1)>2\sqrt{d-1}$, by \Cref{lem:eigpreservation}, we know that $B$ has an eigenvalue greater than $\sqrt{d-1}$. Thus by \Cref{lem:eigpreservation}, the first eigenvalues of $B$ and $B_0$ agree. 

We now proceed to estimate the first eigenvalue of $B_0$ using a trace method. 
We start by estimating the number of $k$-cycles in $\cH$, the random $d$-regular graph on $V_0$. For convenience in counting, in the remainder of this section we assume that every $k$-cycle is an \textit{ordered} $k$-tuple of edges (so $e_1\dots e_k$ and $e_2\cdots e_ke_1$ are treated as distinct cycles). 

\begin{prop}\label{prop:numberofcycles}
 Let $\cH$ be a random $d$-regular graph on $N$ vertices. Let $k$ be a subpolynomial function of $N$ with $k\geq 10\log_{d-1}N$, and let $U_{\cH}$ denote the set of $k$-cycles in $\cH$.
 Then with high probability we have $|U_{\cH}|=(1+\wt{O}(N^{-1}))(d-1)^{k}$.
\end{prop}
\begin{proof}
We first compute $\EE|U_{\cH}|$. Observe that in a random matching on $Nd$ half-edges, the probability that $k$ explicit pairs get matched equals
\begin{align}\label{eq:rrg-subgraph}
    \frac{(dN-2k)!!}{(dN)!!}=\frac{1}{(dN-1)(dN-3)\cdots(dN-2k+1)}=(1+\wt{O}(N^{-1})) \frac{1}{(dN)^k}.
\end{align}
Observe that $|U_{\cH}|$ is at least the number of $k$-cycles in $H$ that occupies $k$ distinct vertices. Thus, letting $(N)_k:=N(N-1)\cdots(N-k+1)$, we can explicitly compute that
\begin{align*}
    \EE[|U_{\cH}|]&\geq (1+\wt{O}(N^{-1}))\frac{(N)_{k}d^{k}(d-1)^{k}}{(dN)^{k}}=(1+\wt{O}(N^{-1})) (d-1)^{k}\cdot \frac{(N)_{k}}{N^{k}}=(1+\wt{O}(N^{-1})) (d-1)^{k}.
\end{align*}
On the other hand, by choosing non-distinct vertices, \eqref{eq:rrg-subgraph} gives
\begin{align*}
    \EE[|U_{\cH}|]&\leq (1+\wt{O}(N^{-1}))\frac{N^kd^{k}(d-1)^{k}}{(dN)^{k}}=(1+\wt{O}(N^{-1})) (d-1)^{k}.
\end{align*}
Combining the two, we get that $\EE[|U_{\cH}|]=(1+\wt{O}(N^{-1})) (d-1)^{k}$. The result then follows from  \Cref{lem:longcycles}.
\end{proof}

\begin{lemma}\label{lem:edgeintersection}
 Let $\cH$ be a random $d$-regular graph on $N$ vertices. Let $k$ be a subpolynomial function of $N$ with $k\geq 10\log_{d-1}N$. For every $1\leq i\leq k$, let $U_{H,i}$  be the set of pairs of $k$-cycles in $H$ that intersect in exactly $i$ edges. Then
\[
\E\left[\sum_{i=1}^k p^{2k-i}|U_{H,i}|\right]=\wt O(N^{-1}p^{2k}\E[|U_{\cH}|]^2).
\]
\end{lemma}
\begin{proof}
For $i=1,\dots,k-1$, if two $k$-cycles intersect in $i$ edges, then their intersection must be a union of $s$ disjoint paths for some $1\leq s\leq i$; thus, their union occupies $2k-(i+s)$ vertices and $2k-i$ edges. For every $s$, there are $\binom{i-1}{s-1}$ ways to partition $i$ edges into $s$ paths and $\binom{k-i-1}{s-1}$ ways to partition $k-i$ edges into $s$ paths. Thus there are $\binom{i-1}{s-1}\binom{k-i-1}{s-1}^2$ possible underlying graphs of the union of these two cycles. Finally, given the underlying graph, there are $k^2$ ways to order the vertices in each cycle.  Thus, we have the upper bound
\begin{align*}
    \EE[|U_{\cH,i}|]&\leq (1+o_N(1))\sum_{s=1}^i\binom{i-1}{s-1}\binom{k-i-1}{s-1}^2\frac{N^{2k-(i+s)}d^{2k-i}(d-1)^{2k-i}}{(dN)^{2k-i}}\cdot k^2\\
    &\leq k^2 \sum_{s=1}^i\binom{i-1}{s-1}\binom{k-i-1}{s-1}^2\frac{(d-1)^{2k-i}}{N^s}\leq k^2(d-1)^{2k-i}\sum_{s=1}^i\frac{k^{3s}}{N^s}\\
    &\leq \frac{2k^5}{N}\cdot  (d-1)^{2k-i}.
\end{align*}
Finally, when $i=k$, it is clear that $\EE[|U_{\cH,k}|]=k\EE[|U_{\cH}|]$.
This gives
\begin{align*}
    \EE\left[\sum_{i=1}^{k}p^{2k-i} |U_{\cH,i}|\right]&\leq k\EE[|U_{\cH}|]\cdot p^{k}+\frac{2k^5}{N}\sum_{i=1}^{k-1} (d-1)^{2k-i}\cdot p^{2k-i}\\
    &=k\EE[|U_{\cH}|]\cdot p^{k}+\frac{2k^5}{N}\sum_{i=k+1}^{2k-1} (d-1)^{i}p^{i}\\
    &\leq \frac{(2k)^5}{N}(p^{k} \EE[|U_{\cH}|])^2=\wt O(N^{-1}p^{2k}\E[|U_{\cH}|]^2).
\end{align*}

\end{proof}
This immediately gives strong concentration of $\tr(B_0)$.

\begin{lemma}\label{lem:trace-nbw}
Let $k=c\log N$, with $c\deq \max\left\{10,\frac{3}{\log(p\sqrt{d-1})}\right\}$.
Then with probability $1-o_N(1)$, we have
\[
\tr(B_0^k)=(1+\wt O(N^{-1}))(p(d-1))^k.
\]
\end{lemma}
\begin{proof}
Note that by our choice of $c$, we have
\begin{align*} k\geq 10\log_{d-1}N,\qquad  N^2d&\leq \left(p\sqrt{d-1}\right)^k.
\end{align*}

Recall that $\cG_0$ is obtained from $\cH$ by including every edge independently with probability $p$. Let $U_{\cG_0}$ denote the set of $k$-cycles in $\cG_0$.
With high probability, $\cH$ satisfies \Cref{prop:numberofcycles} and \Cref{lem:edgeintersection}. Under this assumption, we have 
\[
\E[|U_{\cG_0}|]=p^{k}|U_{\cH}|=(1+\wt{O}(N^{-1}))p^k(d-1)^k.
\]
For the variance, we have 
\[
\E[|U_{\cG_0}|^2]-\E[|U_{\cG_0}|]^2\leq \sum_{i=1}^{k}  \EE\left[p^{2k-i} |U_{\cH,i}|\right]=\wt O(N^{-1}p^{2k}(d-1)^{2k}).
\]
Therefore, by a Chebyshev inequality, with high probability we have
\[
|U_{\cG_0}|=(1+\wt{O}(N^{-1}))p^k(d-1)^k.
\]
Finally, observe that $\tr(B_0^k)$ equals the number of length-$k$ nonbacktracking walks in $\cG_0$ (\Cref{def:nbw-2}) that start and end at the same directed edge. 

By \Cref{lem:fried}, with high probability the number of such walks in $\cH$ is
\[
(d-1)^{k}+O(Nd(d-1)^{k/2}),
\]
meaning that the number of such walks that are not $k$-cycles in $\cH$ (and thus,  not $k$-cycles in $\cG_0$) is $O(Nd(d-1)^{k/2})=O(N^{-1}p^k(d-1)^k)$. Thus, we have that with high probability
\[
\tr(B_0^k)=(1+\wt{O}(N^{-1}))p^k(d-1)^k+O(Nd(d-1)^{k/2})=(1+\wt{O}(N^{-1}))p^k(d-1)^k.
\]
\end{proof}

We are now ready to prove  \Cref{thm:first-ev}.
\begin{proof}[Proof of \Cref{thm:first-ev}]
    Take $k=c\log N$ as in \Cref{lem:trace-nbw}. We know   that with high probability $\tr(B_0^k)=(1+\wt{O}(N^{-1}))p^k(d-1)^k$. By \Cref{prop:ihara},  \Cref{lem:perronvector}, and \Cref{lem:eigpreservation}, $B_0$ has an eigenvalue $\theta=p(d-1)+O(N^{-0.05})$. Thus, we know that
    \[
    \tr(B_0^k)-\theta^k=\wt{O}(N^{-1})p^k(d-1)^k+O(kN^{-0.05})p^k(d-1)^k=O(kN^{-0.05})p^k(d-1)^k,
    \]
    meaning that for any other eigenvalue $\theta'\neq \theta$ of $B_0$, we have that for sufficiently large $N$, 
    \begin{align}\label{eq:spectral-gap}
        |\theta'|\leq  (k^{2/k}\se^{-0.05\log N/k})p(d-1)=(1+o_N(1))\se^{-\frac{1}{20c}}p(d-1).
    \end{align}
Note that $c>0$ is an absolute constant that only depends on $d$ and $\mu$.
Using \Cref{prop:ihara} and \Cref{lem:eigpreservation} again, we get that $\lambda_1(\cG)=(1+O(N^{-0.05}))\mu$, 
and there exists $\delta>0$ such that  $\lambda_i(\cG)\leq (1-\delta)\mu$ for all $i>1$.
\end{proof}

\section{Green's function preliminaries}\label{sec:second-ev}

\subsection{Overview}

In this section, we begin our proof of \Cref{thm:second-ev} through the argument of Huang and Yau \cite{huang2024spectrum}. The following is a simplified version of the parameters used in this previous work.
\begin{dfn}[Choices of parameters]\label{dfn:cop}
We fix $\fc\deq 0.01$. Let $\fR=(\fc/4)\log_{d-1}N$,  $r=(\fc/32)\log_{d-1}N$ and $\ell\in[12\log_{d-1}\log N, 24\log_{d-1}\log N]$.  
With this  choice, 
\begin{align}\label{e:relation1}
\fR/8=r=O(\log N)\gg \ell=\Theta(\log\log N), \quad 
(\log N)^{12}\leq (d-1)^\ell\leq (\log N)^{24}.
\end{align}
\end{dfn} 

For real and imaginary parts $E$ and $\eta$, we restrict ourselves to the domain $z=E(z)+\ri \eta(z)\in \C^+$, with $(\log N)^{300}/N\leq \eta\leq 2d$ and $|E|\leq 2d$. We define $\kappa(z):=\min\{|E(z)-2|, |E(z)+2|\}$ to be the distance from $z$ to the spectral edges $\pm 2$. 
By \eqref{eq:msc} and \eqref{e:md_equation} (see also \cite[Lemma 6.2]{erdHos2017dynamical}) we have
\begin{align}
\label{eq:dynamic}
\begin{split}
\Im[m_{d}(z)]\asymp\Im[m_{sc}(z)]\asymp 
\left\{
\begin{array}{cc}
\sqrt{\kappa(z)+\eta(z)} & \text{ if } |E(z)|\leq 2,\\
\eta(z)/\sqrt{\kappa(z)+\eta(z)} & \text{ if } |E(z)|\geq 2.
\end{array}
\right.
\end{split}
\end{align}
We now fix our error parameter. For $z\in\C^+$,  let
\begin{align}\label{e:defeps0}
\varepsilon_0(z)\deq 
(\log N)^{96}\left(\frac{1}{(d-1)^r}+\sqrt{\frac{\Im[m_d(z)]}{N\eta(z)}}+\frac{1}{(N\eta(z))^{2/3}}\right),
\end{align}
and
\begin{align}
\label{e:defeps}
\varepsilon(z):=\begin{cases}
    \varepsilon_0(z) &\text{ if } \varepsilon_0(z)\leq \frac{\kappa(z)+\eta(z)}{\log N},\\
    (\log N)^4\varepsilon_0(z) &\text{ if } \varepsilon_0(z)> \frac{\kappa(z)+\eta(z)}{\log N}.
\end{cases}
\end{align}\index{$\varepsilon_0(z)$}\index{$\varepsilon(z)$}

The error $\sqrt{\Im[m_d(z)]/(N\eta(z))}$ in \eqref{e:defeps0} is a common error appearing in random matrix theory;
the error $1/(d-1)^r$ is to take into account the radius of approximation we will use. The split definition in \eqref{e:defeps} guarantees that we always achieve an improvement to the error during a bootstrapping argument. 

We further let
\begin{align}\label{e:defphi}
\varepsilon'(z):=(\log N)^3\varepsilon(z),\quad
\varphi(z):=(\log N)^{ 24}\sqrt{\frac{\Im[m_d(z)]+\varepsilon'(z)+\varepsilon(z)/\sqrt{\kappa(z)+\eta(z)+\varepsilon(z)}}{N\eta(z)}}.
\end{align}
Notice that from our choice of parameters, on the spectral domain $z\in \C^+$ with $\Im[z]\geq (\log N)^{300}/N$, we have 
\begin{align}\label{e:relation2}
\varepsilon(z)\leq \frac{1}{(\log N)^{48}},\quad \varphi (z) \leq \frac{\varepsilon(z)}{(\log N)^{48}}.
\end{align}

Our first step will be to approximate entries of the Green's function with their radius $r$ neighborhoods. It will be convenient to do this not for the original graph, but for a ``finitized'' version (see \Cref{fig:3}), where all adjacency from $V_0$ to $V_1$ is replaced with a weighted loop.

\begin{figure}
    \centering
\scalebox{0.5}{
\begin{tikzpicture}
    \node[circle, fill=black, draw, inner sep=0pt,minimum size=3pt] at (0,2) {};
    \node[circle, fill=black, draw, inner sep=0pt,minimum size=3pt] at (2,0) {};
    \node[circle, fill=black, draw, inner sep=0pt,minimum size=3pt] at (0,-2) {};
    \node[circle, fill=black, draw, inner sep=0pt,minimum size=3pt] at (-2,0) {};
    \node[circle, fill=black, draw, inner sep=0pt,minimum size=3pt] at (1.41,1.41) {};
    \node[circle, fill=black, draw, inner sep=0pt,minimum size=3pt] at (1.41,-1.41) {};
    \node[circle, fill=black, draw, inner sep=0pt,minimum size=3pt] at (-1.41,1.41) {};
    \node[circle, fill=black, draw, inner sep=0pt,minimum size=3pt] at (-1.41,-1.41) {};
    \node[circle, fill=black, draw, inner sep=0pt,minimum size=3pt] at (3,0) {};
    \draw(2,0)--(3,0);
    \node[circle, fill=black, draw, inner sep=0pt,minimum size=3pt] at (3.7,0.7) {};
    \draw(3.7,0.7)--(3,0);
    \node[circle, fill=black, draw, inner sep=0pt,minimum size=3pt] at (4,1) {};
    \draw(3.7,0.7)--(4,1);
    \node[circle, fill=black, draw, inner sep=0pt,minimum size=3pt] at (4,0.4) {};
    \draw(3.7,0.7)--(4,0.4);
    \node[circle, fill=black, draw, inner sep=0pt,minimum size=3pt] at (3.7,-0.7) {};
    \draw(3.7,-0.7)--(3,0);
    \node[circle, fill=black, draw, inner sep=0pt,minimum size=3pt] at (4,-0.4) {};
    \draw(3.7,-0.7)--(4,-0.4);
    \node[circle, fill=black, draw, inner sep=0pt,minimum size=3pt] at (4,-1) {};
    \draw(3.7,-0.7)--(4,-1);
    
    \node[circle, fill=black, draw, inner sep=0pt,minimum size=3pt] at (2.41,1.41) {};
    \node[circle, fill=black, draw, inner sep=0pt,minimum size=3pt] at (2.41,1.41+0.3*1.41) {};
    \draw(2.41,1.41+0.3*1.41)--(2.41,1.41);
    \node[circle, fill=black, draw, inner sep=0pt,minimum size=3pt] at (2.41+0.3*1.41,1.41) {};
    \draw(2.41+0.3*1.41,1.41)--(2.41,1.41);
    \node[circle, fill=black, draw, inner sep=0pt,minimum size=3pt] at (1.41,2.41) {};
    \draw(2.41,1.41)--(1.41,1.41);
    \node[circle, fill=black, draw, inner sep=0pt,minimum size=3pt] at (1.41,2.41+0.3*1.41) {};
    \draw(1.41,2.41+0.3*1.41)--(1.41,2.41);
    \node[circle, fill=black, draw, inner sep=0pt,minimum size=3pt] at (1.41+0.3*1.41,2.41) {};
    \draw(1.41+0.3*1.41,2.41)--(1.41,2.41);
    \draw(1.41,2.41)--(1.41,1.41);
    \node[circle, fill=black, draw, inner sep=0pt,minimum size=3pt] at (0.7,2.7) {};
    \draw(0,2)--(0.7,2.7);
    \draw(0,2)--(-0.7,2.7);
    \node[circle, fill=black, draw, inner sep=0pt,minimum size=3pt] at (1,3) {};
    \draw(1,3)--(0.7,2.7);
    \node[circle, fill=black, draw, inner sep=0pt,minimum size=3pt] at (0.4,3) {};
    \draw(0.4,3)--(0.7,2.7);
    \node[circle, fill=black, draw, inner sep=0pt,minimum size=3pt] at (-0.7,2.7) {};
    \draw(-0.7,2.7)--(-1,3);
    \draw(-0.7,2.7)--(-0.4,3);
    \node[circle, fill=black, draw, inner sep=0pt,minimum size=3pt] at (-0.4,3) {};
    \node[circle, fill=black, draw, inner sep=0pt,minimum size=3pt] at (-1,3) {};

    \node[circle, fill=black, draw, inner sep=0pt,minimum size=3pt] at (0,-3) {};
    \draw(0,-2)--(0,-3);
    \draw(0,-2)--(0.7,-2.7);
    \draw(0,-2)--(-0.7,-2.7);
    \node[circle, fill=black, draw, inner sep=0pt,minimum size=3pt] at (0.3,-3.3) {};
    \node[circle, fill=black, draw, inner sep=0pt,minimum size=3pt] at (-0.3,-3.3) {};
    \draw(-0.3,-3.3)--(0,-3);
    \draw(0.3,-3.3)--(0,-3);
    \node[circle, fill=black, draw, inner sep=0pt,minimum size=3pt] at (0.7,-2.7) {};
    \draw(0.7+0.3*1.41,-2.7)--(0.7,-2.7);
    \draw(0.7,-2.7-0.3*1.41)--(0.7,-2.7);
    \node[circle, fill=black, draw, inner sep=0pt,minimum size=3pt] at (0.7+0.3*1.41,-2.7) {};
    \node[circle, fill=black, draw, inner sep=0pt,minimum size=3pt] at (0.7,-2.7-0.3*1.41) {};
    \node[circle, fill=black, draw, inner sep=0pt,minimum size=3pt] at (-0.7,-2.7) {};
    \draw(-0.7,-2.7-0.3*1.41)--(-0.7,-2.7);
    \draw(-0.7-0.3*1.41,-2.7)--(-0.7,-2.7);
    \node[circle, fill=black, draw, inner sep=0pt,minimum size=3pt] at (-0.7,-2.7-0.3*1.41) {};
    \node[circle, fill=black, draw, inner sep=0pt,minimum size=3pt] at (-0.7-0.3*1.41,-2.7) {};

    \draw(-1.41-0.7,-1.41-0.7)--(-1.41,-1.41);
    \draw(1.41+0.7,-1.41-0.7)--(1.41,-1.41);
    \node[circle, fill=black, draw, inner sep=0pt,minimum size=3pt] at (-1.41-0.7,-1.41-0.7) {};
    \node[circle, fill=black, draw, inner sep=0pt,minimum size=3pt] at (1.41+0.7,-1.41-0.7) {};
    \node[circle, fill=black, draw, inner sep=0pt,minimum size=3pt] at (-2.41-0.7,-1.41-0.7) {};
    \draw(-1.41-0.7,-1.41-0.7)--(-2.41-0.7,-1.41-0.7);
    \node[circle, fill=black, draw, inner sep=0pt,minimum size=3pt] at (-2.41-0.7,-1.41-0.3*1.41-0.7) {};
    \draw(-2.41-0.7,-1.41-0.3*1.41-0.7)--(-2.41-0.7,-1.41-0.7);
    \node[circle, fill=black, draw, inner sep=0pt,minimum size=3pt] at (-2.41-0.3*1.41-0.7,-1.41-0.7) {};
    \draw(-2.41-0.3*1.41-0.7,-1.41-0.7)--(-2.41-0.7,-1.41-0.7);
    \node[circle, fill=black, draw, inner sep=0pt,minimum size=3pt] at (-1.41-0.7,-2.41-0.7) {};
    \node[circle, fill=black, draw, inner sep=0pt,minimum size=3pt] at (-1.41-0.7,-2.41-0.3*1.41-0.7) {};
    \draw(-1.41-0.7,-2.41-0.3*1.41-0.7)--(-1.41-0.7,-2.41-0.7);
    \node[circle, fill=black, draw, inner sep=0pt,minimum size=3pt] at (-1.41-0.7-0.3*1.41,-2.41-0.7) {};
    \draw(-1.41-0.3*1.41-0.7,-2.41-0.7)--(-1.41-0.7,-2.41-0.7);
    \draw(-1.41-0.7,-2.41-0.7)--(-1.41-0.7,-1.41-0.7);

        \node[circle, fill=black, draw, inner sep=0pt,minimum size=3pt] at (2.41+0.7,-1.41-0.7) {};
    \draw(1.41+0.7,-1.41-0.7)--(2.41+0.7,-1.41-0.7);
    \node[circle, fill=black, draw, inner sep=0pt,minimum size=3pt] at (2.41+0.7,-1.41-0.3*1.41-0.7) {};
    \draw(2.41+0.7,-1.41-0.3*1.41-0.7)--(2.41+0.7,-1.41-0.7);
    \node[circle, fill=black, draw, inner sep=0pt,minimum size=3pt] at (2.41+0.3*1.41+0.7,-1.41-0.7) {};
    \draw(2.41+0.3*1.41+0.7,-1.41-0.7)--(2.41+0.7,-1.41-0.7);
    \node[circle, fill=black, draw, inner sep=0pt,minimum size=3pt] at (1.41+0.7,-2.41-0.7) {};
    \node[circle, fill=black, draw, inner sep=0pt,minimum size=3pt] at (1.41+0.7,-2.41-0.3*1.41-0.7) {};
    \draw(1.41+0.7,-2.41-0.3*1.41-0.7)--(1.41+0.7,-2.41-0.7);
    \node[circle, fill=black, draw, inner sep=0pt,minimum size=3pt] at (1.41+0.3*1.41+0.7,-2.41-0.7) {};
    \draw(1.41+0.3*1.41+0.7,-2.41-0.7)--(1.41+0.7,-2.41-0.7);
    \draw(1.41+0.7,-2.41-0.7)--(1.41+0.7,-1.41-0.7);

    \node[circle, fill=black, draw, inner sep=0pt,minimum size=3pt] at (-2.7,-0.7) {};
    \node[circle, fill=black, draw, inner sep=0pt,minimum size=3pt] at (-2.7,0.7) {};
    \draw(-2,0)--(-2.7,-0.7);
    \draw(-2,0)--(-2.7,0.7);
    \node[circle, fill=black, draw, inner sep=0pt,minimum size=3pt] at (-3,1) {};
    \node[circle, fill=black, draw, inner sep=0pt,minimum size=3pt] at (-3,0.4) {};
    \draw(-2.7,0.7)--(-3,1);
    \draw(-2.7,0.7)--(-3,0.4);
    \draw(-2.7,-0.7)--(-3,-1);
    \draw(-2.7,-0.7)--(-3,-0.4);
    \node[circle, fill=black, draw, inner sep=0pt,minimum size=3pt] at (-3,-1) {};
    \node[circle, fill=black, draw, inner sep=0pt,minimum size=3pt] at (-3,-0.4) {};

    \draw(0,2)--(-1.41,-1.41);
    \draw(-1.41,1.41)--(1.41,1.41);
    \draw(2,0)--(1.41,-1.41);
    \draw(2,0)--(-1.41,-1.41);
    \draw(-1.41,1.41)--(1.41,-1.41);
    \draw(-2,0)--(-1.41,1.41);

    \node[] at (5,0) {$\cdots$};
    \node[] at (-4,0) {$\cdots$};
    \node[rotate=90] at (0,4) {$\cdots$};
    \node[rotate=90] at (0,-4) {$\cdots$};
    \node[rotate=45] at (2*1.41,2*1.41) {$\cdots$};
    \node[rotate=135] at (2*1.41+0.7,-2*1.41-0.7) {$\cdots$};
    \node[rotate=45] at (-2*1.41-0.7,-2*1.41-0.7) {$\cdots$};
\end{tikzpicture}
}
\qquad 
\scalebox{0.7}{
\begin{tikzpicture}
    \node[circle, fill=black, draw, inner sep=0pt,minimum size=3pt] at (0,2) {};
    \node[circle, fill=black, draw, inner sep=0pt,minimum size=3pt] at (2,0) {};
    \node[circle, fill=black, draw, inner sep=0pt,minimum size=3pt] at (0,-2) {};
    \node[circle, fill=black, draw, inner sep=0pt,minimum size=3pt] at (-2,0) {};
    \node[circle, fill=black, draw, inner sep=0pt,minimum size=3pt] at (1.41,1.41) {};
    \node[circle, fill=black, draw, inner sep=0pt,minimum size=3pt] at (1.41,-1.41) {};
    \node[circle, fill=black, draw, inner sep=0pt,minimum size=3pt] at (-1.41,1.41) {};
    \node[circle, fill=black, draw, inner sep=0pt,minimum size=3pt] at (-1.41,-1.41) {};
    \draw(0,2)--(-1.41,-1.41);
    \draw (0,2) to [out=60,in=120, loop, style={min distance=12mm}] (0,2) -- (0,2);
    \draw (0,-2) to [out=-60,in=-120, loop, style={min distance=12mm}] (0,-2) -- (0,-2);
    \draw (2,0) to [out=30,in=-30, loop, style={min distance=12mm}] (2,0) -- (2,0);
    \draw (1.41,-1.41) to [out=-15,in=-75, loop, style={min distance=12mm}] (1.41,-1.41) -- (1.41,-1.41);
    \draw (-1.41,-1.41) to [out=-105,in=-165, loop, style={min distance=12mm}] (-1.41,-1.41) -- (-1.41,-1.41);
    \draw (-2,0) to [out=-150,in=150, loop, style={min distance=12mm}] (-2,0) -- (-2,0);
    \draw (1.41,1.41) to [out=15,in=75, loop, style={min distance=12mm}] (1.41,1.41) -- (1.41,1.41);
    \draw(-1.41,1.41)--(1.41,1.41);
    \draw(2,0)--(1.41,-1.41);
    \draw(2,0)--(-1.41,-1.41);
    \draw(-1.41,1.41)--(1.41,-1.41);
    \draw(-2,0)--(-1.41,1.41);
\end{tikzpicture}
}
    \caption{The ``finitization'' process}
    \label{fig:3}
\end{figure}
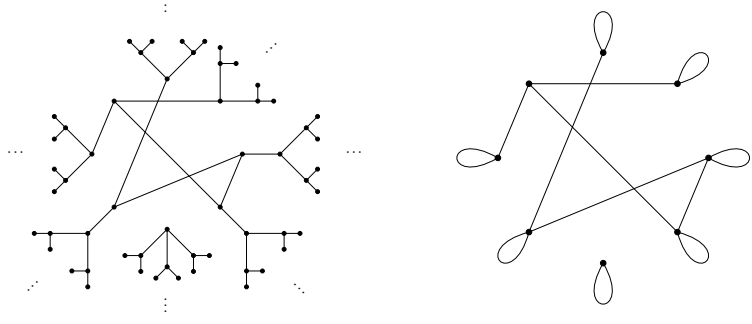

\begin{dfn}\label{dfn:finitization}
    Given an infinite graph $\GG$ sampled from \Cref{def:construction}, we generate a finite graph $\GG_f$ by removing all edges in $E[V_0,V_1]$ (so with one vertex in $V_0$ and one vertex in $V_1$). For each edge we have removed, we add a loop of weight $-m_{sc}/\sqrt{d-1}$ to its end vertex in $V_0$. We then pass to the subgraph induced by $V_0$. 
\end{dfn}
One of the useful facts about Green's functions is that, through the Schur complement formula, the entries of the Green's function on $\GG$ and $\GG_f$ on $V_0$ are the same.

\begin{lemma}[{Reduced version of \cite[Lemma 5.4]{bauerschmidt2019local}}]
 For all $ x,y\in V_0$ and $z\in \C^+$, $G_{xy}(z,\GG)=G_{xy}(z,\GG_f)$.
\end{lemma}\label{lem:finitization}

Typically, when approximating the Green's function by local neighborhoods, we will consider the finite graph $\GG_f$. However, for some calculations it will be convenient to consider the infinite graph $\GG$.

We now wish to come up with a self-consistent equation satisfied by $m_{sc}$ that $Q$ is close to satisfying. For a graph $\cG$ with subgraph $\cH\subset \cG$,
we define the \emph{truncation function} $g_{\cH}(v):V\rightarrow \Z_{\geq 0}$ to be the change in degree after passing from $\cG$ to $\cH$. 
\begin{equation}\label{eq:truncationfunction}
g_{\cH}(v):=\deg_{\GG}(v)-\deg_{\cH}(v).
\end{equation}
We want to show that our Green's function is reasonably approximated after passing to this neighborhood. Therefore, we define the Green's function of the graph after \emph{extending} with weight function $\Delta(z)$ as
\[
G(\Ext(\cH,\Delta(z)),z)\deq\left(-z+H|_{\cH}-\frac{\Delta}{d-1}\sum_{v\in \cH} g_{\cH}(v)e_ve_v^*\right)^{-1}.
\]
Note that the extension contains two types of weighted loops: weights of $-m_{sc}/\sqrt{d-1}$ corresponding to infinite tree extensions, and weights of $-\Delta/\sqrt{d-1}$ corresponding to the boundary of $\cH$.

Before proving sharp proximity of $Q$ with $m_{sc}$, we prove a weaker bound on all entries of the Green's function of $\GG_{f}$.
\begin{dfn}
Recall the definition of $Q$ from \eqref{eq:qdef}, $\GG_f$ from \eqref{dfn:finitization}, and $\vareps$ from \eqref{e:defeps}. For any $z\in \C^+$, we define  $\Omega(z)\subseteq \oOmega$ to be the event that 
\begin{align}\label{eq:omegaeq}
&|G_{ij}(\GG,z)-G_{ij}(\Ext(\cB_r(\{i,j\},\GG_f),Q(\GG,z)),z)|\leq \vareps,\\
&|Q(\GG,z)-m_{sc}(z)|\leq \frac{\vareps }{\sqrt{\kappa+\eta+\vareps}}\label{eq:omegaeq2}
\end{align}
for every $i,j\in [N]$.
\end{dfn}

Our first goal will be to show that, with high probability, $\Omega(z)$ occurs simultaneously for every $z$ in our domain.
\begin{prop}\label{thm:entrywisebound}
    For any constants $\fq>0$, with probability $1-O(N^{-\fq})$ conditioned on $\overline\Omega$, $\GG$ satisfies $\Omega(z)$ for every $z=E+i\eta$ with $\eta\geq \log^{300}N/N$.
\end{prop}

In order to show this and create a self-consistent equation that shows that $Q\approx m_{sc}$, we define $\mathcal{Y}$ to be the infinite $(d-1)$-ary tree with root $o$ and $\mathcal{X}$  to be the infinite $d$-regular tree with root $o$. For any function $\Delta(z):\C^+\rightarrow \C^+$ and $\ell\in \Z_{\geq 0}$, define
\begin{equation}\label{eq:XdfnYdfn}
X_\ell(\Delta(z)):=G_{oo}(\Ext(\cB_\ell(o,\mathcal{X}),\Delta(z)),z),\quad Y_\ell(\Delta(z)):=G_{oo}(\Ext(\cB_\ell(o,\mathcal{Y}),\Delta(z)),z).
\end{equation}

Note that $X_\ell(m_{sc}(z))=m_d(z)$ and $Y_\ell(m_{sc}(z))=m_{sc}(z)$. The next proposition gives the Taylor expansion around these fixed points. These follow from the resolvent identity \eqref{e:resolv} and the Schur complement formula \eqref{e:Schur1}. Here and throughout, we omit the dependence on $z$ when it is clear. 
\begin{prop}[{\cite[Proposition 2.10]{huang2024spectrum}}]\label{prop:XYexpansion}
    If $\ell|\Delta-m_{sc}|\ll 1$, then 
    \begin{align}
X_\ell(\Delta,z)-m_d&=\frac{d}{d-1}m_d^2m_{sc}^{2\ell}(\Delta-m_{sc})+O(\ell|\Delta-m_{sc}|^2),\label{eq:xexpansion}\\
Y_\ell(\Delta,z)-m_{sc}&=m_{sc}^{2\ell+2}(\Delta-m_{sc})+m_{sc}^{2\ell+3}\left(\frac{1-m_{sc}^{2\ell+2}}{1-m_{sc}^2}\right)(\Delta-m_{sc})^2+O(\ell^2|\Delta-m_{sc}|^3).\label{eq:yexpansion}
    \end{align}
\end{prop}

More generally, for any extension of the $(d-1)$-ary tree containing the root, the value of the root is $m_{sc}$. Therefore we can do a Taylor expansion around $m_{sc}$ to produce the following. 
\begin{prop}[{\cite[Proposition 2.11]{huang2024spectrum}}]\label{prop:QvsExt}
    For $\ell>0$, consider a (potentially infinite) graph $\GG$ which is $d$-regular except at a distinguished vertex $o$ which has degree $d-1$, with tree neighborhood of radius at least $\ell+1$. Then if $\ell|\Delta-m_{sc}|\ll1$, 
    \[
    G_{oo}\left(\Ext (\cB_\ell(o,\GG),\Delta),z\right)-\Delta=O\left(\ell(\sqrt{\kappa+\eta}|\Delta-m_{sc}|+|\Delta-m_{sc}|^2)\right).
    \]
\end{prop}

Our goal will be to show that $Q$ is close to satisfying a self-consistent equation that is satisfied by $m_{sc}$, and therefore is close to $m_{sc}$ itself. This self-consistent equation is based on the idea that given a vertex pair $(o,i)\in \vec{E}_0$, we can use a Schur complement formula to write $G_{oo}^{(i)}$ in terms of the Green's function of vertices at distance $\ell+1$ from $o$ in $\GG^{(i)}$. An important difference between our model and the random regular graph is that we do not expect the contribution of vertices on the boundary of $\mathcal B_\ell(o,\GG^{(i)})$ to be uniform. Rather, the goal is to show that the contribution of vertices on the boundary in $V_0$ is approximately $Q$, whereas we know that vertices from the infinite trees always give a contribution of $m_{sc}$. In the random regular graph, there are no contributions of the second type.

It will be convenient to consider the expected contribution of the boundary at distance $\ell+1$. For $o\sim i\in V_0$, assume that $o$ has tree neighborhood of radius $\ell+1$ in $\GG^{(i)}$ (we know this will be true for almost all vertices by \Cref{lem:tanglefree}). Out of the $(d-1)^{\ell+1}$ vertices on the boundary of $\cB_{\ell+1}(o,\GG^{(i)})$, we expect there to be $p^{\ell+1}(d-1)^{\ell+1}$ vertices in $V_0$, as the path from $o$ to such a vertex must be made completely of vertices in $V_0$. Thus we expect $(1-p^{\ell+1})(d-1)^{\ell+1}$ to be in $V\setminus V_{0}$. Therefore, given the function $\Delta(z)$, we define
\begin{equation}\label{eq:Qldef}
    \Delta_\ell(z)\deq p^{\ell+1}\Delta(z)+(1-p^{\ell+1})m_{sc}(z)
\end{equation}
to signify the average contribution of the Green's function if vertices in $V_0$ contribute $\Delta$ and vertices in $V\setminus V_{0}$ contribute $m_{sc}$. Our goal will be to show that $Q-Y_\ell(Q_\ell)\approx 0$.

\subsection{Switching dynamics}\label{sec:switching}

In this section, we recall the local resampling and its properties as introduced in \cite{huang2024optimal,huang2024spectrum}.
Recall that for graph $\cG=(V,E)$, we use $\vec{E}$ to denote the set of directed edges in $\cG$ (\Cref{def:directed-edges}).
For a subset $\vec{S}\subseteq \vec{E}(\cG)$, we denote by $S$ the set of corresponding non-oriented edges.

\begin{dfn}
A (simple) switching is encoded by two oriented edges $\vec S=\{(v_1, v_2), (v_3, v_4)\} \subseteq \vec{E}$.
Then the switching consists of
replacing the edges $\{v_1,v_2\}, \{v_3,v_4\}$ by the edges $\{v_1,v_4\},\{v_2,v_3\}$.
We denote the graph after the switching by $T_{\vec S}(\cG)$.
\end{dfn}

The local resampling is a sequence of switchings of the edges in $\vec{E}_0$. We fix a vertex $o\in V_0$ to be the center of the resampling. We denote by $\cT\deq \cB_\ell(o,\GG_f)$ the ball of radius $\ell$ surrounding $o$ in $\GG$ and $\bT$ to be the vertex set of $\cT$. 

We enumerate the edges of $\partial_E \cT$ as $ \partial_E \cT = \{e_1,e_2,\dots, e_\mu\}$, where $e_\alpha=\{l_\alpha, a_\alpha\}$ with $l_\alpha \in \bT$ and $a_\alpha \in V_0\backslash \bT$. Thus $\mu$ is the number of vertices at distance $\ell+1$ to $o$ in $V_0$. We orient the edges $e_\alpha$ by defining $\vec{e}_\alpha=(l_\alpha, a_\alpha)$.
We note that $\mu$ and the edges $e_1,e_2, \dots, e_\mu$ depend on $\cG$. The edges $e_\alpha$ are distinct, but
the vertices $a_\alpha$ are not necessarily distinct, nor are the vertices $l_\alpha$. 
To perform our local resampling, we choose $(b_1,c_1), \dots, (b_\mu,c_\mu)$ to be independently and uniformly chosen oriented edges from the random edges in $\cG_0$ within $(V_0\backslash \bT)\times (V_0\backslash \bT)$, and we define
\begin{equation}\label{e:defSa}
  \vec{S}_\alpha= \{\vec{e}_\alpha, (b_\alpha,c_\alpha)\},
  \qquad
  {\mathbf S}=(\vec S_1, \vec S_2,\dots, \vec S_\mu).
\end{equation}
The sets $\mathbf S$ will be called the \emph{resampling data} for $\cG$. We remark that repetitions are allowed in the resampling data $(b_1, c_1), (b_2, c_2),\cdots, (b_\mu, c_\mu)$.
We define an indicator that will be crucial to the definition of switch.

\begin{dfn}\label{def:w}Recall the definition of $\GG^{(\bV)}$ in \Cref{subsec:green}.
For $\alpha\in\qq{\mu}$,
we define 
$I_\alpha = I_\alpha(\cG,{\mathbf S})$
 to be the indicator function that
 \begin{enumerate}
\item
 the subgraph $\cB_{\fR/4}(\{a_\alpha, b_\alpha, c_\alpha\}, \cG^{(\bT)})$ after adding the edge $\{a_\alpha, b_\alpha\}$ is a tree;
\item 
$\dist_{\cG^{(\bT)}}(\{a_\alpha,b_\alpha,c_\alpha\}, \{a_\beta,b_\beta,c_\beta\})> {\fR/4}$ for all $\beta\in \qq{\mu}\setminus \{\alpha\}$.
\end{enumerate}
\end{dfn}

We define the set
\begin{align}\label{Wdef}
\mathsf W={\mathsf W}_{\mathbf S}\deq \{\alpha\in \qq{\mu}: I_\alpha(\cG,{\mathbf S}) \}.
\end{align}
This will be the set of indices on which we perform the switch. We denote the set $\bW=\bW_{\mathbf S}:=\{b_\alpha:\alpha\in {\mathsf W}_{\mathbf S}\}$\index{$\bW_{\mathbf S}$}. 
Let $\nu\deq |{\mathsf W}|$ be the number of admissible switchings and $\alpha_1,\alpha_2,\dots, \alpha_{\nu}$
be an arbitrary enumeration of ${\mathsf W}$.
Then we define the switched graph by
\begin{equation} \label{e:Tdef1}
T_{\mathbf S}(\cG) \deq  \left(T_{\vec S_{\alpha_1}}\circ \cdots \circ T_{\vec S_{\alpha_\nu}}\right)(\cG).
\end{equation}

When the switch is clear from the context we will write $\wt\GG\deq T_\bfS(\GG)$. More generally, the tilde accent will denote the function applied to $\wt \GG$. For example, $\wt G(z)\deq G(\wt \GG,z), \wt{Q}\deq Q(\wt \GG,z),$ and $\tilde a_\alpha\in \{a_\alpha, c_\alpha\}$ is the neighbor of $l_\alpha$ after the switch.

To make the structure more clear, we introduce an enlarged probability space.
Equivalent to the definition in \eqref{e:defSa}, the sets $\vec{S}_\alpha$  are uniformly distributed over 
\begin{align*}
{\sf S}_{\alpha}(\cG):=\{\vec S\subseteq \vec{E}: \vec S=\{\vec e_\alpha, \vec e\},\, \text{$\vec{e}$ is not incident to $\cT$}\},
\end{align*}
i.e., the set of pairs of oriented edges in $\vec{E}$ containing $\vec{e}_\alpha$ and another oriented edge in $\cG^{(\bT)}$.
Therefore ${\mathbf S}=(\vec S_1,\vec S_2,\dots, \vec S_\mu)$ is uniformly distributed over the set
${\sf S}(\cG):=\sf S_1(\cG)\times \cdots \times \sf S_\mu(\cG)$.

We introduce the following notations on the probability and expectation with respect to the randomness of the $\bfS\in \sf S(\cG)$.
\begin{dfn}\label{def:PS}
    Given any graph $\cG$, we 
    denote $\bP_\bfS(\cdot)$ the uniform probability measure on ${\sf S}(\cG)$;
 and $\E_\bfS[\cdot]$ the expectation  over the choice of $\bfS$ according to $\bP_\bfS$. 
\end{dfn}

Key to our analysis is that this random swapping of edges does not change the distribution. This is an identical result to \cite[Proposition 3.3 and Lemma 7.3]{huang2024spectrum}, but we reproduce the proof in \Cref{sec:exchangeproof} as the probability space of the infinite graph is different. 
\begin{lemma}\label{lem:exchangeablepair}
Fix $d\geq 3$. We recall the operator $T_\bfS$ from \eqref{e:Tdef1}. Let $\cG$ be a random $d$-regular graph  and $\bfS$ uniformly distributed over $\sf S(\cG)$, then the graph pair $(\cG, T_{\mathbf S}(\cG))$ forms an exchangeable pair:
\begin{align*}
(\cG, T_{\mathbf S}(\cG))\stackrel{law}{=}(T_{\mathbf S}(\cG), \cG).
\end{align*}
\end{lemma}
In the following lemma we show that typically $\sf W$ is very large. 
\begin{lemma}\label{lem:configuration}
Assuming $\cG\in \oOmega(z)$, with probability $1-N^{-\omega_N(1)}$, $\mu-|\mathsf{W}|\leq \log N$, and if $o$ has tree neighborhood of radius $\fR$, probability $1-O(N^{-1+2\fc})$, $|\sf W|=\mu$. 
\end{lemma}

\begin{proof}
We only prove the first statement, as the second is analogous. By the assumption that $\GG$ is $\fR$-tangle free, 
the $\fR$-neighborhood of $o$ has at most one cycle. Moreover, there cannot be a cycle that is contained in $\cB_{\fR/4}(a_\alpha,\GG^{(\bT)})$ for two different $\alpha's$, as that will create two cycles in the $\fR$-neighborhood of $o$, one in $\GG^{(\bT)}$ and another in $\GG$ containing $o$. Thus there is at most one $a_\alpha$ such that $\cB_{\fR/4}(a_\alpha,\GG^{(\bT)})$ is not a tree. 

Suppose $a_\alpha$ has a tree neighborhood of radius $\fR/4$. By \Cref{lem:tanglefree}, there are at most $N^{\fc}$ choices of $b_\alpha$ that do not have tree neighborhood of radius $\fR$. Assuming we have not made such a choice, all conditions of $I_\alpha$ are satisfied if $\dist_{\cG^{(\bT)}}(b_{\alpha},a_\alpha\cup \{b_\beta\}_{\beta\neq \alpha})\geq \fR/4+1$. The number of vertices $v$ such that $\dist_{\cG^{(\bT)}}(v,a_\alpha\cup \{ b_\beta\}_{\beta\neq \alpha})\leq \fR/4$ is at most $d(d-1)^{\ell}d(d-1)^{\frac{\fc}4\log N}\leq N^{\fc/2}$. As the choice of each $b_\alpha$ is independent, the probability that we choose at least $\log N-1$ many bad $b_\alpha$ is at most $\binom{\mu}{\log N-1}N^{-\log N(1-\fc)}\leq (d-1)^{\log N\ell}N^{-\log N(1-\fc)}=N^{-\omega_N(1)}$.
\end{proof}

\subsection{Improved error on the Green's function}
To show \Cref{thm:entrywisebound}, we will bootstrap by showing that being in $\Omega$ implies that with high probability, we can improve our error. This will allow us to reason about the statement for different $z$ by using that entries of the Green's function are $\frac1{\eta}$-Lipschitz.

\begin{dfn}
Define  $\Omega_o^{-}(\GG,z)\subseteq \Omega(z)$ to be the event that the randomly selected graph $\GG$ satisfies the stronger bounds
\begin{align}\label{eq:firsthalving}
&|Q(\GG,z)-m_{sc}(z)|\leq \frac{\vareps (z)}{\log N\sqrt{\kappa+\eta+\vareps}},\\
&\label{eq:secondhalving}|G_{oi}(\GG,z)-G_{oi}(\Ext(\cB_r(\{o,i\},\GG_f),Q(\GG,z)),z)|\leq \frac{\vareps}{(\log N)^2}\qquad \text{ for all } i\in [N].
\end{align}

\end{dfn}

The bootstrapping follows by the following proposition.

\begin{prop}\label{prop:minusimprovement}
    Assume that $\GG\in \Omega$. We then perform the local resampling at depth $\ell$ around vertex $o$. For any fixed $\fq>0$, with probability $1-O(N^{-\fq})$, we have that the resampled graph $\wt{\GG}\in \Omega_{o}^-(z)$.
\end{prop}
We will show \Cref{prop:minusimprovement} in the next section. For now we discuss how it implies \Cref{thm:entrywisebound}.
As the distribution of the random regular graph is invariant under edge swaps, taking a union bound over all vertices immediately implies the following.
\begin{cor}\label{cor:mainsecond}
    Under the assumptions of \Cref{prop:minusimprovement}, with probability $1-O(N^{-\fq+1})$,
    \begin{align*}
&|Q(\GG,z)-m_{sc}(z)|\leq \frac{\vareps (z)}{\log N\sqrt{\kappa+\eta+\vareps}},\\
&|G_{ij}(\GG,z)-G_{ij}(\Ext(\cB_r(\{i,j\},\GG_f),Q(\GG,z)),z)|\leq \frac{\vareps}{(\log N)^2}\qquad\text{for all } i,j\in [N].
\end{align*}
\end{cor}

 \Cref{cor:mainsecond} implies \Cref{thm:entrywisebound} through the argument of \cite[Proposition 4.9]{huang2024spectrum}. For an idea of the proof, we start with for $|z|\geq 2d$  for which \Cref{thm:entrywisebound} is true through the Combes--Thomas method \cite{combes1973asymptotic}. We then then bootstrap by applying \Cref{cor:mainsecond}. Entries of the Green's function are Lipschitz with constant $1/\Im[z]$. Therefore, we know that $\Omega(z')$ holds for any $|z'-z|\leq \Im[z]/2$. We then can apply \Cref{cor:mainsecond} for $z'=E+\ri\eta/2$. We repeat this until we reach $\eta\approx \log^{300}N/N$. We apply this process to create a net of $z$ where \Cref{cor:mainsecond} is satisfied, giving $\Omega(z)$ for the entire region.

\subsection{Initial bounds}

Our starting point to show \Cref{prop:minusimprovement} will be the combinatorial bounds of Section 5 of \cite{huang2024spectrum}. Fundamental to these ideas is that of the Ward identity. 
\begin{lemma}[Ward identity for $V_0$]\label{lemma:infward}
For any vertex $i$ and truncation function $f(\cdot)\deq f_{\cG_0}(\cdot)$ as in \eqref{eq:truncationfunction}, we have 
\begin{equation}\label{eq:Ward}
\Im[G_{ii}]=\eta\sum_{j\in V_0}|G_{ij}|^2+\frac{\Im[m_{sc}]}{d-1}\sum_{j\in V_0}f(j) |G_{ij}|^2.
\end{equation}
\end{lemma}
\begin{proof}
Let $H_0$ be the normalized adjacency matrix of $\GG_0$. By the resolvent identity \eqref{e:resolv}, we have
\begin{align}
\begin{split}\label{eq:wardproof}
\Im[G_{ii}]&=\frac{1}{2\ri}\left(\frac1{-z+H_0-\frac{m_{sc}}{d-1}\sum_{j}f(j) e_je_j^*}-\left(\frac1{-z+H_0-\frac{m_{sc}}{d-1}\sum_{j}f(j) e_je_j^*}\right)^*\right)_{ii}\\
&=\left(G\left(\Im [z]+\frac{\Im[m_{sc}]}{d-1}\sum_{j}f(j) e_je_j^*\right)G^*\right)_{ii}\\
&=\eta\sum_{j\in V_0}|G_{ij}|^2+\frac{\Im[m_{sc}]}{d-1}\sum_{j\in V_0}f(j)|G_{ij}|^2.
\end{split}
\end{align}
\end{proof}

In particular, we have
\begin{equation}\label{eq:wardeasy}
\frac1{N}\sum_{j\in V_0}|G_{ij}|^2\leq \frac{\Im[G_{ii}]}{N\eta}.
\end{equation}
Note that this is stronger than the bound given for finite graphs, in which case \eqref{eq:wardeasy} is an equality. 

There is a series of propositions in  \cite[Section 5]{huang2024spectrum} that bound the change of Green's function under adding or removing vertices. These bounds all come from the following  properties of random regular graphs:
\begin{enumerate}
    \item 
    Distance bound: randomly selected vertices are typically far from each other in graph distance. This expresses itself, for example, through \Cref{lem:tanglefree}.
    \item 
    Green's function bound: by the Ward identity (\Cref{lemma:infward}), most pairs of vertices have Green's function entry at most $\varphi$, for $\varphi$ as defined in \eqref{e:defphi}.
\end{enumerate}

In our model, both of these statements have only improved, as distance has only increased, and using the Ward identity to bound the Green's function now uses an inequality rather than an equality. Specifically, by the definition of $\Omega(z)$, \Cref{prop:QvsExt}, and \eqref{eq:wardeasy}, we have
\[
\frac1{N}\sum_{j\in V_0}|G_{ij}|^2\leq \frac{\Im[G_{ii}]}{N\eta}\lesssim  \frac{\Im[m_d(z)]+\vareps'+\vareps/\sqrt{\kappa+\eta+\vareps}}{N\eta}.\]
This means that we can apply the preliminary bounds in \cite{huang2024spectrum} stated for the random regular graph to our model. In the rest of the current section, we introduce the preliminary results we can immediately apply. In \Cref{sec:newgf} we state and prove statements that are different in our setting.

In this section, we assume $\Omega(z)$ and perform the switch as described in \Cref{sec:switching} around a fixed vertex $o\in V_0$. We define $\cT=\cB_\ell(o, \cG)$, $T_{\mathbf S}$ to be the local resampling  with resampling data ${\mathbf S}=\{(l_\alpha, a_\alpha), (b_\alpha, c_\alpha)\}_{\alpha\in\qq{\mu}}$ around  $o$, and $\widetilde \GG\deq T_{\mathbf S}\GG$. We recall the sets $\bT, {\sf W}, \bW$ associated with the switching. The set of comparisons is that of   \cite[Remark 5.19]{huang2024spectrum}, which is 
\[
G\rightarrow G^{(\bT)}\rightarrow G^{(\bT \bW)}=\wt{G}^{(\bT \bW)}\rightarrow \wt{G}^{(\bT)}\rightarrow \wt{G}.
\]

The first statement we use is that deleting the vertex set $\bT$ does not affect the approximation given by $\Omega$ (that is, \eqref{eq:omegaeq}) too strongly. This result follows from expanding $G_{ij}^{(\bT)}$ with the Schur complement formula, and then approximating all relevant terms. Because most of the vertices at the boundary of $\bT$ are far from each other after deleting $\bT$, the contribution of these in the Schur complement expansion is minimal. 

\begin{lemma}[{\cite[Proposition 5.1]{huang2024spectrum}}]\label{lem:removeT}
    For any $\GG\in \Omega$ and $i,j\in [N]\backslash \bT$,
    \[
    |G_{ij}^{(\bT)}-G_{ij}(\Ext(\cB_r(\{i,j\},\GG^{(\bT)}_f),Q),z)|\leq \log^3N\varepsilon.
    \]
\end{lemma}

The next result is that for most vertices involved in the switch, most relevant entries of the Green's function are not significantly larger than what we would expect from the Ward identity. To track this, we use the following definition.
\begin{dfn}\label{dfn:udef}
The set $\sf U\subseteq \sf W$ is the set of indices $\alpha$ such that there is no $\beta\in \sf W$, $\beta\neq \alpha$ where 
\begin{equation}\label{eq:cellcondition1}
\dist_\GG(b_\alpha, b_\beta)\leq \fR/4
\end{equation}
or 
\begin{equation}\label{eq:cellcondition2}
|G_{b_\alpha b_\beta}|\geq \varphi,
\end{equation}
where $\fR$ and $\varphi$ are defined in \Cref{dfn:cop}.
\end{dfn}

We can use the combinatorial distance in the graph and the Ward identity to show that with high probability almost all indices are in $\sf U$, and that the Green's function associated with these vertices is well-behaved. 
\begin{prop}[{\cite[Proposition 5.10, Proposition 5.18 (5.51-5.53) and Proposition 6.3]{huang2024spectrum}}] \label{prop:stabilitytGT}
 Let $\GG \in \Omega$. For any fixed $\fq>0$, with probability at least $1-O(N^{-\fq})$ over the choice of resampling data $\bfS$,  the following holds.

\begin{enumerate}
    \item $|\sf W\setminus \sf U|\lesssim \log \textrm{N}$, with ${\sf U}$ defined in \Cref{dfn:udef}.
    \item for any fixed $\alpha\in \sf W\backslash \sf U$, at most $O_{\fq}(1)$ indices $\beta\in \sf W$ satisfy \eqref{eq:cellcondition1} or \eqref{eq:cellcondition2}. 
    \item We have
    \begin{align}
\label{e:greendistIJ}
|\wt G^{(\bT)}_{ic_\alpha}(z)| &\lesssim (\varepsilon'(z))^2 +\varphi(z), &&\quad\text{if $i=c_\beta$ for some $\beta\in [1,\mu]\setminus \sf U$},
\\
\label{e:greendistJJ}
|\wt G^{(\bT)}_{ic_\alpha}(z)| &\lesssim (\varepsilon'(z))^3 +\varphi(z), &&\quad \text{if $i=c_\beta$ for some $\beta\in \sf U\setminus \{\alpha\}$},
\\
\label{e:greendistNJ}
|\wt G^{(\bT)}_{ic_\alpha}(z)| &\lesssim (\varepsilon'(z))^2 +\varphi(z), &&\quad  \text{if $|G_{ic_\alpha}|,|G_{ib_\alpha}|\leq\varphi$ and $\dist_{\tilde \GG^{(\bT)}}(i,a_\alpha,b_\alpha,c_\alpha)\geq { \fR/4}$}.
\end{align}
\item  For any index $\alpha\in \sf U$, we have
    \be\label{eq:Ubound}
    |\wt G_{c_\alpha c_\alpha}^{(\bT)}-G^{(\bT b_\alpha)}_{c_\alpha c_\alpha}|\lesssim (d-1)^\ell \varphi^2.
    \ee
\end{enumerate}

\end{prop}

Because $Q$ is an average over all vertices, the difference of $Q$ between the original and switched graphs is smaller. We obtain a bound from decomposing $Q$ into its constituent parts, then using a Ward identity \eqref{eq:wardeasy} where $i$ is each vertex in the switch. 
\begin{lemma}[{\cite[Proposition 5.22]{huang2024spectrum}}]\label{lem:Qlemma}
With probability $1-O(N^{-\fq})$, we have
\begin{align}\label{eq:Qlemma}
    |Q(\GG,z)-Q(\wt \GG,z)|
\lesssim \frac{(d-1)^{2\ell} \left(\Im[m_d]+\vareps'+\vareps/\sqrt{\kappa+\eta+\vareps}\right)}{N\eta}+N^{-1+\fc}.
\end{align}
\end{lemma}

The last result that follows immediately from the previous work is an updated error bound for $\wt G_{oi}$, which, compared to   \eqref{eq:omegaeq}, is improved apart from a new error term $|\wt Q-m_{sc}|$. This is because the first order dependence of $\wt G_{oi}$ is a weighted average  $\sum_{\alpha\in\mu}w_\alpha G_{\wt a_\alpha \wt a_\alpha}^{(\bT)}$ that corresponds to walks that go from $o$ to $\widetilde a_\alpha$ and back to $i$, where $\sum_\alpha |w_\alpha|^2\leq (d-1)^{-(\ell-\ell_i)}$ with $\ell_i=\dist_{\GG}(o,i)$. For all $\alpha\in \sf U$, these terms are independent, and we can use an Azuma inequality giving that $\wt{G}_{oi}$ concentrates well, with error bound $\sqrt{\sum_{\alpha}|w_\alpha|^2}$. 
\begin{prop}[{\cite[Proposition 6.4 (6.17-6.18)]{huang2024spectrum}}]\label{prop:improvedswitch} Define $\wt P\deq G(\Ext(\cB_r(o,\wt\GG_f),Q),z)$ and $\ell_i=\dist_{\GG}(o,i)$.  Under the assumption that $\GG\in \Omega$, with high probability,
\begin{equation}\label{eq:improveddiag}
|\wt G_{oi}-\wt P_{oi}|\lesssim \left(\frac{\log N}{(d-1)^{\ell/2}}+\frac{(\log N)^2}{(d-1)^{\ell-\ell_i/2}}
\right)\vareps'
+\log N(\sqrt{\kappa+\eta}|\wt Q-m_{sc}|+O(|\wt Q-m_{sc}|^2)),\quad  \forall i\in \bT,
\end{equation}
and
\begin{equation}\label{eq:improvedoff}
|\wt G_{oi}-\wt P_{oi}|\lesssim \frac{(\log N )^2\vareps'}{(d-1)^{\ell/2}}
+\log N(\sqrt{\kappa+\eta}|\wt Q-m_{sc}|+O(|\wt Q-m_{sc}|^2)),\quad \forall i\in [N]\backslash\bT.
\end{equation}
\end{prop}
\section{New Green's function proofs}\label{sec:newgf}
\subsection{Entrywise bound}
Our results and our proofs begin to differ from those for the random regular graph at the next step, which is to quantify the difference between $Q$ and $Y_\ell(Q_\ell)$. To do this, we consider $(o,i)\in \vec{E}_0$ and show $G_{oo}^{(i)}$ is well approximated by $Y_\ell(Q_\ell)$. A key difference is that in $\GG$, unlike the random regular graph, the neighborhood of a randomly selected vertex is not almost deterministic (i.e., the Benjamini-Schramm limit is a non-trivial distribution). 

First, we define a quantity that will help approximate $G_{oo}^{(i)}-Y_\ell(Q_\ell)$. Next we will define a complex number $\pi_{oi}$ to be the actual weight of vertices in $V_0$ on the boundary versus what is expected from $Y_\ell(Q_\ell)$. 
\begin{dfn}\label{dfn:pidef}
    For $P\deq G(\Ext(\cB_\ell(\mathcal{Y},o),Q_\ell),z)$ and an arbitrary vertex $l$ such that $\dist(o,l)=\ell$, we define
    \[
    \pi_{oi}:=\frac{P_{ol}^2}{d-1}\left(\left|V_0\cap \left\{v\in V:\dist(o,v)=\ell+1,\, \dist(i,v)=\ell+2\right\}\right|-p^{\ell+1}(d-1)^{\ell+1}\right).
    \]
\end{dfn}
Incorporating this allows us to achieve a tighter concentration bound. 
\begin{prop}\label{prop:improvedQYQ}Recall the definition of $Y_\ell(Q_\ell)$ from \eqref{eq:XdfnYdfn}.
For $\GG\in \oOmega$ and any constant $\fq$, if vertex $o\sim i$ has a radius $\fR$ tree neighborhood, then with probability at least $1-O(N^{-\fq})$,
\begin{equation}\label{eq:improvedQ}
|G_{oo}^{(i)}-Y_\ell(Q_\ell)-\pi_{oi}(Q-m_{sc})|\lesssim \frac{\log N\vareps '}{(d-1)^{\ell/2}}.
\end{equation}
\end{prop}
 \Cref{prop:improvedQYQ} says that regardless of the number of weighted loops close to $o$, we can always treat our neighborhood as being completely tree-like, without loops, by subtracting the term $\pi_{oi}(Q-m_{sc})$ and adjusting the weight on the boundary from $Q$ to $Q_\ell$.
Note that by the Kesten-Stigum theorem, the distribution of $\pi_{oi}$ is nontrivial \cite{kesten1966limit}, so we cannot remove this extra factor of $\pi_{oi}(Q-m_{sc})$ without knowledge that $Q-m_{sc}$ is small.

To show \Cref{prop:improvedQYQ}, we will give a concentration result on $\sum_\alpha (\wt G_{\tilde a_\alpha \tilde a_\alpha}^{(\bT)}-Q)$ which is the first term created when doing a resolvent expansion of \eqref{eq:improvedQ}. The concentration is based on a concentration of measure bound from an Azuma inequality.
\begin{prop}\label{prop:concentration}
Assume that $\GG\in \Omega$ and $o$ has a radius $\fR/4$ tree neighborhood. Recall that $\bT$ is the vertex set of $\cB_\ell(o,\GG)$, and $\wt a_\alpha\in\{a_\alpha,c_\alpha\}$ is the neighbor of $l_\alpha$ after the switch.  
Then for any constant $\fq>0$, with probability $1-O(N^{-\fq})$ over the randomness of the switch, we have
\begin{align}\label{eq:azumaconc}
&\left|\sum_\alpha (\wt G_{\tilde a_\alpha \tilde a_\alpha}^{(\bT)}-Q)\right|\lesssim \frac{\log N\vareps '}{(d-1)^{\ell/2}}.
\end{align}
\end{prop}

\begin{proof}
First, we want to give a bound on the infinity norm. With probability $1-O(N^{-\fq})$ we have $\wt \GG\in \Omega$, which gives
\begin{align}\label{eq:tildeinfbound}
\begin{split}
&\phantom{{}={}}\left|\wt{G}_{\wt a_\alpha \wt a_\alpha}^{(\bT)}-Q\right|\overset{\ref{eq:Qlemma}}{=}\left|\wt{G}_{\wt a_\alpha \wt a_\alpha}^{(\bT)}-\wt Q\right|+O\left(\frac{(d-1)^{2\ell}\left(\Im[m_d]+\vareps'+\vareps/\sqrt{\vareps+\kappa+\eta}\right)}{N\eta}\right)\\
&=\left|(\wt{G}_{\wt a_\alpha \wt a_\alpha}^{(\bT)}-G_{\wt a_\alpha \wt a_\alpha}(\Ext(\cB_r(\wt{ a}_\alpha, \wt{\GG}^{(\bT)}),\wt{Q}),z)-(\wt Q-G_{\wt a_\alpha \wt a_\alpha}(\Ext(\cB_r(\wt{ a}_\alpha, \wt{\GG}^{(\bT)}),\wt{Q}),z))\right|\\
&\hspace{6cm}+O\left(\frac{(d-1)^{2\ell}\left(\Im[m_d]+\vareps'+\vareps/\sqrt{\vareps+\kappa+\eta}\right)}{N\eta}\right).
\end{split}
\end{align}
Here by the assumption that $\wt \GG\in \Omega$ and \Cref{lem:removeT}, we have $$|\wt{G}_{\wt a_\alpha \wt a_\alpha}^{(\bT)}-G_{\wt a_\alpha \wt a_\alpha}(\Ext(\cB_r(\wt{ a}_\alpha, \wt{\GG}^{(\bT)}),\wt{Q}),z)|\leq \vareps'.$$ By \Cref{prop:QvsExt}, we have $$|\wt Q-G_{\wt a_\alpha \wt a_\alpha}(\Ext(\cB_r(\wt{ a}_\alpha, \wt{\GG}^{(\bT)}),\wt{Q}),z)|=O\left(\log N(\sqrt{\kappa+\eta}|\wt Q-m_{sc}|+|\widetilde Q-m_{sc}|^2)\right)=O(\vareps').$$

Next, we want to reduce to a sum over centered independent random variables. We define $\chi_\alpha$ to be the indicator that $c_\alpha$ is distance at least $\fR/4$ from $o$. Note that $\chi_\alpha=1$ for all $\alpha\in\sf W$: if $c_\alpha$ is $\fR/4$-close to $o$, it will be $(\fR/4+\ell)$-close to all other $a_{\alpha'}\in[\mu]$ so it will not satisfy \Cref{def:w}.  By \Cref{lem:configuration} and \Cref{prop:stabilitytGT}, $\mu-|\sf W|\lesssim \log N$ . Therefore, we have 
\begin{align}\begin{split}
\phantom{{}= {}}\sum_{\alpha\in [\mu]}(\wt{G}_{\wt a_\alpha \wt a_\alpha}^{(\bT)}-Q)&=\sum_{\alpha\in [\mu]}\chi_\alpha(\wt{G}_{c_\alpha c_\alpha}^{(\bT)}-Q)+O\left(\log N\vareps'\right)\\
&=\sum_{\alpha\in [\mu]}\chi_\alpha(G_{c_\alpha c_\alpha}^{(\bT b_\alpha)}-Q)+O\left(\log N\vareps'+(d-1)^{2\ell}\varphi^2\right)
\end{split}
\end{align}
where the last line is true by \eqref{eq:Ubound}.

For every $\alpha$ with $\chi_\alpha=1$, we use \eqref{e:Schurixj} to obtain
\[
G_{c_\alpha c_\alpha}^{(\bT b_\alpha)}=G^{(\bT)}_{c_\alpha c_\alpha}-G^{(\bT)}_{c_\alpha b_\alpha}(G_{b_\alpha b_\alpha}^{(\bT)})^{-1}G^{(\bT)}_{b_\alpha c_\alpha }.
\]
By \Cref{lem:removeT}, we can approximate this as 
\begin{multline*}
G_{c_\alpha c_\alpha}(\Ext(\cB_r(c_\alpha,\GG^{(\bT)}),Q),z)\\
-G_{c_\alpha b_\alpha}(\Ext(\cB_r(\{c_\alpha, b_\alpha\},\GG^{(\bT)}),Q),z)(G_{b_\alpha b_\alpha}(\Ext(\cB_r(b_\alpha,\GG^{(\bT)}),Q),z))^{-1}G_{b_\alpha c_\alpha}(\Ext(\cB_r(\{c_\alpha,b_\alpha\},\GG^{(\bT)}),Q),z)+O(\varepsilon').
\end{multline*}
Recall that $b_\alpha$ and $c_\alpha$ are adjacent in the graph $\GG$. It is not hard to show that we can pass from the ball around $b_\alpha$ or around $\{b_\alpha,c_\alpha\}$ to that around $c_\alpha$ with minimal error (see \cite[Proposition 2.14]{huang2024spectrum}). Thus the above equals
\begin{align*}
&\phantom{{}={}}G_{c_\alpha c_\alpha}(\Ext(\cB_r(c_\alpha,\GG^{(\bT)}),Q),z)\\
&-G_{c_\alpha b_\alpha}(\Ext(\cB_r(c_\alpha,\GG^{(\bT)}),Q),z)(G_{b_\alpha b_\alpha}(\Ext(\cB_r(c_\alpha, \GG^{(\bT)}),Q),z))^{-1}G_{b_\alpha c_\alpha}(\Ext(\cB_r(c_\alpha,\GG^{(\bT)}),Q),z)+O(\vareps')\\
&\overset{\eqref{e:Schurixj}}{=}G_{c_\alpha c_\alpha}(\Ext(\cB_r(c_\alpha,\GG^{(\bT b_\alpha)}),Q),z)+O(\vareps').
\end{align*}
Moreover, we know from \Cref{prop:QvsExt} that
\[
|G_{c_\alpha c_\alpha}(\Ext(\cB_r(c_\alpha,\GG^{(\bT b_\alpha)}),Q),z)-Q|\lesssim \vareps'.
\]
This means that for every $\alpha$ with $\chi_\alpha=1$, we have 
\be\label{eq:infboundTb}
|\chi_\alpha(G^{(\bT b_\alpha)}_{c_\alpha c_\alpha}-Q)|\lesssim \vareps',
\ee which will be incorporated into the Azuma inequality.

We now show that $\E_{\bf S}[{G_{c_\alpha c_\alpha}^{(\bT b_\alpha)}}]\approx Q$.
To get the expectation of $G_{c_\alpha c_\alpha}^{(\bT b_\alpha)}$ over the randomly selected $c_\alpha$, we expand according to \eqref{e:Schur1} and obtain that
\be\label{eq:wardtype}
G_{c_\alpha c_\alpha}^{(\bT b_\alpha)}=G_{c_\alpha c_\alpha}^{(b_\alpha)}-(G^{(b_\alpha)}(G^{(b_\alpha)}|_{\bT})^{-1}G^{(b_\alpha)})_{c_\alpha c_\alpha}.
\ee
The first term on the right-hand side has expectation $ Q$, by definition. We want to use a Ward identity to upper bound the second term. However, we cannot directly use this on, for example, $\E_{\bf S}[|G^{(b_\alpha)}_{c_\alpha x}|^2]$, because as we vary $c_\alpha$ the deleted vertex $b_\alpha$ changes, thereby also changing the graph. Therefore, we  expand

\begin{align}\label{eq:CStoWard}
\begin{split}
\E_{\bf S}[|(G^{(b_\alpha)}(G^{(b_\alpha)}|_{\bT})^{-1}G^{(b_\alpha)})_{c_\alpha c_\alpha}|]&=\E_{\bf S}[|\sum_{xy\in \bT}G^{(b_\alpha)}_{c_\alpha x}(G^{(b_\alpha)}|_{\bT})^{-1}_{xy}G^{(b_\alpha)}_{y c_\alpha}|]\\
& \leq \E_{\bf S}[\sum_{xy\in \bT}|(G^{(b_\alpha)}|_{\bT})^{-1}_{xy}|(|G^{(b_\alpha)}_{c_\alpha x}|^2+|G^{(b_\alpha)}_{y c_\alpha}|^2)].
\end{split}
\end{align}
The first step is to show that $|(G^{(b_\alpha)}|_{\bT})_{xy}^{-1}|$ is not too large. Thus we define $P:=G(\Ext(\cB_r(\bT,\GG),Q),z)$ and expand
\be\label{eq:Gbalphaexpansion}
(G^{(b_\alpha)}|_{\bT})^{-1}=\sum_{k=0}^{\infty}(P|_{\bT})^{-1}\left((P|_{\bT}-G^{(b_\alpha)}|_{\bT})(P|_{\bT})^{-1}\right)^k.
\ee
To bound this, we have by the Schur complement formula \eqref{e:Schur2} that 
\[
\|(P|_{\bT})^{-1}\|=\|H|_{\bT}-z-B^*QB\|\lesssim 1.
\]
Therefore, splitting $(G^{(b_\alpha)}|_{\bT})^{-1}$ in \eqref{eq:Gbalphaexpansion} according to $k=0$ and $k>0$ gives that for $x,y\in \bT$, we have
\begin{align}\label{eq:Gbalphabound}
    |(G^{(b_\alpha)}|_{\bT})^{-1}_{xy}|=|(P|_{\bT})^{-1}_{xy}|+O(\vareps)\lesssim 1
\end{align}
where we used that $(P|_{\bT}-G^{(b_\alpha)}|_{\bT})_{xy}=O(\vareps)$ by $\Omega$.

Returning to \eqref{eq:CStoWard}, we are left to upper bound
\[
\E_{\bf S}[\sum_{xy\in \bT}(|G^{(b_\alpha)}_{c_\alpha x}|^2+|G^{(b_\alpha)}_{y c_\alpha}|^2)]\lesssim (d-1)^\ell \sum_{x\in \bT}\E_{\bf S}[|G^{(b_\alpha)}_{c_\alpha x}|^2].
\]
We expand
\[
    G^{(b_\alpha)}_{c_\alpha x}\overset{\eqref{e:Schurixj}}{=}G_{c_\alpha x}-G_{c_\alpha b_\alpha}(G_{b_\alpha b_\alpha})^{-1}G_{b_\alpha x}.
\]
We know from the Ward identity \eqref{eq:wardeasy} that
\[
    \E_{\bf S}[|G_{c_\alpha x}|^2]\lesssim \frac{\Im[G_{xx}]}{N\eta}\lesssim \frac{\Im[m_d]+\vareps'+\vareps/\sqrt{\vareps+\kappa+\eta}}{N\eta}.
\]
Because of $\Omega$, it is not hard to show that $|G_{b_\alpha c_\alpha}|,|G_{b_\alpha b_\alpha}|\asymp 1$ (see \cite[Proposition 2.12]{huang2024spectrum}), meaning
\[
    \E_{\bf S}[|G_{c_\alpha b_\alpha}(G_{b_\alpha b_\alpha})^{-1}G_{b_\alpha x}|^2]\lesssim \E_{\bf S}[|G_{b_\alpha x}|^2]\lesssim \frac{\Im[G_{xx}]}{N\eta}\lesssim \frac{\Im[m_d]+\vareps'+\vareps/\sqrt{\vareps+\kappa+\eta}}{N\eta}.
\]
Altogether, we get that
\begin{align}\label{eq:Gbxbound}
    \sum_{x\in \bT}\E_{\bf S}[|G^{(b_\alpha)}_{c_\alpha x}|^2]\lesssim \frac{\Im[m_d]+\vareps'+\vareps/\sqrt{\vareps+\kappa+\eta}}{N\eta}.
\end{align}

Putting together \eqref{eq:CStoWard}, \eqref{eq:Gbalphabound} and \eqref{eq:Gbxbound} gives that
\be\label{eq:Tbalphaerrorexp}
\E_{\bf S}[|(G^{(b_\alpha)}(G^{(b_\alpha)}|_{\bT})^{-1}G^{(b_\alpha)})_{c_\alpha c_\alpha}|]\leq \frac{(d-1)^{2\ell}\left(\Im[m_d]+\vareps'+\vareps/\sqrt{\vareps+\kappa+\eta}\right)}{N\eta}.
\ee
Using an Azuma lemma and \eqref{eq:infboundTb} gives
\[
\Pr\left(\left|\sum_\alpha \chi_\alpha(G^{(\bT b_\alpha)}_{c_\alpha c_\alpha}-Q)\right|\geq t\vareps'(d-1)^{\ell/2}+\frac{(d-1)^{2\ell}\left(\Im[m_d]+\vareps'+\vareps/\sqrt{\vareps+\kappa+\eta}\right)}{N\eta}\right)\leq \se^{-t^2/C}.
\]
We can set $t=\log N$ to make a sufficiently small probability event. This gives \eqref{eq:azumaconc}.
\end{proof}

We can now show \Cref{prop:improvedQYQ} using \Cref{prop:concentration}.

\begin{proof}[Proof of \Cref{prop:improvedQYQ}]
We recall the definition of $P$ and $P_{ol}$ from \Cref{dfn:pidef}. We note that 
\be\label{eq:Polbound}
|P_{ol}|\lesssim (d-1)^{-\ell/2}
\ee
(see, e.g.,  \cite[Lemma 4.4]{huang2024optimal}).

Now, we expand around $o$ as if it is part of the original infinite graph $\GG$, rather than the finite graph, as per \Cref{dfn:finitization}. We define
$B\deq H_{\bT \bT^\complement}$, so $B$ is the adjacency operator from vertices of distance $\ell$ to $o$ to distance $\ell+1$. We have
\[
\wt{G}_{oo}^{(i)}-Y_\ell(\wt Q_\ell)\overset{\eqref{e:Schur2}}{=}\left(-z+H_{\bT^{(i)}}- B\wt{G}^{(\bT)}B^*\right)^{-1}_{oo}-\left(-z+H_{\bT^{(i)}}-B\left(p^{\ell+1}\wt{Q}+(1-p^{\ell+1})m_{sc}\right)\bI B^*\right)^{-1}_{oo}.
\]
We then expand according to \eqref{e:resolv}. Letting $\widetilde P\deq G(\Ext(\cB_\ell(o,\GG^{(i)}),\widetilde Q_\ell),z)$ (so $\widetilde P_{oo}=Y_\ell(\widetilde Q_\ell)$), we have

\begin{align}\label{eq:gminusy-1}
&\phantom{{}={}}\wt G_{oo}^{(i)}-Y_\ell(\wt {Q}_\ell)=(\wt{G} B(\wt G^{(\bT)}-\wt {Q}_\ell\bI)B^*\wt{P})_{oo}\notag\\
&=(\wt{P}B(\wt G^{(\bT)}-\wt {Q}_\ell \bI)B^*\wt{P})_{oo}+(\wt GB(\wt G^{(\bT)}-\wt {Q}_\ell \bI)B^*\wt{P}B(\wt G^{(\bT)}-\wt {Q}_\ell \bI)B^*\wt{P})_{oo}\notag\\
&=\frac{\wt{P}_{ol}^2}{d-1}\sum_{\alpha\in[\mu]}\wt G_{\tilde a_\alpha \tilde a_\alpha}^{(\bT)}+\frac{\wt{P}_{ol}^2}{d-1}\sum_{\alpha\neq \beta}\wt G_{\wt a_\alpha \wt a_\beta}^{(\bT)}+\frac{\wt{P}_{ol}^2}{d-1}((d-1)^{\ell+1}-\mu)m_{sc}-\wt P_{ol}^2 (d-1)^{\ell} \wt Q_\ell\\
&\hspace{6cm}+(\wt GB(\wt G^{(\bT)}-\wt {Q}_\ell \bI)B^*\wt{P}B(\wt G^{(\bT)}-\wt {Q}_\ell \bI)B^*\wt{P})_{oo}.\notag
\end{align}
In the last line, the first three terms are parts of $(\wt{P}B\wt G^{(\bT)} B^*\wt{P})_{oo}$ corresponding to the different choices of boundary vertices when passing from  $\mathbb T$ to $\mathbb T^\complement$. The first term is if we choose the same boundary vertex twice, the second is if we choose two different ones, and the third is if we choose a boundary vertex not in $V_0$.   The fourth term in the last line equals $(\wt{P}B \wt {Q}_\ell \bI B^*\wt{P})_{oo}$.

By \Cref{lem:removeT} and \Cref{prop:QvsExt}, the last term is $O((d-1)^{2\ell}(\vareps')^2)$. Moreover, by another resolvent expansion (see \cite[Remark 2.13]{huang2024spectrum}), we have $\wt P_{ol}=P_{ol}+O(\ell(d-1)^{-\ell/2}\epsilon)$. Therefore, up to negligible error, we can reduce to
\[
\frac{P_{ol}^2}{d-1}\sum_{\alpha\in[\mu]}\wt G_{\tilde a_\alpha \tilde a_\alpha}^{(\bT)}+\frac{P_{ol}^2}{d-1}\sum_{\alpha\neq \beta}\wt G_{\wt a_\alpha \wt a_\beta}^{(\bT)}+\frac{P_{ol}^2}{d-1}((d-1)^{\ell+1}-\mu)m_{sc}-P_{ol}^2 (d-1)^{\ell} \wt Q_\ell.
\]
For the second term, by splitting into cases whether $\alpha,\beta\in\sf U$, we know from \Cref{prop:stabilitytGT} that
\[
\left|P_{ol}^2\sum_{\alpha\neq \beta}\wt G_{\wt a_\alpha \wt a_\beta}^{(\bT)}\right|\lesssim (d-1)^{\ell}((\vareps')^3+\varphi)+(d-1)^{-\ell}\log N \vareps.
\]
Thus, we can reduce to bounding
\begin{align}\label{eq:gminusy-2}
\begin{split}
&\phantom{{}={}}\frac{P_{ol}^2}{d-1}\sum_{\alpha\in[\mu]}\wt G_{\tilde a_\alpha \tilde a_\alpha}^{(\bT)}+P_{ol}^2((d-1)^{\ell}-\mu/(d-1))m_{sc}-P_{ol}^2 (d-1)^{\ell} \wt Q_\ell\\
&=\frac{P_{ol}^2}{d-1}\sum_{\alpha\in[\mu]}\wt G_{\tilde a_\alpha \tilde a_\alpha}^{(\bT)}+P_{ol}^2((d-1)^{\ell}-\mu/(d-1))m_{sc}-P_{ol}^2 (d-1)^{\ell} Q_\ell+O\left(\frac{(d-1)^{2\ell}\left(\Im[m_d]+\vareps'+\vareps/\sqrt{\vareps+\kappa+\eta}\right)}{N\eta}\right)\\
&=\left(\frac{P_{ol}^2}{d-1}\sum_{\alpha\in[\mu]}(\wt G_{\tilde a_\alpha \tilde a_\alpha}^{(\bT)}-Q)\right)+\pi_{oi}(Q-m_{sc}) +O\left(\frac{(d-1)^{2\ell}\left(\Im[m_d]+\vareps'+\vareps/\sqrt{\vareps+\kappa+\eta}\right)}{N\eta}\right).
\end{split}
\end{align}
Here, the first equality results from \Cref{lem:Qlemma} and the second results from the definition of $Q_\ell$ from \eqref{eq:Qldef} and of $\pi_{oi}$ from \Cref{dfn:pidef}.

Therefore, by \Cref{lem:Qlemma}, \Cref{prop:concentration} and \eqref{eq:gminusy-2}, we have
\begin{align*}
&\phantom{{}={}}\wt G_{oo}^{(i)}-Y_\ell(\wt Q_\ell)-\pi_{oi}(Q-m_{sc})\\
&=\frac{P_{ol}^2}{d-1}\sum_{\alpha\in[\mu]}(\wt G_{\tilde a_\alpha \tilde a_\alpha}^{(\bT)}-Q) +O\left(\frac{\log N\left(\Im [m_d]+\vareps'+\vareps/\sqrt{\vareps+\kappa+\eta}\right)}{N\eta}+N^{-1+\fc}\right)\\
&=O\left(\log N\vareps'(d-1)^{-\ell/2}+\frac{\log N\left(\Im [m_d]+\vareps'+\vareps/\sqrt{\vareps+\kappa+\eta}\right)}{N\eta}+N^{-1+\fc}\right)=O\left(\log N\vareps'(d-1)^{-\ell/2}\right)
\end{align*}
with probability $1-O(N^{-\fq})$.
\end{proof}

\begin{proof}[Proof of \Cref{prop:minusimprovement}]
We first show \eqref{eq:firsthalving}. Indeed, this can be done by using \Cref{prop:improvedQYQ}. We have
\begin{equation}\label{eq:QYQexp}
\wt{Q}-Y_\ell(\wt{Q}_\ell)=\frac1{|\vec{E}_0|}\sum_{(o,i)\in \vec{E}_0} \wt{G}_{oo}^{(i)}-Y_\ell(\wt{Q}_\ell)-\pi_{oi}(\wt{Q}-m_{sc})+\frac1{|\vec{E}_0|}\sum_{(o,i)\in \vec{E}_0} \pi_{oi}(\wt{Q}-m_{sc}).
\end{equation}

By \Cref{lem:tanglefree}, the contribution from all pairs of vertices $(o,i)\in \vec{E}_0$ that do not have a tree neighborhood of radius $\fR$ in this sum is $O(N^{-1+\fc})$. For the parts that do have such a neighborhood, \Cref{prop:improvedQYQ} gives that this sum is
\[
\left|\frac1{|\vec{E}_0|}\sum_{(o,i)\in \vec{E}_0} \wt{G}_{oo}^{(i)}-Y_\ell(\wt{Q}_\ell)-\pi_{oi}(\wt{Q}-m_{sc})\right|\lesssim \frac{\log N\vareps '}{(d-1)^{\ell/2}}.
\]

We now want to show that $\pi_{oi}$ is well concentrated. To do this, we use a McDiarmid inequality.   If we toggle each edge from existing to not existing, we change $|\frac1{|\vec{E}_0|}\sum_{o\in V_0, i\sim o} \pi_{oi}|$ by at most $2(d-1)^\ell/N$. Moreover, according to our model, 
\[\E[\pi_{oi}]=0
\]
for every $o,i\in \vec{E}_0$ with a tree neighborhood of radius $\ell$. Once again using that $|\vec{E}_0|$ is bounded from below by \eqref{eq:E1size}, we have that
\be\label{eq:McDiarmid}
\Pr\left(\left|\frac1N\sum_{oi}\pi_{oi}\right|>N^{-1/3}\right)\lesssim \se^{-\frac{N^{1/3}}{(d-1)^\ell}}.
\ee
This means that with high probability, the total contribution of the last term in \eqref{eq:QYQexp} is at most $\vareps/N^{1/3}$, which gives
\begin{equation}\label{eq:QYQbound}
|\wt{Q}-Y_\ell(\wt Q_\ell)|\leq \frac{\log N\vareps'}{(d-1)^{\ell/2}}.
\end{equation}
Given this, we can proceed as per the proof of  \cite[Proposition 4.12]{huang2024spectrum}, beginning with equation 4.31, to give \eqref{eq:firsthalving}.
\eqref{eq:secondhalving} then follows immediately from combining \eqref{eq:firsthalving} and \Cref{prop:improvedswitch}.
\end{proof}

\subsection{Proving optimal spectral gap}
We now will improve our above error terms to such an extent that we can preclude the existence of large eigenvalues. Here, we do not calculate the contribution of the top eigenvalue explicitly. Instead we use our knowledge of the weak spectral gap from \Cref{thm:first-ev} to let us ignore the contribution of the top eigenvalue and compare different $\ell$. We define  $\bD\in \R^{N}$ to be the vector such that $\bD(x)$ is the degree of $x$ in $\GG_0$ (so the number of neighbors of $x$ that are in $V_0$), and we recall the all ones matrix $J\in \R^{N\times N}$.

We first explicitly write out the difference between $Q$ and $m_N$ and their tree counterparts. As we will see, this will be negligible for our $z$ interest. We define
\begin{align}\begin{split}\label{eq:deltaQdef}
\delta_Q&:=m_{sc}^2p^{2}(d-2)\frac{p(d-1)-p^{-\ell}(d-1)^{-\ell}}{p(d-1)-1}\left(1+\frac{m_{sc}}{\sqrt{d-1}}\right)^2(m_{sc}p\sqrt{d-1})^{2\ell}\frac{\bD^*G\bD}{\bD^*J\bD}\\
    \delta_m&:=\frac{d}{d-1}m_d^2p^{2}\left((d-1)^{-\ell}+(d-2)\frac{p(d-1)-p^{-\ell}(d-1)^{-\ell}}{p(d-1)-1}\right)\left(1+\frac{m_{sc}}{\sqrt{d-1}}\right)^2(m_{sc}p\sqrt{d-1})^{2\ell}\frac{\bD^*G\bD}{\bD^*J\bD}.
    \end{split}
\end{align}
We now proceed to show  \Cref{thm:second-ev}, which is implied by the following statement on high moments. 
\begin{lemma}\label{claim:momentbound}
For any integer $\rho\geq0$, we have
\be\label{eq:QYQhighmoment}
\E\left[|Q-Y_\ell(Q_\ell)-\delta_Q|^{2\rho} {\mathbf 1}(\GG\in \overline\Omega)\right]\lesssim\left(\vareps^2+\vareps^{1/2}\left(\frac{(d-1)^{3\ell}(\Im[m_d]+\vareps'+\vareps/\sqrt{\kappa+\eta})}{N\eta}\right)^{1/2}\right)^{2\rho}
\ee
and
\be\label{eq:mXQhighmoment}
\E\left[|m_N-X_\ell(Q_\ell)-\delta_m|^{2\rho} {\mathbf 1}(\GG\in \overline\Omega)\right]\lesssim\left(\vareps^2+\vareps^{1/2}\left(\frac{(d-1)^{3\ell}(\Im[m_d]+\vareps'+\vareps/\sqrt{\kappa+\eta})}{N\eta}\right)^{1/2}\right)^{2\rho}.
\ee
\end{lemma}
We first show that \Cref{claim:momentbound} implies  \Cref{thm:second-ev}. 
\begin{proof}[Proof of  \Cref{thm:second-ev}]
We fix integer $\rho\geq \fq/2\ff$. Then by Markov's inequality and \eqref{eq:QYQhighmoment},
\begin{align}\label{eq:smallQYQ}
\begin{split}
&\phantom{{}={}}\Pr\left[|Q-Y(Q_\ell)-\delta_Q|\geq N^{\ff}\left(\vareps^2+\vareps^{1/2}\left(\frac{(d-1)^{3\ell}(\Im[m_d]+\vareps'+\vareps/\sqrt{\kappa+\eta})}{N\eta}\right)^{1/2}\right)\right]\\
&=\Pr\left[|Q-Y(Q_\ell)-\delta_Q|^{2\rho}\geq N^{2\rho \ff}\left(\vareps^2+\vareps^{1/2}\left(\frac{(d-1)^{3\ell}(\Im[m_d]+\vareps'+\vareps/\sqrt{\kappa+\eta})}{N\eta}\right)^{1/2}\right)^{2\rho}\right]\\
&=O(N^{-2\rho\ff})=O(N^{-\fq}).
\end{split}
\end{align}
We then have that under this high probability event, from the expansion of \eqref{eq:yexpansion},
\begin{align*}
&
\phantom{{}={}}O\left(N^{\ff}\left(\vareps^2+\vareps^{1/2}\left(\frac{(d-1)^{3\ell}(\Im[m_d]+\vareps'+\vareps/\sqrt{\kappa+\eta})}{N\eta}\right)^{1/2}\right)\right)\\
&=Q-Y_\ell(Q_\ell)-\delta_Q\\
&=Q-m_{sc}-m_{sc}^{2\ell+2}(Q_\ell-m_{sc})-\left(m_{sc}^{2\ell+3}\left(\frac{1-(m_{sc})^{2\ell+2}}{1-(m_{sc})^{2}}\right)+O(\ell^2|Q_\ell-m_{sc}|)\right)(Q_\ell-m_{sc})^2-\delta_Q\\
&=Q-m_{sc}-m_{sc}^{2\ell+2}p^{\ell+1}(Q-m_{sc})-\left(m_{sc}^{2\ell+3}\left(\frac{1-(m_{sc})^{2\ell+2}}{1-(m_{sc})^{2}}\right)+O(\ell^2p^{\ell+1}|Q-m_{sc}|)\right)p^{2\ell+2}(Q-m_{sc})^2-\delta_Q.
\end{align*}

This gives a quadratic equation in $(Q-m_{sc})$. The two solutions to an equation of the form $a x^2+bx+c=0$ are $r_1=-b/a+O(|c/b|)$ and $r_2=O(|c/b|)$. By the same argument as the proof of  \cite[Proposition 4.12]{huang2024optimal}, only the second solution is feasible, which gives
\be\label{eq:accQmsc}
\left|Q-m_{sc}-\frac{\delta_Q}{1-p^{\ell+1}m_{sc}^{2\ell+2}}\right|
\lesssim \frac{N^{\ff}}{1-p^{\ell+1}m_{sc}^{2\ell+2}}\left(\vareps^2+\vareps^{1/2}\left(\frac{(d-1)^{3\ell}(\Im[m_d]+\vareps'+\vareps/\sqrt{\kappa+\eta})}{N\eta}\right)^{1/2}\right).
\ee
We can do a similar calculation as \eqref{eq:smallQYQ} for $m_N-X_\ell(Q_\ell)$, giving 
\be\label{eq:accmnmd}
|m_N-X_\ell(Q_\ell)-\delta_m|\lesssim N^{\ff}\left(\vareps^2+\vareps^{1/2}\left(\frac{(d-1)^{3\ell}(\Im[m_d]+\vareps'+\vareps/\sqrt{\kappa+\eta})}{N\eta}\right)^{1/2}\right).
\ee
By performing the expansion \eqref{eq:xexpansion} and combining \eqref{eq:accQmsc} and \eqref{eq:accmnmd}, we have
\begin{align*}
m_N-m_d&=m_N-X_\ell(Q_\ell)+X_\ell(Q_\ell)-m_d\\
&=\delta_m+\frac{dm_d^2m_{sc}^{2\ell}p^{\ell+1}}{(d-1)(1-p^{\ell+1}m_{sc}^{2\ell+2})}\delta_Q+\frac{O(N^\ff)}{\sqrt{\kappa+\eta}}\left(\vareps^2+\vareps^{1/2}\left(\frac{(d-1)^{3\ell}(\Im[m_d]+\vareps'+\vareps/\sqrt{\kappa+\eta})}{N\eta}\right)^{1/2}\right).
\end{align*}
Ignoring the negligible error term, the right hand side is equal to
\begin{multline}\label{eq:mndiff}
    \frac{d}{d-1}m_d^2p^{2}\left((d-1)^{-\ell}+(d-2)\frac{p(d-1)-p^{-\ell}(d-1)^{-\ell}}{p(d-1)-1}\left(1+p^{\ell+1}\frac{m_{sc}^{2\ell+2}}{1-p^{\ell +1}m_{sc}^{2\ell+2}}\right)\right)\\
    \cdot \left(1+\frac{m_{sc}}{\sqrt{d-1}}\right)^2(m_{sc}p\sqrt{d-1})^{2\ell}\frac{\bD^*G\bD}{\bD^*J\bD}.
\end{multline}

For $\upsilon:=2d/(\log\log_{d-1} N)$, in order to show \Cref{thm:second-ev}, by \Cref{thm:first-ev}, it is sufficient to show that there is no eigenvalue $\lambda_2(\cG)$ satisfying $2\sqrt{d-1}+\upsilon\leq \lambda_2\leq (1-\delta)(p(d-1)+p^{-1})$. Therefore, assume such a $\lambda_2$ exists. 
Then we set $z=E+\ri \eta$, where $E=2+\kappa=\lambda_2/\sqrt{d-1}$ and $\eta=\log^{300}N/N$.  In this case, by \Cref{lem:localization} and \eqref{e:m}, $\Im[m_N]\geq 1/(N(\log_{d-1}\log N)^2 \eta)$. Therefore, by \eqref{eq:dynamic},
\begin{align}
\begin{split}
|m_N-m_d|&\geq |\Im[m_N]-\Im[m_d]|\\
&\geq \frac1{N(\log_{d-1}\log N)^2\eta}-O(\eta/\sqrt{\kappa+\eta})\\
&\geq (1-o_N(1))\frac{1}{N(\log_{d-1}\log N)^2 \eta}.
\end{split}
\end{align}
On the other hand, by \eqref{eq:mndiff} and \eqref{eq:DJbound}
\begin{align}\label{eq:doittwice}
\begin{split}
&\phantom{{}={}}m_N-m_d\\
&=(1+O((\log_{d-1} N)^{-6}))\frac{d}{d-1}m_d^2p^{2}\left((d-2)\frac{p(d-1)}{p(d-1)-1}\right)\left(1+\frac{m_{sc}}{\sqrt{d-1}}\right)^2(m_{sc}p\sqrt{d-1})^{2\ell}\frac{\bD^*G\bD}{\bD^*J\bD}.
\end{split}
\end{align}

Thus we must have
\[
\left|(m_{sc}p\sqrt{d-1})^{2\ell}\frac{\bD^*G\bD}{\bD^*J\bD}\right|\gtrsim \frac1{N(\log_{d-1}\log N)^2 \eta}.
\]

Suppose  $z=\zeta+1/\zeta$, for $\zeta>1$. Then we have $m_{sc}(z)=-1/\zeta$ by \eqref{eq:msc}. In general, since $m_{sc}(z)$ is 1-Lipschitz and $2+\frac{\upsilon}{\sqrt{d-1}}\leq \text{Re}(z)\leq (1+o_N(1))\left(\se^{-\frac{1}{21c}}p\sqrt{d-1}+\frac{1}{\se^{-\frac{1}{21c}}p\sqrt{d-1}}\right)$ (as seen in \eqref{eq:spectral-gap}), we get that
\[
\frac{\se^{\frac{1}{21c}}}{p\sqrt{d-1}}\leq |m_{sc}(z)|\leq\frac1{1+\upsilon/2}.
\]
Therefore, it is impossible that \eqref{eq:doittwice} is satisfied for all $\ell\in [12\log_{d-1}\log N,24\log_{d-1}\log N]$. Indeed,  the choices of $\ell=12\log_{d-1}\log N$ and $\ell=24\log_{d-1}\log N$ which, when applied to \eqref{eq:doittwice}, have a difference far larger than the error term of $\frac{N^\ff}{\sqrt{\kappa+\eta}}\left(\vareps^2+\vareps^{1/2}\left(\frac{(d-1)^{3\ell}(\Im[m_d]+\vareps'+\vareps/\sqrt{\kappa+\eta})}{N\eta}\right)^{1/2}\right)$. Thus there can be no such $\lambda_2$. 
\end{proof}

We now show \Cref{claim:momentbound}. We first show that $\delta_Q,\delta_m$ do not greatly change when we apply the switch.
\begin{lemma}\label{lem:Qswitch}
We have that
\[
|\delta_Q-\delta_{\wt{Q}}|,\, |\delta_{m}-\delta_{\wt{m}}|\lesssim (d-1)^{3\ell}\frac{\Im[m_d]+\vareps'+\vareps/\sqrt{\vareps+\kappa+\eta}}{N\eta}+N^{-1+\fc}.
\]
\end{lemma}
\begin{proof}
 This will follow from a similar argument as what is shown in \cite[Lemma 5.22]{huang2024spectrum}.
We define $\chi_x$ to be the indicator that $x\notin\bT\bW$. This is satisfied by $N-N^{o_N(1)}$ vertices. 

We can split
\begin{align*}
(G-\wt{G})_{xy}=(\chi_x\chi_y+(1-\chi_x\chi_y))(G-\wt{G})_{xy}=\chi_x\chi_y(G-\wt{G})_{xy}+O(N^{1+o_N(1)}).
\end{align*}
By \eqref{e:Schur1},
\begin{align*}
\phantom{{}={}}\chi_x\chi_y\left|(G-\wt{G})_{xy}\right|=\chi_x\chi_y\left|\left(G(G|_{\bT \bW})^{-1}G-\wt{G}(\wt{G}|_{\bT \bW})^{-1}\wt{G}\right)_{xy}\right|\lesssim \sum_{u,v\in \bT\bW}\left|G_{xu}\right|^2+\left|G_{yv}\right|^2,
\end{align*}
as the entries of $(G|_{\bT \bW})^{-1}$ can be bounded using an argument similar to \eqref{eq:Gbalphabound}.
We can then sum over all $x,y$ and $u,v$. By \eqref{eq:wardeasy}, we have 
\[
\sum_{x,y\in [N]}\mathbf D(x)\mathbf D(y)\chi_x\chi_y((G-\wt{G})_{xy})\lesssim \sum_{u,v\in \bT\bW}N\frac{\Im[G_{uu}]+\Im[G_{vv}]}{\eta}\lesssim(d-1)^{2\ell}N\frac{\Im[m_d]+\vareps'+\vareps/\sqrt{\vareps+\kappa+\eta}}{\eta}.
\]

Therefore,
\be\label{eq:qshift1}
|\bD^*(G-\wt G)\bD|\lesssim (d-1)^{2\ell}N\frac{\Im[m_d]+\vareps'+\vareps/\sqrt{\vareps+\kappa+\eta}}{\eta}+N^{1+o_N(1)}.
\ee
By \eqref{eq:E1size},
\be\label{eq:DJbound}
\bD^*J\bD=4|E_0|^2\geq p^2d^2N^2.
    \ee
By \eqref{eq:deltaQdef}, \eqref{eq:qshift1}, and \eqref{eq:DJbound}, we have that
\[
|\delta_Q-\delta_{\wt{Q}}|,\, |\delta_m-\delta_{\wt{m}}|\lesssim (d-1)^{3\ell}\frac{\Im[m_d]+\vareps'+\vareps/\sqrt{\vareps+\kappa+\eta}}{N\eta}+N^{-1+\fc}.
\]
\end{proof}

\begin{proof}[Proof of \Cref{claim:momentbound}]
We will use the same strategy as \cite[Section 7.1]{huang2024spectrum}. The main differences in our proof is that $\delta_Q$ is not deterministic, the contribution from different vertices is non-uniform, and we work with $Q_\ell$ rather than $Q$.  Here we describe the proof for \eqref{eq:QYQhighmoment}. The proof of \eqref{eq:mXQhighmoment} is analogous.  

Recall that $o$ is any vertex in $V_0$ that we choose to be the center of resampling, and $i$ is any neighbor of $o$ in $\GG_0$. Define the indicator $I\deq\prod_{\alpha}I_\alpha$, where $I_\alpha$ was defined in \Cref{sec:switching}. We can then write, for $Q'\deq Q-\delta_Q$,
\begin{align*}
&\phantom{{}={}}\E[|Q'-Y_\ell(Q_\ell)|^{2p} {\bf1}(\GG\in \oOmega)]\\
&=\E[N^2/|\vec{E}_0|\cdot A_{oi}({G}_{ii}^{(o)}-\delta_Q-Y_\ell( {Q}_\ell))( {Q'}-Y_\ell( {Q}_\ell))^{p-1}(\overline{ {Q'}-Y_\ell( {Q}_\ell)})^{p} {\bf1}(\GG\in \oOmega)]\\
&=\E[N^2/|\vec{E}_0|\cdot A_{oi}({G}_{ii}^{(o)}-\delta_Q-Y_\ell( {Q}_\ell))( {Q'}-Y_\ell( {Q}_\ell))^{p-1}(\overline{ {Q'}-Y_\ell( {Q}_\ell)})^{p} {\bf1}(\GG,\wt \GG\in \Omega,I)]\\
&\qquad  +\E[N^2/|\vec{E}_0|\cdot A_{oi}({G}_{ii}^{(o)}-\delta_Q-Y_\ell( {Q}_\ell))( {Q'}-Y_\ell( {Q}_\ell))^{p-1}(\overline{ {Q'}-Y_\ell( {Q}_\ell)})^{p} {\bf1}(\GG\in \oOmega)(1-\vone(\GG,\wt \GG\ \in \Omega,I)]\\
&=\E[N^2/|\vec{E}_0|\cdot A_{oi}({G}_{ii}^{(o)}-\delta_Q-Y_\ell( {Q}_\ell))( {Q'}-Y_\ell( {Q}_\ell))^{p-1}(\overline{ {Q'}-Y_\ell( {Q}_\ell)})^{p} {\bf1}(\GG,\wt \GG\in \Omega,I)]\\
&\hspace{10cm}+N^{-1+\fc}\E(|Q-Y_\ell(Q_\ell)|^{2p} {\bf1}(\GG\in \Omega)),
\end{align*}
by \Cref{lem:configuration} and \Cref{thm:entrywisebound}, letting us reduce to the first term.

By the exchangeability of the switch, we can write this as
\begin{align*}
&\phantom{{}={}}\E\left[N^2/|\vec{E}_0|\cdot A_{oi}({G}_{ii}^{(o)}-\delta_Q-Y_\ell( {Q_\ell}))( {Q'}-Y_\ell( {Q_\ell}))^{p-1}(\overline{ {Q'}-Y_\ell( {Q_\ell})})^{p} {\bf1}(\GG,\wt \GG\in \Omega,I)\right]\\
&=\E\left[N^2/|\vec{E}_0|\cdot A_{oi}(\wt{G}_{ii}^{(o)}-\delta_{\wt Q}-Y_\ell(\wt Q_\ell))(\wt {Q'}-Y_\ell(\wt Q_\ell))^{p-1}(\overline{\wt Q'-Y_\ell(\wt Q_\ell)})^{p} {\bf1}(\GG,\wt \GG\in \Omega,I)\right].
\end{align*}
Moreover, by \Cref{lem:Qlemma}, \eqref{eq:yexpansion}, and \Cref{lem:Qswitch},   we have
\[
|{Q'}-\wt {Q'}|,\, |Y_\ell(Q_\ell)-Y_\ell(\wt{Q}_\ell)|,\, |\delta_Q-\delta_{\wt Q}|\lesssim \frac{(d-1)^{3\ell}\left(\Im[m_d]+\vareps'+\vareps/\sqrt{\kappa+\eta+\vareps}\right)}{N\eta}.
\]
Therefore,
\begin{align*}
&\phantom{{}={}}\E(N^2/|\vec{E}_0| \cdot A_{oi}(\wt G_{ii}^{(o)}-\delta_{\wt Q}-Y_\ell(\wt{Q}_\ell))(\wt {Q'}-Y_\ell(\wt{Q}_\ell))^{p-1}(\overline{\wt {Q'}-Y_\ell(\wt{Q}_\ell)})^{p} {\bf1}(\GG,\wt \GG\in \Omega,I))\\
&=\E(N^2/|\vec{E}_0|\cdot A_{oi}(\wt G_{ii}^{(o)}-\delta_Q-Y_\ell(Q_\ell))( {Q'}-Y_\ell(Q_\ell))^{p-1}(\overline{{Q'}-Y_\ell(Q_\ell)})^{p} {\bf1}(\GG,\wt \GG\in \Omega,I))\\
&\qquad  +O\left(\vareps \cdot \frac{(d-1)^{3\ell}(\Im[m_d]+\vareps'+\vareps/\sqrt{\kappa+\eta+\vareps})}{N\eta}\E\left[|{Q'}-Y_\ell(Q_\ell)|^{2p-2}\right]\right)\\
&\qquad\qquad +O\left(\frac{(d-1)^{3\ell}(\Im[m_d]+\vareps'+\vareps/\sqrt{\kappa+\eta+\vareps})}{N\eta}\E\left[|{Q'}-Y_\ell(Q_\ell)|^{2p-1}\right]\right).
\end{align*}
If the first term  inside the $O(\cdot)$ were dominant, a H\"older inequality would imply that
\[
\E[|Q'-Y_\ell(Q_\ell)|^{2p}]\lesssim \vareps^{p}\left(\frac{(d-1)^{3\ell}(\Im[m_d]+\vareps'+\vareps/\sqrt{\kappa+\eta+\vareps})}{N\eta}\right)^{p}
\]
as desired.
Similarly, the second term inside the $O(\cdot)$ were dominant, a H\"older inequality would imply that
\[
\E[|Q'-Y_\ell(Q_\ell)|^{2p}]\lesssim \left(\frac{(d-1)^{3\ell}(\Im[m_d]+\vareps'+\vareps/\sqrt{\kappa+\eta+\vareps})}{N\eta}\right)^{2p}\lesssim\vareps^{p} \left(\frac{(d-1)^{3\ell}(\Im[m_d]+\vareps'+\vareps/\sqrt{\kappa+\eta+\vareps})}{N\eta}\right)^{p}
\]
as $\vareps>\frac{1}{(N\eta)^{2/3}}\gg\frac{1}{N\eta}$.

Therefore, we are left to bound
\[
\E[N^2/|\vec{E}_0|\cdot A_{oi}(\wt G_{ii}^{(o)}-\delta_Q-Y_\ell(Q_\ell))( {Q'}-Y_\ell(Q_\ell))^{p-1}(\overline{{Q'}-Y_\ell(Q_\ell)})^{p} {\bf1}(\GG,\wt \GG\in \Omega,I)].
\]
We set $P\deq G(\Ext(\cB_\ell(o,\GG^{(i)}),Q),z)$, and $P_{ol}:=P_{ol_\alpha}$ for any $\alpha$. By the expansion \eqref{eq:gminusy-1} and  \eqref{eq:gminusy-2} as in \Cref{prop:improvedQYQ}, we have 
\begin{align*}
\phantom{{}={}}\wt{ G}_{oo}^{(i)}-Y_\ell(Q_\ell)-\delta_Q&=\pi_{oi}(Q-m_{sc})+\frac{P_{ol}^2}{d-1}\sum_{\alpha} (\wt {G}_{c_\alpha c_\alpha}^{(\bT)}-Q)+\frac{P_{ol}^2}{d-1}\sum_{\alpha\neq \beta}\wt{G}_{c_\alpha c_\beta}^{(\bT)}-\delta_Q\\
&\qquad +(\wt GB(\wt G^{(\bT)}-P^{(\bT)})B^*PB(\wt G^{(\bT)}-P^{(\bT)})B^*P)_{oo}+O(N^{-1+\fc}).
\end{align*}

By the assumption that we are in $\Omega$, the last term is $O((d-1)^{2\ell}(\vareps')^2)$. Moreover, taking expectation over all vertices and using \eqref{eq:McDiarmid}, we have
\begin{align*}
&\phantom{{}={}}\left|\E[\pi_{oi}(Q-m_{sc})(Q-Y_\ell(Q_\ell))^{p-1}(\overline{Q-Y_\ell(Q_\ell)})^{p}]\right|\\
&\leq \frac{\frac{\vareps}{\sqrt{\kappa+\eta+\vareps}}}{N^{1/3}}\E\left[\left|Q-Y_\ell(Q_\ell)\right|^{2p-1}\right]\leq \vareps^2\E\left[\left|Q-Y_\ell(Q_\ell)\right|^{2p-1}\right].
\end{align*}
Thus by a H\"older inequality again, if this were the dominant term, then 
\[
\E\left[\left|Q-Y_\ell(Q_\ell)\right|^{2p}\right]\lesssim (\vareps^{2})^{2p}
\]
as desired. 

Thus we are left to show that 
\be\label{eq:lasteq}
\E\left[\left(\frac{P_{ol}^2}{d-1}\sum_{\alpha} (\wt {G}_{c_\alpha c_\alpha}^{(\bT)}-Q)+\frac{P_{ol}^2}{d-1}\sum_{\alpha\neq \beta}\wt{G}_{c_\alpha c_\beta}^{(\bT)}-\delta_Q\right)(Q-Y_\ell(Q_\ell))^{p-1}(\overline{Q-Y_\ell(Q_\ell)})^{p}\vone(\GG,\wt\GG\in \Omega,I)\right]
\ee
is approximately $0$.

Recall from \Cref{sec:switching} that $\mathbb W=\{b_\alpha:\alpha\in[\mu]\}$. In particular we have $\GG\setminus \mathbb T\mathbb W=\widetilde \GG\setminus \mathbb T\mathbb W$.
We have through \eqref{e:Schur1} for $\alpha,\beta\in [\mu]$ (with the possibility that $\alpha=\beta$) that
\begin{align}
\begin{split}
&\phantom{{}={}}\wt{ G}^{(\bT)}_{c_\alpha c_\beta}\\
&=G^{(\bT\bW)}_{c_\alpha c_\beta}+(G^{(\bT \bW)}\wt B^*(\wt{G}^{(\bT)}|_{\bW})\wt BG^{(\bT \bW)})_{c_\alpha c_\beta}\\
&=G^{(b_\alpha b_\beta)}_{c_\alpha c_\beta}-(G^{(b_\alpha b_\beta)}(G^{(b_\alpha b_\beta)}|_{\bT \bW\backslash b_\alpha b_\beta})^{-1}G^{(b_\alpha b_\beta)})_{c_\alpha c_\beta}+(G^{(\bT \bW)}\wt B^*(\wt{G}^{(\bT)}|_{\bW})\wt BG^{(\bT \bW)})_{c_\alpha c_\beta}.
\end{split}
\end{align}
By the assumption of $I$ and a similar argument to \eqref{eq:Tbalphaerrorexp}, we have
\[
\left|-(G^{(b_\alpha b_\beta)}(G^{(b_\alpha b_\beta)}|_{\bT \bW\backslash b_\alpha b_\beta})^{-1}G^{(b_\alpha b_\beta)})_{c_\alpha c_\beta}+(G^{(\bT \bW)}\wt B^*(\wt{G}^{(\bT)}|_{\bW})\wt BG^{(\bT \bW)})_{c_\alpha c_\beta}\right|=O((d-1)^{\ell} (\vareps')^2).
\]

If $\alpha=\beta$, then, letting $\E_{\bfS}$ denote the expectation over the switch, by definition of $Q$, 
\be\label{eq:lasteq2}
\E_{\bfS}[G_{c_\alpha c_\alpha}^{(b_\alpha)}-Q]=O(N^{-1+\fc}),
\ee
which once again gives a negligible contribution. 

Thus, we now assume $\alpha\neq \beta$. We further expand, defining $P'=G(\Ext(\cB_r(\{b_\alpha,b_\beta\},\GG),Q),z)$. Now, whenever it is possible to replace $G$ with $P'$ while only introducing error $O(\vareps^2)$ we do so using \Cref{thm:entrywisebound}. We also note that $P'$ is block diagonal, with two blocks corresponding to the block around $b_\alpha$ and $b_\beta$, respectively. We then expand
\begin{align*}
&\phantom{{}={}}G_{c_\alpha c_\beta}^{(b_\alpha b_\beta)}\\
&=G_{c_\alpha c_\beta}-(G(G|_{\{b_\alpha ,b_\beta\}})^{-1}G)_{c_\alpha c_\beta}\\
&\overset{\eqref{e:resolv}}{=}G_{c_\alpha c_\beta}-(G(P'|_{\{b_\alpha ,b_\beta\}})^{-1}G)_{c_\alpha c_\beta}-(G(P'|_{\{b_\alpha ,b_\beta\}})^{-1}((P'|_{\{b_\alpha ,b_\beta\}})-(G|_{\{b_\alpha ,b_\beta\}}))(P'|_{\{b_\alpha ,b_\beta\}})^{-1}G)_{c_\alpha c_\beta}+O(\vareps^2)\\
&=G_{c_\alpha c_\beta}-G_{c_\alpha b_\alpha}(P'_{b_\alpha b_\alpha})^{-1}G_{b_\alpha c_\beta}-G_{c_\alpha b_\beta}(P'_{b_\beta b_\beta})^{-1}G_{b_\beta c_\beta}+G_{c_\alpha b_\alpha}(P'_{b_\alpha b_\alpha})^{-1}G_{b_\alpha b_\beta}(P'_{b_\beta b_\beta})^{-1}G_{b_\beta c_\beta}+O(\vareps^{2})\\
&=G_{c_\alpha c_\beta}-P_{c_\alpha b_\alpha}'(P'_{b_\alpha b_\alpha})^{-1}G_{b_\alpha c_\beta}-G_{c_\alpha b_\beta}(P'_{b_\beta b_\beta})^{-1}P'_{b_\beta c_\beta}+P'_{c_\alpha b_\alpha}(P'_{b_\alpha b_\alpha})^{-1}G_{b_\alpha b_\beta}(P'_{b_\beta b_\beta})^{-1}P'_{b_\beta c_\beta}+O(\vareps^{2})
\end{align*}

Taking the expectation over all vertices gives that 
\begin{align*}
\E_{\bfS}[G_{c_\alpha c_\beta}^{(b_\alpha b_\beta)}]&=\E_{\bfS}[(1-P_{c_\alpha b_\alpha}'(P'_{b_\alpha b_\alpha})^{-1})^2G_{b_\alpha b_\beta}]+O(\vareps^2)\\
&=\E_{\bfS}[(1-P_{c_\alpha b_\alpha}'(P'_{b_\alpha b_\alpha})^{-1})^2]\frac{\bD^*G\bD}{\bD^*J\bD}+O(\vareps^2).
\end{align*}
By \eqref{eq:DJbound}
 $\frac{\bD^*G\bD}{\bD^*J\bD} =O(1/(N\eta))$. Therefore, by \eqref{eq:omegaeq2}, we can replace 
 \[
 \E_{\bfS}[(1-P_{c_\alpha b_\alpha}'(P'_{b_\alpha b_\alpha})^{-1})^2]\frac{\bD^*G\bD}{\bD^*J\bD}= \left(1+\frac{m_{sc}}{\sqrt{d-1}}\right)^2\frac{\bD^*G\bD}{\bD^*J\bD}+O(\vareps^2).
 \]

We now need to bound the expected number of potential pairs $(b_\alpha,c_\alpha)$ counted here, i.e., the expected number of pairs $x,y$ that have distance $\ell+1$ to $o$ in $\cH$ and are both retained in $\cB_{\ell+1}(o,\cG_0)$.
For every such pair $x,y$ in $\cH$, suppose the length-$(\ell+1)$ path from $o$ to $x$ and the length-$(\ell+1)$ path from $o$ to $y$ overlap in $\ell+1-k$ edges. Then for every fixed $x$ and $1\leq k\leq \ell+1$, there are $(d-2)(d-1)^{k-1}$ such $y$ in $\cH$, and given that $x$ is retained in $\cB_{\ell+1}(o,\cG_0)$, the probability that one such $y$ is retained is $p^k$. Altogether, we have
\begin{align*}
\phantom{{}={}}\sum_{y\neq x} \E(1_{x}1_y)&=p^{\ell+1}\sum_{k=1}^{\ell+1}p^{k}(d-2)(d-1)^{k-1}=p^{\ell+2}(d-2)\frac{p^{\ell+1}(d-1)^{\ell+1}-1}{p(d-1)-1}.
\end{align*}

By \Cref{lem:bordenaveconcentration}, the average number of candidate pairs $(b_\alpha,c_\alpha)$ concentrates around this, with a small error term (similar to $\pi_{oi})$. Taking the coefficient $\frac{P_{ol}^2}{d-1}$ from \eqref{eq:lasteq}, and summing over $x$, the overall contribution of this part is 
\begin{align*}
&\phantom{{}={}}m_{sc}^{2\ell+2}\left(p^{\ell+2}(d-2)\frac{p^{\ell+1}(d-1)^{\ell+1}-1}{p(d-1)-1}\right)\left(1+\frac{m_{sc}}{\sqrt{d-1}}\right)^2\frac{\bD^*G\bD}{\bD^*J\bD}=\delta_Q
\end{align*}
as desired.
\end{proof}
We make the following observation, which we do not prove. 
\begin{rmk}Let $\bd$ be the normalized vector $\bd:=\bD/||\bD||$.
    When $z\approx p\sqrt{d-1}+1/(p\sqrt{d-1})$, we can further reduce to
\begin{align*}
m_N(z)-m_{d}(z)\approx\frac1N\cdot \frac{p((d-1)^2p-1)}{(d-1)^2p^2-1}\cdot \bd^*G\bd\approx \frac1N\psi^*G\psi
\end{align*}
where $\psi$ is the vector with independent entries following the limiting Kesten-Stigum distribution. $\psi$ is the limiting Perron eigenvector.
\end{rmk}

\bibliographystyle{plain}
\bibliography{ref}

\begin{thebibliography}{10}

\bibitem{alon1986eigenvalues}
Noga Alon.
\newblock Eigenvalues and expanders.
\newblock {\em Combinatorica}, 6(2):83--96, 1986.

\bibitem{alon2021high}
Noga Alon, Shirshendu Ganguly, and Nikhil Srivastava.
\newblock High-girth near-{R}amanujan graphs with localized eigenvectors.
\newblock {\em Israel Journal of Mathematics}, 246(1):1--20, 2021.

\bibitem{alon2023limit}
Noga Alon and Fan Wei.
\newblock The limit points of the top and bottom eigenvalues of regular graphs.
\newblock {\em arXiv preprint arXiv:2304.01281}, 2023.

\bibitem{anantharaman2023friedman}
Nalini Anantharaman and Laura Monk.
\newblock Friedman-{R}amanujan functions in random hyperbolic geometry and application to spectral gaps.
\newblock {\em arXiv preprint arXiv:2304.02678}, 2023.

\bibitem{angel2015non}
Omer Angel, Joel Friedman, and Shlomo Hoory.
\newblock The non-backtracking spectrum of the universal cover of a graph.
\newblock {\em Transactions of the American Mathematical Society}, 367(6):4287--4318, 2015.

\bibitem{arras2023existence}
Adam Arras and Charles Bordenave.
\newblock Existence of absolutely continuous spectrum for {G}alton--{W}atson random trees.
\newblock {\em Communications in Mathematical Physics}, 403(1):495--527, 2023.

\bibitem{Bas92}
Hyman Bass.
\newblock The {I}hara-{S}elberg zeta function of a tree lattice.
\newblock {\em Internat. J. Math.}, 3(6):717--797, 1992.

\bibitem{bauerschmidt2019local}
Roland Bauerschmidt, Jiaoyang Huang, and Horng-Tzer Yau.
\newblock Local {K}esten--{M}ckay law for random regular graphs.
\newblock {\em Communications in Mathematical Physics}, 369:523--636, 2019.

\bibitem{biggins1993large}
JD~Biggins and NH~Bingham.
\newblock Large deviations in the supercritical branching process.
\newblock {\em Advances in Applied Probability}, 25(4):757--772, 1993.

\bibitem{bordenave2015new}
Charles Bordenave.
\newblock A new proof of {F}riedman's second eigenvalue theorem and its extension to random lifts.
\newblock {\em arXiv preprint arXiv:1502.04482}, 2015.

\bibitem{bordenave2015quantum}
Charles Bordenave.
\newblock On quantum percolation in finite regular graphs.
\newblock In {\em Annales Henri Poincar{\'e}}, volume~16, pages 2465--2497. Springer, 2015.

\bibitem{bordenave2015non}
Charles Bordenave, Marc Lelarge, and Laurent Massouli{\'e}.
\newblock Non-backtracking spectrum of random graphs: community detection and non-regular {R}amanujan graphs.
\newblock In {\em 2015 IEEE 56th Annual Symposium on Foundations of Computer Science}, pages 1347--1357. IEEE, 2015.

\bibitem{bordenave2017mean}
Charles Bordenave, Arnab Sen, and B{\'a}lint Vir{\'a}g.
\newblock Mean quantum percolation.
\newblock {\em Journal of the European Mathematical Society}, 19(12):3679--3707, 2017.

\bibitem{broder1987second}
Andrei Broder and Eli Shamir.
\newblock On the second eigenvalue of random regular graphs.
\newblock In {\em 28th Annual Symposium on Foundations of Computer Science (sfcs 1987)}, pages 286--294. IEEE, 1987.

\bibitem{chen2024new}
Chi-Fang Chen, Jorge Garza-Vargas, Joel~A Tropp, and Ramon van Handel.
\newblock A new approach to strong convergence.
\newblock {\em arXiv preprint arXiv:2405.16026}, 2024.

\bibitem{Cio06}
Sebastian~M. Cioab\v{a}.
\newblock On the extreme eigenvalues of regular graphs.
\newblock {\em J. Combin. Theory Ser. B}, 96(3):367--373, 2006.

\bibitem{colin2013scattering}
Yves Colin~de Verdi{\`e}re and Fran{\c{c}}oise Truc.
\newblock Scattering theory for graphs isomorphic to a regular tree at infinity.
\newblock {\em Journal of Mathematical Physics}, 54(6), 2013.

\bibitem{combes1973asymptotic}
Jean-Michel Combes and L~Thomas.
\newblock Asymptotic behaviour of eigenfunctions for multiparticle {S}chr{\"o}dinger operators.
\newblock {\em Communications in Mathematical Physics}, 34:251--270, 1973.

\bibitem{erdHos2017dynamical}
L{\'a}szl{\'o} Erd{\H{o}}s and Horng-Tzer Yau.
\newblock {\em A dynamical approach to random matrix theory}, volume~28.
\newblock American Mathematical Soc., 2017.

\bibitem{FL96}
Keqin Feng and Wen-Ch'ing~Winnie Li.
\newblock Spectra of hypergraphs and applications.
\newblock {\em J. Number Theory}, 60(1):1--22, 1996.

\bibitem{friedman2003proof}
Joel Friedman.
\newblock A proof of {A}lon's second eigenvalue conjecture.
\newblock In {\em Proceedings of the thirty-fifth annual ACM symposium on Theory of computing}, pages 720--724, 2003.

\bibitem{ganguly2021non}
Shirshendu Ganguly and Nikhil Srivastava.
\newblock On non-localization of eigenvectors of high girth graphs.
\newblock {\em International Mathematics Research Notices}, 2021(8):5766--5790, 2021.

\bibitem{garmo1999asymptotic}
Hans Garmo.
\newblock The asymptotic distribution of long cycles in random regular graphs.
\newblock {\em Random Structures \& Algorithms}, 15(1):43--92, 1999.

\bibitem{hide2023near}
Will Hide and Michael Magee.
\newblock Near optimal spectral gaps for hyperbolic surfaces.
\newblock {\em Annals of Mathematics}, 198(2):791--824, 2023.

\bibitem{huang2024optimal}
Jiaoyang Huang, Theo McKenzie, and Horng-Tzer Yau.
\newblock Optimal eigenvalue rigidity of random regular graphs.
\newblock {\em arXiv preprint arXiv:2405.12161}, 2024.

\bibitem{huang2024spectrum}
Jiaoyang Huang and Horng-Tzer Yau.
\newblock Spectrum of random d-regular graphs up to the edge.
\newblock {\em Communications on Pure and Applied Mathematics}, 77(3):1635--1723, 2024.

\bibitem{jiang2021equiangular}
Zilin Jiang, Jonathan Tidor, Yuan Yao, Shengtong Zhang, and Yufei Zhao.
\newblock Equiangular lines with a fixed angle.
\newblock {\em Annals of Mathematics}, 194(3):729--743, 2021.

\bibitem{kahale1995eigenvalues}
Nabil Kahale.
\newblock Eigenvalues and expansion of regular graphs.
\newblock {\em Journal of the ACM (JACM)}, 42(5):1091--1106, 1995.

\bibitem{kesten1966limit}
Harry Kesten and Bernt~P Stigum.
\newblock A limit theorem for multidimensional {G}alton-{W}atson processes.
\newblock {\em The Annals of Mathematical Statistics}, 37(5):1211--1223, 1966.

\bibitem{kollar2021gap}
Alicia Koll{\'a}r and Peter Sarnak.
\newblock Gap sets for the spectra of cubic graphs.
\newblock {\em Communications of the American Mathematical Society}, 1(1):1--38, 2021.

\bibitem{letrouit2024maximal}
Cyril Letrouit and Simon Machado.
\newblock Maximal multiplicity of {L}aplacian eigenvalues in negatively curved surfaces.
\newblock {\em Geometric and Functional Analysis}, 34(5):1609--1645, 2024.

\bibitem{LP16}
Eyal Lubetzky and Yuval Peres.
\newblock Cutoff on all {R}amanujan graphs.
\newblock {\em Geom. Funct. Anal.}, 26(4):1190--1216, 2016.

\bibitem{lyons1995conceptual}
Russell Lyons, Robin Pemantle, and Yuval Peres.
\newblock Conceptual proofs of l log l criteria for mean behavior of branching processes.
\newblock {\em The Annals of Probability}, pages 1125--1138, 1995.

\bibitem{magee2024limit}
Michael Magee.
\newblock The limit points of the bass notes of arithmetic hyperbolic surfaces.
\newblock {\em arXiv preprint arXiv:2403.00928}, 2024.

\bibitem{mckay2004short}
Brendan~D McKay, Nicholas~C Wormald, and Beata Wysocka.
\newblock Short cycles in random regular graphs.
\newblock {\em the electronic journal of combinatorics}, pages R66--R66, 2004.

\bibitem{mckenzie2021high}
Theo McKenzie and Sidhanth Mohanty.
\newblock High-girth near-{R}amanujan graphs with lossy vertex expansion.
\newblock In {\em 48th International Colloquium on Automata, Languages, and Programming (ICALP 2021)}. Schloss-Dagstuhl-Leibniz Zentrum f{\"u}r Informatik, 2021.

\bibitem{mckenzie2021support}
Theo McKenzie, Peter Michael~Reichstein Rasmussen, and Nikhil Srivastava.
\newblock Support of closed walks and second eigenvalue multiplicity of graphs.
\newblock In {\em Proceedings of the 53rd Annual ACM SIGACT Symposium on Theory of Computing}, pages 396--407, 2021.

\bibitem{Nil91}
A.~Nilli.
\newblock On the second eigenvalue of a graph.
\newblock {\em Discrete Math.}, 91(2):207--210, 1991.

\bibitem{Nil04}
A.~Nilli.
\newblock Tight estimates for eigenvalues of regular graphs.
\newblock {\em Electron. J. Combin.}, 11(1):Note 9, 4, 2004.

\bibitem{wormald1999models}
Nicholas~C Wormald et~al.
\newblock Models of random regular graphs.
\newblock {\em London mathematical society lecture note series}, pages 239--298, 1999.

\end{thebibliography}

\appendix
\section{Green's function identities}
\label{app:Green}

Throughout this paper, we repeatedly use some (well-known) identities for Green's functions,
which we collect in this appendix.

\subsection{Resolvent identity}

The following well-known identity is referred as resolvent identity:
for two invertible matrices $A$ and $B$ of the same size, we have
\begin{equation} \label{e:resolv}
  A^{-1} - B^{-1} = A^{-1}(B-A)B^{-1}=B^{-1}(B-A)A^{-1}.
\end{equation}

\subsection{Schur complement formula}

Given an $N\times N$ matrix $M$ and index set $\mathbb{V} \subseteq \qq{N}$, up to rearrangement of indices, $M$ can be written in the block form
\[
M=\begin{bmatrix}
    A & B^*\\
    B & D
\end{bmatrix}.
\]
The Schur complement formula asserts that the inverse of $M$ satisfies
\[
M^{-1}=\begin{bmatrix}
    (A-B^*D^{-1}B)^{-1} & -(A-B^*D^{-1}B)^{-1}B^*D^{-1}\\
    -D^{-1}B(A-B^*D^{-1}B)^{-1} & D^{-1}+D^{-1}B(A-B^*D^{-1}B)^{-1}B^*D^{-1}
\end{bmatrix}.
\]
Taking $M=H-z$ where $H$ is the normalized adjacency matrix of $\GG$, we know that $M^{-1}=G$ is the Green's function of $\GG$, and $D^{-1}=G^{(\mathbb{V})}$ is the Green's function of $\GG^{(\mathbb{V})}$.

Let $G|_{\mathbb{V}},G|_{\mathbb{V}\mathbb{V}^\complement},G|_{\mathbb{V}^\complement\mathbb{V}},G|_{\mathbb{V}^\complement}$ denote the four blocks of $G$, so that
\[
G=\begin{bmatrix}
    G|_{\mathbb{V}} & G|_{\mathbb{V}\mathbb{V}^\complement}\\
    G|_{\mathbb{V}^\complement\mathbb{V}} & G|_{\mathbb{V}^\complement}
\end{bmatrix}.
\]
Then by the Schur complement formula, with $B=G|_{\mathbb V^\complement \mathbb V}$, we have the identity
\begin{align} \label{e:Schur2}
  G|_{\mathbb V} &= (H-z-B^* G^{(\bV)} B)^{-1}
\end{align}
and
\begin{align} \label{e:Schur3}
  G|_{\mathbb{V}\mathbb{V}^\complement}&=-G|_{\mathbb{V}}B^*G^{(\mathbb{V})}.
\end{align}
 Moreover, we have 
\begin{align}\label{e:Schur1}
G|_{\mathbb{V}^\complement}&=D^{-1}+D^{-1}B(A-B^*D^{-1}B)^{-1}B^*D^{-1}\notag\\
&=G^{(\mathbb{V})}+G^{(\mathbb{V})}B(G|_{\mathbb{V}})B^*G^{(\mathbb{V})}\notag\\
&=D^{-1}+D^{-1}B(A-B^*D^{-1}B)^{-1}(A-B^*D^{-1}B)(A-B^*D^{-1}B)^{-1}B^*D^{-1} \\
&=G^{(\mathbb{V})}+(G|_{\mathbb{V}^\complement\mathbb{V}})(G|_{\mathbb{V}})^{-1}(G|_{\mathbb{V}\mathbb{V}^\complement}).\notag
\end{align}
Taking $\mathbb V$ to be a single vertex set $\{k\}$, we get the special case that
\begin{equation} \label{e:Schurixj}
G_{ij}^{(k)} = G_{ij}-G_{ik}G_{kk}^{-1}G_{kj}.
\end{equation}

\section{Proof of \Cref{lem:exchangeablepair}}\label{sec:exchangeproof}
Define $\GNd$ to be the set of possible graphs given the degree distribution of $V_0$. For $\cH\in \GNd$  denote by $\iota(\cG) := \{\cG\} \times \sf S(\cG)$ the fiber
of local resamplings of $\cG$ (with respect to vertex $o$),
and define the enlarged probability space
\begin{align*}
\GNdp = \iota(\GNd) = \bigcup_{\cG\in \GNd}\iota(\cG),
\end{align*}
with the probability measure $\wt{\P}(\cG, {\bf S}):= \P(\cG)\P_{\cG}({\bf S}) = (1/|\GNd|)(1/|\sS(\cG)|)$
for any $(\cG, {\bf S})\in  \GNdp$.
Here $\P_{\cG}$ is the uniform probability measure on $\sS(\cG)$.\index{$\tilde \bP$}\index{$\bP_\cG$}

Let $\pi: \GNdp \to \GNd$, $(\cG,{\bf S}) \overset{\pi}{\mapsto} \cG$ be the canonical projection onto the first component.
Note $\pi$ is measure preserving: $\P = \wt{\P} \circ \pi^{-1}$.

On the enlarged probability space, we define the maps 
\begin{align}
\label{e:Ttildedef}
\tilde T &: \GNdp \to \GNdp, &\quad
\tilde T(\cal G, {\bf S}) &:= (T_{\bf S}(\cG), T({\bf S})),
\\
\label{e:Tdef}
T &: \GNdp \to \GNd, &\quad
T(\cal G, {\bf S}) &:= \pi(\tilde T(\cG,{\bf S})) = T_{\bf S}(\cal G).
\end{align}
For any finite graph $\cT$ on a subset of $\qq{N}$,
we define $\GNd(\cT):=\{\cG\in \GNd: \cB_\ell(o,\cG)=\cT\}$
to be the set of $d$-regular graphs whose radius $\ell$ neighborhood of the vertex $o$ in $\cG$ is $\cT$. We now claim that for any graph $\cT$, 
\begin{equation}\label{e:tildeTT}
  \tilde T(\iota(\GNd(\cT))) = \iota(\GNd(\cT)),
\end{equation}

Since our local resampling does not change the radius $\ell$ neighborhood of $o$, from our construction, $\tilde T(\iota(\GNd(\cT))) \subseteq \iota(\GNd(\cT))$. Next we show that $\tilde T$ is an involution, with $ \tilde T(\iota(\GNd(\cT))) = \iota(\GNd(\cT))$.

To verify that $\tilde T$ is an involution,
let $(\cG, {\bf S}) \in \GNdp$ and abbreviate $(\tcG, \tilde {\bf S}) := \tilde T(\cG, {\bf S})$.
Then by \eqref{e:tildeTT}, the edge boundaries of the radius $\ell$ neighborhoods of $o$ have the same number of edges $\mu$
in $\tcG$ and $\cG$.
Moreover, we can choose the  enumeration of the boundary of the $\ell$-ball in $\tcG$ such that,
for any $\alpha \in \qq{1,\mu}$, we have $T_\alpha(\vec S_\alpha) \in \sS_\alpha(\tcG)$.
Define
\begin{align*}
\tilde {\sf W}_{\tilde {\bf S}}:=\{\alpha\in \qq{1,\mu}: I_\alpha(\tcG,\tilde{\bf S})\}.
\end{align*}
We claim that $\tilde {\sf W}_{\tilde {\bf S}} = {\sf W}_{\bf S}$.

First, by the definition of switchings, we have that $\dist_{\wt\GG}(\{a_\alpha,b_\alpha,c_\alpha\},\{a_\beta,b_\beta,c_\beta\})\leq\fR/4$ if and only if $\dist_{\GG}(\{a_\alpha,b_\alpha,c_\alpha\},\{a_\beta,b_\beta,c_\beta\})\leq\fR/4$.
Thus, assume that  $\{a_\beta,b_\beta,c_\beta\}\not\in \cB_{\fR/4}(\{a_\alpha,b_\alpha,c_\alpha\},\GG^{(\bT)})$ for all $\beta\neq \alpha$. If $(\{a_\alpha,b_\alpha,c_\alpha\},\cG^{(\bT)})\cup \{b_\alpha,c_\alpha\}$ is not a tree, then we do not perform the switch around $\alpha$, so $(\{a_\alpha,b_\alpha,c_\alpha\},\wt \cG^{(\bT)})\cup \{b_\alpha,c_\alpha\}$ is not a tree either. If alternatively $(\{a_\alpha,b_\alpha,c_\alpha\},\cG^{(\bT)})\cup \{b_\alpha,c_\alpha\}$ is a tree, then after performing the switch, by the assumptions of the switch and the tree neighborhood, we obtain $(\{a_\alpha,b_\alpha,c_\alpha\},\wt \cG^{(\bT)})\cup \{b_\alpha,c_\alpha\}$ is also a tree. This implies $I_\alpha(\tcG,{\tilde{\bf S}})=1$. 
 Therefore we have that the $\fR/4$-neighborhoods of $\{a_\alpha, b_\alpha, c_\alpha\}$ never change, i.e., $\cB_{\fR/4}(\{a_\alpha,b_\alpha,c_\alpha\}, \cG^{(\bT)})=\cB_{\fR/4}(\{a_\alpha,b_\alpha,c_\alpha\}, \tcG^{(\bT)})$. Thus $\tilde {\sf W}_{\tilde {\bf S}}={\sf W}_{\bf S}$. By the definition of our switching,
it follows that $T(\tilde {\bf S})=\bf S$ and $T_{\tilde{\bf S}}(\tcG)=\cG$.
Therefore $\tilde T$ is an involution.

This means that edge swaps form a reversible Markov chain which proves \Cref{lem:exchangeablepair}.

\end{document}